\documentclass[11pt]{article}\textwidth 160mm\textheight 235mm
\usepackage{amsmath,epsfig, subfigure,amsthm,
amsfonts,amsbsy,amssymb,latexsym,amsxtra}
\usepackage{graphicx,wrapfig}
\usepackage{enumitem}
\usepackage{tikz}
\usepackage{mathrsfs}
\usepackage{adjustbox}
\usepackage{cancel}
\usepackage{todonotes}
\usepackage{nicematrix}
\usetikzlibrary{matrix}
\usetikzlibrary{arrows,shapes}
\usepackage{cite}
\usepackage{hyperref}
\usepackage[title]{appendix}
\hypersetup{colorlinks=true, urlcolor=blue}
\expandafter\let\expandafter\oldproof\csname\string\proof\endcsname
\let\oldendproof\endproof
\renewenvironment{proof}[1][\proofname]{%
  \oldproof[\ttfamily \scshape \bf #1. ]%
}{\oldendproof}
\def\O{{\cal O}}
\def\S{{\bf {S}}}

\def\B{\mathbb{B}}
\def\R{{\bf R}}
\def\N{{\rm I\!N}}

\def\ox{\bar{x}}
\def\oy{\bar{y}}

\def\ov{\bar{v}}

\def\ss{\scriptsize }
\def\diag{\mbox{\rm diag}\,}
\def\Hat{\widehat}

\def\ve{\varepsilon}

\def\X{{\bf X}}
\def\Y{{\bf Y}}
\def\Z{{\bf Z}}

\def\tilde{\widetilde}
\def\emp{\emptyset}

\def\span{{\rm span}\,}

\def\Lm{{\Lambda}}
\def\tto{\rightrightarrows}

\def\d{{\rm d}}
\def\sub{\partial}

\def\Hat{\widehat}
\def\Tilde{\widetilde}
\def\Bar{\overline}
 \def\para{{\rm par}\,}
\def\ra{\rangle}
\def\la{\langle}
\def\ve{\varepsilon}

\def\conv{\mbox{\rm conv}\,}
\def\cone{\mbox{\rm cone}\,}
\def\span{\mbox{\rm span}\,}
\def\quadr{ {\rm quad}\,}
\def\ri{\mbox{\rm ri}\,}

\def\gph{\mbox{\rm gph}\,}
\def\epi{\mbox{\rm epi}\,}

\def\dim{\mbox{\rm dim}\,}
\def\dom{\mbox{\rm dom}\,}

\def\ker{\mbox{\rm ker}\,}

\def\diag{\mbox{\rm diag}\,}

\def\dn{\downarrow}
\def\O{\Omega}

\def\ph{\varphi}

\def\emp{\emptyset}

\def\oR{\Bar{\R}}
\def\lm{\lambda}

\def\gg{\gamma}
\def\dd{\delta}
\def\al{\alpha}

\def\th{\theta}

\def\sm{\hbox{${1\over 2}$}}

\def \b{{\}_{k\in\N}}}

\def\sce{\setcounter{equation}{0}}
\DeclareMathOperator*{\limi}{e-liminf\;}
\begin{document}
\newtheorem{Theorem}{Theorem}[section]
\newtheorem{Proposition}[Theorem]{Proposition}
\newtheorem{Remark}[Theorem]{Remark}
\newtheorem{example}[Theorem]{Example}
\newtheorem{Lemma}[Theorem]{Lemma}
\newtheorem{Corollary}[Theorem]{Corollary}
\newtheorem{Definition}[Theorem]{Definition}
\newtheorem{Example}[Theorem]{Example}
\newtheorem{Algorithm}[Theorem]{Algorithm}
\numberwithin{equation}{section}
\renewcommand{\thefootnote}{\fnsymbol{footnote}}
\begin{center}
{\bf \Large Characterizations of Tilt-Stable Local Minimizers of a Class of Matrix Optimization Problems}
\\[1 ex]
CHAO DING\footnote{State Key Laboratory of Mathematical Sciences, Academy of Mathematics and Systems Science, Chinese Academy of Sciences, Beijing 100190, China; School of Mathematical Sciences, University of Chinese Academy of Sciences, Beijing 100049, China; Institute of Applied Mathematics, Academy of Mathematics and Systems Science, Chinese Academy of Sciences, Beijing 100190, China (dingchao@amss.ac.cn). Research of this author is supported in part by the National Key R\&D Program of China (No.~2021YFA1000300, No.~2021YFA1000301), National Natural Science Foundation of China (No.~12531014), and CAS Project for Young Scientists in Basic Research (No.~YSBR-034). }\quad
EBRAHIM SARABI\footnote{Department of Mathematics, Miami University, Oxford, OH 45065, USA (sarabim@miamioh.edu). Research of this    author is partially supported by the U.S. National Science Foundation  under the grant DMS 2108546.}\quad
SHIWEI WANG\footnote{Institute of Operational Research and Analytics, National University of Singapore, Singapore (wangshiwei@amss.ac.cn).}
\end{center}
\vspace*{0.05in}
\small{\bf Abstract.} 
Tilt stability plays a pivotal role in understanding how local solutions of an optimization problem respond to small, targeted perturbations of the objective. Although quadratic bundles are powerful tools for capturing second-order variational behavior, their characterization remains incomplete beyond well-known polyhedral and certain specialized nonpolyhedral settings. To help bridge this gap, we propose a new point-based criterion for tilt stability in prox-regular, subdifferentially continuous functions by exploiting the notion of minimal quadratic bundles. Furthermore, we derive an explicit formula for the minimal quadratic bundle associated with a broad class of general spectral functions, thus providing a practical and unifying framework that significantly extends existing results and offers broader applicability in matrix optimization problems. 
 \\[1ex]
{\bf Key words.} tilt stability, quadratic bundle, polyhedral spectral function.\\[1ex] 
{\bf  Mathematics Subject Classification (2010)} 90C31, 49J52, 49J53

\normalsize
\section{Introduction}\sce \label{intro}

The problems of the form
\begin{equation*}
\mbox{minimize}\;\;  f(x)\quad \mbox{subject to}\;\; x\in \X,
\end{equation*}
where $\X$ is an Euclidean space and $f:\X\to\oR :=[-\infty,+\infty]$ is a possibly extended-real-valued function, frequently appear in optimization and operations research. The stability of local solutions under perturbations lies at the heart of optimization theory, shaping both algorithmic robustness \cite{CC,ZKZhu,br,LPT} and our understanding of variational landscapes \cite{pr98,LPR2000}. Among stability concepts, one natural route for studying local stability properties of a local minimizer $\bar{x}$ involves examining the behavior of the solution to the ``tilted" problem
\begin{equation}\label{eq:tilted-prob}
\mbox{minimize}\;\; \quad f(x)-\langle x,v\rangle\quad \mbox{subject to}\;\; x\in \X,
\end{equation}
in which one adds a linear term $-\langle x, v \rangle$ to $f(x)$. Such linear (or ``tilt") perturbations form a fundamental class since they capture first-order approximations of more general perturbations and therefore provide a unifying lens for sensitivity considerations. 

A local minimizer $\bar{x}$ is said to be \emph{tilt-stable} if, for all sufficiently small $v$, there is a unique local solution $x(v)$ to the tilted problem \eqref{eq:tilted-prob} that depends in a Lipschitz way on $v$ (see \eqref{til} in Section \ref{sec:Preliminaries} for the definition). This notion of \emph{tilt stability}, introduced and formalized by Poliquin and Rockafellar \cite{pr98}, highlights local minima at which small perturbations to the objective function do not harm uniqueness or cause large deviations in the solution. Such a stability property is a cornerstone for sensitivity analysis, numerical methods, and applications ranging from machine learning to control systems\cite{Higham,MSarabi21,Rowan}.

In the classical smooth setting, for a twice continuously differentiable function $f$, one has the well-known characterization: if $\nabla f(x^*)=0$, then $x^*$ is tilt-stable if and only if the Hessian $\nabla^2f(x^*)$ is positive definite \cite[Proposition~1.2]{pr98}. When the function is not everywhere differentiable but remains prox-regular and subdifferentially continuous, the tilt stability of a local minimum is correctly captured by the positive definiteness of a generalized Hessian object, sometimes described by coderivatives of the limiting subdifferential \cite{pr98}. 

Beyond such optimality-based explanations, tilt-stable solutions are also interesting from a computational perspective because they fix the unrobustness of strong local minimizers \cite[Definition 1.1]{DLewis} to small perturbations to the objective function. In 2013, for subdifferentially continuous, lower semicontinuous $f$, \cite{DLewis} established its equivalence to the uniform quadratic growth, strongly metrically regular of the subdifferential mapping $\partial f$. In the same year, \cite{LZhang13} established its equivalence to the strong criticality of the local minimizer, which is a locally quadratic minimizer, and the subdifferential contains zero in its relative interior, for prox-regular and ${\cal C}^2$-partly smooth function relative to the ${\cal C}^2$-smooth manifold. In \cite{mn14}, its equivalence to the uniform second-order growth condition is further extended to the prox-regular and subdifferentially continuous function. In 2018, \cite{chn} explored this problem via the positive definiteness of the graphical derivative of its limiting subdifferential. 
If we further suppose the validity of constraint nondegeneracy, \cite{MNR15} proved the equivalence between tilt stability and the well-known strong second order sufficient condition (SSOSC) for nonlinear semidefinite programming (NLSDP). 
Recently, in \cite{rockvtn}, for the general optimization problem, the variationally strong convexity implies the tilt stability \cite[Definition 3]{rockvtn} of the local minimizer. However, it is worth noting that tilt stability usually does not imply variationally strong convexity, as explained in \cite[Remark 2.8]{kmp} and \cite{rockvtn}. 
In the last year, the study of tilt stability has raised more attention, specially for (locally) convex problems \cite{Nghia}.

The main objective of this paper is to demonstrate that the second subderivative and its epigraphical limits, called the quadratic bundles, suffice to characterize tilt-stable local minimizers of a function. 
The main advantage of the second subderivative over the other second-order constructions, including the graphical derivative and coderivative in \cite{chn, mn14,pr98} which were used extensively for the characterization of tilt-stable minimizers, is that it can be calculated for major classes of functions, important for their applications into constrained and composite optimization problems,  such as spectral functions for which other second-order constructions seem difficult to calculate and are not available yet.   

We also aim to calculate the minimal quadratic bundle for an important class of spectral functions. Quadratic bundles, introduced recently in \cite{r23},  offer a powerful approach for capturing second-order variational behavior in optimization. However, beyond the well-understood settings of polyhedral sets and certain specialized nonpolyhedral sets (such as second-order cones \cite{r23,WDZZhao} and the semidefinite cones \cite{WDZZhao}), a complete mathematical description of minimal quadratic bundles is not yet available.  We also  investigate how to employ the quadratic bundle to obtain the explicit characterizations of tilt-stable local minimizers for structured matrix optimization problems, including those featuring spectral functions. Spectral-related matrix optimization problems arise in a variety of applications such as the maximum Markov chain mixing rates \cite{BDParrilo, BDSXiao,BDXiao}, matrix approximations under doubly stochastic constraints \cite{ZS}, unsupervised learning \cite{Ghahramani}, or semidefinite programming \cite{Sun06,TT}. Such problems also include various low-rank regularization models (e.g., robust PCA \cite{CLM}, matrix completion \cite{CP}) and indeed often feature non-polyhedral feasible sets or objectives. Our main contributions can be summarized as follows:

\begin{itemize}
\item We develop a new neighborhood  criterion (Theorem~\ref{ssti}) for tilt stability in the setting of prox-regular, subdifferentially continuous functions using the notion of the second subderivative.
Pervious attempt in \cite[Corollary~2.5]{Nghia} gave a neighborhood characterization 
of tilt-stable minimizers of convex functions without determining the relationship between the modulus of tilt stability. Our result gives an extension of this observation for prox-regular functions and determine the modulus of tilt stability as well.

\item Using our neighborhood characterization, we obtain a pointwise characterization of tilt-stable local minimizers of prox-regular, subdifferentially continuous functions via the concept of the quadratic bundles. Similar results were established recently in \cite[Theorem~5.2]{kmpv} using a different approach. Our results go one step further than \cite{kmpv} and provide a characterization of tilt-stable minimizers via strict positiveness of quadratic bundles (Theorem~\ref{pos-tilt}), which is important for applications to constrained and composite optimization problems.  When a minimal quadratic bundle exists, we show that the tilt-stability is equivalent to positiveness of the minimal quadratic bundle.

\item We demonstrate that an important class of spectral functions enjoys the minimal quadratic bundle. In this case, we
provide an explicit formula for the minimal quadratic bundle, thus substantially extending the known formulas for polyhedral settings \cite{r23} and the particular semidefinite cone \cite{WDZZhao}. The result yields a practical tool for verifying tilt stability in a broader scope of matrix optimization problems. 
\end{itemize}

The paper is organized as follows. Section~\ref{sec:Preliminaries} outlines the essential notation and some known facts about variational analysis and tilt stability. Section~\ref{tilt} explores quadratic bundles and second subderivatives, leading to a fresh characterization of tilt stability. Section~\ref{sec:qbdofspfunc} then focuses on the explicit form of minimal quadratic bundles for polyhedral spectral functions and illustrates how these formulas assist in verifying tilt stability in complex matrix optimizing scenarios. Finally, Section~\ref{tilt-matrix} consolidates these results to yield new insights into tilt-stable local minimizers in a class of matrix optimization settings. Concluding remarks and directions for further study appear in the last section.

\section{Preliminaries}\label{sec:Preliminaries}
In this section, we introduce some notations and preliminary results that are frequently used throughout this paper. 
Suppose that $\X$, $\Y$, and $\Z$ are given Euclidean spaces. In the  product space $\X\times \Y$, its norm is defined as $\|(w,u)\|=\sqrt{\|w\|^2+\|u\|^2}$ for any $(w,u)\in \X\times \Y$. 
Denote $\B$ as  the closed unit ball and $\B_r(x):=x+r\B$ as the closed ball centered at $x$ with radius $r>0$. Given a nonempty set $C\subset\X$, we apply  $\ri C$, $C^*$,  $\conv C$, $\cone C$,   and $\mbox{aff}\, C$ to represent its relative interior, polar cone,  the convex hull,  the conic hull, and the affine hull of $C$, respectively. 
Let $\{C^t\}_{t>0}$ be a parameterized family of sets in $\X$. Its outer limit set (cf. \cite[Definition~4.1]{rw}) is defined as
\begin{eqnarray*}
\limsup_{t\searrow 0} C^t&=& \big\{x\in \X|\; \exists \, t_k \searrow 0 \;\exists\,   \; x^{t_k}\to x \;\;\mbox{with}\;\; x^{t_k}\in C^{t_k}\big\};
\end{eqnarray*}
{the notation $t\searrow 0$ means $t\rightarrow0$ and each $t>0$}. 
$\{C^t\}_{t>0}$ is said to converge to $C$ if the outer limit set coincides with its inner limit set (cf. \cite[Definition~4.1]{rw}) with both of them equal to $C$. We denote it as $C^t \to C$ when $t\searrow 0$. 
A sequence $\{f^k\b$ of functions $f^k:\X\to \oR $ is said to {\em epi-converge} to a function $f:\X\to \oR$ if we have $\epi f^k\to \epi f$ as $k\to \infty$, where $\epi f=\{(x,\al)\in \X\times \R|\, f(x)\le \al\}$ is  the epigraph of $f$;
 see \cite[Definition~7.1]{rw} for more details on  epi-convergence. We denote by $f^k\xrightarrow{e} f$ the  epi-convergence of  $\{f^k\b$ to $f$. 
Fix any $\ox\in C$, the tangent  cone to $C$ at $\ox$ is defined   by
\begin{equation*} 
T_C(\ox) = \limsup_{t\searrow 0} \frac{C - \ox}{t}.
\end{equation*}

Consider a set-valued mapping $F:\X\tto\Y$, its domain  and graph are defined, respectively, by  
$ \dom F:= \{x\in\X\big|\;F(x)\ne\emp \}$ and $\gph F=\{(x,y)\in \X\times \Y|\, y\in F(x)\}$. The {\em graphical derivative} \cite[Definition~8.33]{rw} of $F$ at $\ox$ for $\oy$ with $(\ox,\oy) \in \gph F$ is the set-valued mapping $DF(\ox, \oy): \X\tto \Y$ defined by 
$\gph DF(\ox, \oy) = T_{\ss \gph F}(\ox, \oy)$.  
When the `$\limsup$' in the definition of $T_{\ss \gph F}(\ox, \oy)$ becomes a full limit, 
we say that $F$ is proto-differentiable at $\ox$ for $\oy$. 
Given $\Omega\subset \X$ and $\ox\in \Omega$,  its  regular normal cone $\Hat N_{\Omega}(\ox)$  at $\ox$ is defined by
 $\Hat N_{\Omega}(\ox) = T_\Omega(\ox)^*$. For $\ox\notin \Omega$, we set $ \Hat N_\Omega(\ox) = \varnothing$. The (limiting/Mordukhovich) normal cone $N_\Omega(\ox)$ to $\Omega$ at $\ox$ is given by 
${N}_\Omega (\ox):= \limsup\limits_{x\rightarrow\ox}{N}_\Omega (x)$. 
When $\Omega$ is convex, both normal cones coincide with the normal cone in the sense of  convex analysis.  
Given a function $f:\X \to \oR$ and a point $\ox\in\X$ with $f(\ox)$ finite. Denote the subdifferential of $f$ at $\ox$ as $\partial f(\ox)=\{v\mid (v,-1)\in N_{\ss \epi f}(\ox,f(\ox))\}$.  
A function  $f\colon\X\to\oR$ is called prox-regular at $\ox$ for $\ov$ if $f$ is finite at $\ox$ and locally lower semicontinuous (lsc)   around $\ox$ with $\ov\in\sub f(\ox)$, and there exist 
constants $\ve>0$ and $r\ge 0$ such that
\begin{equation}\label{prox}
\begin{cases}
f(x')\ge f(x)+\la v,x'-x\ra-\frac{r}{2}\|x'-x\|^2\;\mbox{ for all }\; x'\in\B_{\ve}(\ox)\\
\mbox{whenever }\;(x,v)\in(\gph\sub f)\cap\B_{\ve}(\ox,\ov)\; \mbox{ with }\; f(x)<f(\ox) +\ve.
\end{cases}
\end{equation}
The function $f$ is called subdifferentially continuous at $\ox$ for $\ov$ if the convergence $(x^k,v^k)\to(\ox,\ov)$ with $v^k\in\sub f(x^k)$ yields $f(x^k)\to f(\ox)$ as $k\to\infty$. 
It is well known that a wide range of functions satisfy the above two properties, e.g., convex functions \cite[Example~13.30]{rw} and fully amenable functions in the sense of \cite[Definition~10.23]{rw}.

\begin{Remark}\label{proxr} 
{\rm Note that if the function $f$ is prox-regular and subdifferentially continuous at $\ox$ for $\ov$ with constants $\ve$ and $r$
satisfying \eqref{prox}, then $f$ enjoys prox-regularity for any 
$(x,v)\in \gph \sub f$ sufficently close to $(\ox,\ov)$ with the same constant $r$ but possibly with a smaller radius $\ve$. To see this, take $\ve$ from 
\eqref{prox} and choose  $\dd\in (0,\ve/2)$ so that $(x,v)\in (\gph \sub f)\cap B_{\dd}(\ox,\ov)$ the estimate $|f(x)-f(\ox)|<\ve/2$ holds, which results from 
the assumed subdifferential continuity of $f$. One can easily 
 gleaned from \eqref{prox} that for any $(x,v)\in (\gph \sub f)\cap B_{\dd}(\ox,\ov)$, $f$ enjoys \eqref{prox} with $\ve:= \dd/2$ and the same constant $r$. Moreover, it follows from \cite[Proposition~2.3]{pr96} that for any $(x,v)\in (\gph \sub f)\cap B_{\dd}(\ox,\ov)$, we get $f(x')\to f(x)$ as 
 $(x',v')\to (x,v)$ and $v'\in \sub f(x')$, meaning that $f$
 is subdifferentially continuous at $x$ for $v$. In summary, we showed that 
 prox-regularity and subdifferential continuity holds at $x$ for $v$
 for any $(x,v)\in \gph \sub f$ close to $(\ox,\ov)$.
}
\end{Remark}

Let $f:\X\rightarrow\overline{\R}$ and $\ox\in \X$,  and let $f(\bar{x})$ be finite. The function $f$ is said to be twice epi-differentiable at $\bar{x}$ for $\ov$ if the functions
	$$\Delta^2_{t}f(\bar{x}, \ov)(u)=\frac{f(\bar{x}+tu)-f(\bar{x})-t\langle \ov,u\rangle}{\frac{1}{2}t^2}, \quad u\in \X$$
	epi-converge to $\d^2f(\bar{x}, v)$ as $t\searrow0$, where $\d^2f(\bar{x}, \ov)$ is the second subderivative of $f$ at $\bar{x}$ for $\ov$ defined by
	\begin{equation}\label{eq:defofssubd}\d^2f(\bar{x}, \ov)(w)=\liminf_{t\searrow0,u\rightarrow w}\Delta^2_{t}f(\bar{x}, \ov)(u).\end{equation}

\begin{Remark}\label{monotone} 
{\rm Suppose that the function $f$ is prox-regular  and subdifferentially continuous at $\ox$ for $\ov$ with constants $\ve$ and $r$
satisfying \eqref{prox}. Take any such a pair $(x,v)\in \gph \sub f$  sufficiently close to $(\ox,\ov)$ so that $f$ is prox-regular and subdifferentially continuous at $x$ for $v$, which can be done according to Remark~\ref{proxr}, and assume further that $\sub f$ is proto-differentiable at $x$ for $v$. Then, it follows from \cite[Proposition~4.8]{pr96} and Remark~\ref{proxr} that $\sub f+ rI$
is monotone, with $I$ standing for identity mapping on $\X$,
which together with \cite[Theorem~12.64]{rw} tells us that $D(\sub f)(x,v)+rI$
is a monotone mapping. Appealing now to \cite[Theorem~13.40]{rw} leads us to 
\begin{equation}\label{newq}
D(\sub f)(x,v)+rI =  \sub\big(\frac{1}{2} \d^2 f(x, v)+ \frac{r}{2}\|\cdot\|^2\big),
\end{equation}
which implies that  the mapping $w\mapsto \frac{1}{2} \d^2 f(x, v)(w)+ \frac{r}{2}\|w\|^2$ is convex; see \cite[Theorem~12.17]{rw}. We want to highlight 
that   
the constant $r$ is the same for any such a point $(x,v)$, which plays an important role in the proof of Proposition~\ref{quadn}.  
}
\end{Remark}

A point $\ox$ is said to be a tilt-stable local minimizer of the function $f\colon \X\rightarrow\oR$ if $  f(\ox)$ is finite  and there is $\gg>0$ such that the mapping
\begin{equation}\label{til}
M_{\gg}\colon v\mapsto {\mbox{argmin}}\{f(x)-\la v,x\ra\;|\; x\in \B_{\gg}(\ox)\}
\end{equation}
is single-valued and Lipschitz continuous on a neighborhood of $\ov=0$ with $M_{\gg}(\ov)=\{\ox\}$. Moreover, we say that $\ox$ is a tilt-stable minimizer of $f$ 
with constant $\kappa>0$ if the mapping $M_\gg$ is Lipschitz continuous  with constant $\kappa$ on a neighborhood of $\ov=0$ with $M_{\gg}(\ov)=\{\ox\}$.

Recall also that a set-valued mapping $F$ admits a {\em single-valued graphical localization} around $(\ox,\oy)\in\gph F$ if there
exist some neighborhoods $U$ of $\ox$ and $V$ of $\oy$ together with a single-valued mapping $f\colon U\to V$ such that $\gph F\cap(U\times V)=\gph f$.
The following characterization of tilt stability was established in \cite[Theorem~3.2]{mn14}.
\begin{Proposition}[tilt stability via the second-order growth condition]\label{usogc2}
Let $f\,\colon\X\to\oR$ be prox-regular and subdifferentially continuous at
$\ox$ for $\ov=0$. Then the following conditions  are equivalent.
\begin{itemize}[noitemsep]
\item [\rm{(a)}] The point $\ox$ is a tilt-stable minimizer of $f$ with constant $\kappa>0$.
\item [\rm{(b)}] There are neighborhoods $U$ of $\ox$ and $V$ of $\ov$ such that mapping $(\sub f)^{-1}$ admits a single-valued localization 
$\vartheta\colon V\to U$ around $(\ov,\ox)$ and that for any pair $(v,u)\in \gph \vartheta=\big( \gph \sub f\big) \cap (V\times U)$ we have the{ uniform second-order growth condition}
\begin{equation}\label{usogc}
f(x)\geq f(u)+\la v,x-u \ra +\frac{1}{2\kappa}\|x-u\|^2\quad\quad\mbox{for all}\;\; x\in  U.
\end{equation}
\end{itemize}
\end{Proposition}

\section{Characterizations of Tilt-Stable Minimizers}\label{tilt}

This section aims to provide characterizations of tilt-stable local minimizers for a prox-regular function that mostly revolves around the concept of the second subderivative. We begin by recalling the concept of 
a generalized quadratic form \cite{r23} that will be used extensively in this paper. 
\begin{Definition}\label{def:GQF} Suppose that  $\ph:\X\to \oR$ is a proper function. 
\begin{itemize}[noitemsep,topsep=2pt]
\item [ \rm {(a)}] We say that the subgradient mapping $\sub \ph:\X\tto \X$ is {  generalized linear} if its graph is a linear subspace of $\X\times \X$.
\item [ \rm {(b)}]  We say that $\ph$ is a generalized quadratic form (GQF) on $\X$ if $\dom \ph$ is a linear subspace of $\X$ and there exists a linear symmetric operator $L$ {\rm(}i.e. $\la Lx,y\ra=\la x,Ly\ra$
for any $x,y\in \dom \ph${\rm)} from $\dom \ph$ to $\X$ such that $f$ has a representation of form 
$$
\ph(x)=\la Lx,x\ra \quad \mbox{for all}\;\; x\in \dom \ph.
$$ 
\end{itemize}
 \end{Definition}
 Given a proper lsc convex function $\ph:\X\to \oR$ with $\ph(0)=0$, suppose that the subgradient mapping $\sub \ph$ is generalized linear.
 Thus,  it is possible to demonstrate that there is a linear symmetric and positive definite 
 operator $T:S\to S$ such that 
\begin{equation*}\label{rep}
 \sub \ph(x)=T(x)+S^\perp\quad \mbox{for all}\;\; x\in S,
 \end{equation*}
 where $S=\dom \sub \ph$ and that $S$ is a linear subspace of $X$ and $S^\perp =\sub \ph(0)$; see \cite[page~6]{hs23} for a detailed discussion on this claim. For any nonempty set $C$ in $\X$, denote {$\Pi_C$} as the projection mapping and $\dd_C$ as its indicator function.  
Define $M:\X\to \X$ by  $M=T\circ  {\Pi}_S$. It is not hard to see that $M(x)=T(x)$ for any $x\in S$  and that 
 $M$ is a linear symmetric and positive definite 
 operator on $\X$ and 
\begin{equation}\label{gqf}
 \ph(x)=  \frac{1}{2}\la M(x),x\ra+\dd_S(x) \quad \mbox{for all}\;\; x\in \X;
\end{equation}
see again  \cite[page~6]{hs23} for more details. 
It is worth adding that the converse of the above result holds as well. Indeed, it follows from \cite[Proposition~4.1]{r85} that 
 for a proper lsc convex function $\ph$ with $\ph(0)=0$,  $\ph$ is a GQF on $\X$ if and only if $  \sub \ph$ is generalized linear. Below, we record a similar result for the second subderivative of a prox-regular function.
 
 \begin{Proposition}\label{quadform}Assume that $f:\X\to \oR$ is prox-regular and subdifferentially continuous at $\ox$ for  $\ov\in   \sub f(\ox)$ and 
 that $f$ is twice epi-differentiable at $\ox$ for  $\ov$. Then 
 $D(\sub f)(\ox,\ov)$
 is generalized linear if and only if  $\d^2 f(\ox,\ov)$ is a GQF on $\X$, meaning that $S:=\dom \d^2 f(\ox,\ov)$ is a linear subspace of $\X$
 and there is a linear symmetric operator $T:S\to S$ such that 
 $$
  \d^2 f(\ox,\ov)(w)=\la (T\circ  {\Pi}_S)(w),w\ra+\dd_S(w),\quad \mbox{for all}\;\; w\in X.
 $$
 \end{Proposition} 
 \begin{proof}The claimed equivalence can be proven using a similar argument
 as \cite[Proposition~3.2]{hs23}. The main driving force in the proof is the fact that by Remark~\ref{monotone}, there exists $r\ge 0$ such that the function $\ph$, defined by 
$\ph(w)=\d^2 f(\ox,\ov)(w)+ r\|w\|^2$ for any $w\in \X$, is convex. The known equivalence for convex functions from \cite[Proposition~4.1]{r85} immediately justifies the claimed equivalence. The representation of the second subderivative falls also directly out of \eqref{gqf}.
 \end{proof}

The GQF property of the second subderivative of prox-regular functions is prevalent in a neighborhood of a point of their graphs of subdifferential mappings in the sense that it holds for almost any point in such a neighborhood, as shown below. We are going to take advantage of this phenomenon in this paper to conduct a thorough analysis of tilt-stable local minimizers of a class of matrix optimization problems.

  \begin{Remark}\label{lip-man}{\rm
Given a  prox-regular and subdifferentially continuous function $f:\X\to \oR$ at $\ox\in \X$ for $\ov\in \sub f(\ox)$, it is well-known that $\gph \sub f$ is a Lipschitzian manifold of dimension $n$
in the sense of \cite[Definition~9.66]{rw}, where $n=\dim \X$; see \cite[Proposition~13.46]{rw} for more detail. This implies that 
for any $(x,v)\in \gph \sub f$ in a neighborhood of $(\ox,\ov)$, the graphical set $\gph \sub f$ is smooth at {almost all} $(x,v)$ in the sense of  \cite[Definition~2.2(a)]{hs23}. The latter amounts in our current framework to saying that $\sub f$ is proto-differentiable at $x$ for $v$ and 
 the tangent cone $T_{\ss \gph \sub f}(x,v)$ is an $n$ dimensional linear subspace of $\X\times \X$ for all such $(x,v)$. Recall from Remark~\ref{proxr} that prox-regularity and subdifferential continuity hold for all such $(x,v)\in \gph \sub f$ in a neighborhood of $(\ox,\ov)$.
 Thus,  it follows from  \cite[Theorem~13.40]{rw}  that 
$$
D(\sub f)(x, v)= \sub\big(\frac{1}{2} \d^2 f(x, v)\big).
$$
Since $D(\sub f)(x, v)$ is generalized linear, we conclude from Proposition~\ref{quadform} that the second subderivative $\d^2 f(x, v)$ is a  GQF. In summary, for any prox-regular and subdifferentially continuous function $f$ at $\ox$ for $\ov$,
 we find a neighborhood ${\cal O}$ of $(\ox,\ov)$ for which at almost every point  $(x,v)\in (\gph \sub f)\cap {\cal O}$,
the second subderivative $\d^2 f(x, v)$ is a GQF. 
In what follows,  the set of all such pairs $(x,v)\in  {\cal O}$ is denoted by $\Tilde{\cal O}$. 
It is also important to add to this discussion that 
in this framework, the set of all points $(x,v)\in \gph \sub f$ close to $(\ox,\ov)$ at which $\sub f$ is proto-differentiable is included in the set of all points $(x,v)$, where $D(\sub f)(x,v)$ is generalized linear. 
}
 \end{Remark}

 Using the discussion above, we characterize tilt-stable local minimizers of prox-regular functions. Note that \cite[Theorem~2.1]{chn} presents a similar neighborhood characterization of tilt-stable local minimizers of a function without 
the restriction to all points at which either proto-differentiability or generalized linearity is satisfied. 

\begin{Proposition}[second-order characterization of tilt stability]\label{tiltch}
 Let $f:\X\to\oR$ be prox-regular and subdifferentially continuous at
$\ox$ for $\ov=0$ and $\kappa>0$.   Then, the following are equivalent.
\begin{itemize}[noitemsep]
\item [\rm{(a)}]  The point $\ox$ is a tilt-stable local  minimizer for $f$ with constant $\kappa$.
\item [\rm{(b)}]  There is a constant $\eta>0$ such that for all $w\in \X$ we have
\begin{equation*}
\la u,w\ra\ge \frac{1}{\kappa}\|w\|^2\quad\mbox{for all}\quad u\in D(\partial f)(x,v)(w) 
\end{equation*}
and all $(x,v)\in \big( \gph \sub f\big) \cap \B_\eta(\ox,0)$ such that $\sub f$ is proto-differentiable at $x$ for $v$.  
\item [\rm{(c)}]  There is a constant $\eta>0$ such that for all $w\in \X$ we have
\begin{equation*}
\la u,w\ra\ge \frac{1}{\kappa}\|w\|^2\quad\mbox{for all}\quad u\in D(\partial f)(x,v)(w) 
\end{equation*}
and all $(x,v)\in \big( \gph \sub f\big) \cap \B_\eta(\ox,0)$ such that $D(\partial f)(x,v)$ is generalized linear. 
\end{itemize}
\end{Proposition}
\begin{proof} The implication    (a)$\implies$(b) can be established 
using a similar argument as that of the proof of the implication (i)$\implies$(ii) in \cite[Theorem~2.1]{chn}, which is based on a direct use of the characterization of tilt-stable local minimizers via the uniform quadratic growth condition from Proposition~\ref{usogc2}. Since clearly (b) implies (c), we are going to show that (c) yields (a).
Assume now (c) holds. We can assume without loss of generality  that 
$\ox=0$ and $f(\ox)=0$, and that $f(x)\ge \frac{r}{2}\|x\|^2$ for all $x\in \X$ by replacing $f$ by $f+\dd_{\B_{\ve}(\ox)}$ with constants $\ve$ and $r$ taken from  \eqref{proxr}. 
We can use some ideas in the proof of the implication (b)$\implies$(a) in \cite[Theorem~1.3]{pr98} to justify (a). Pick $(x,v)\in \big( \gph \sub f\big) \cap \B_\eta(\ox,0)$ such that 
$\gph D(\partial f)(x,v)$ is  generalized linear. Choosing a smaller $\eta$ if necessary, we can assume via Remark~\ref{proxr} that $f$ is prox-regular and subdifferentially continuous at $x$ for $v$.
By Lemma~\ref{quadform}, there is a linear symmetric operator  $T_{x,v}:K_{x,v}\to K_{x,v}$ such that 
$$
\d^2 f(x,v)(w)=\la (T_{x,v}\circ  {\Pi}_{K_{x,v}})(w),w\ra +\dd_{K_{x,v}}(w)\quad \mbox{for all}\;\; w\in \X,
$$
where $K_{x,v}=\dom \d^2 f(x,v)$ is a linear subspace.  
This tells us that 
$$
D(\partial f)(x,v)(w)=\sub\big(\frac{1}{2}\d^2 f(x,v)\big)(w)=  (T_{x,v}\circ  {\Pi}_{K_{x,v}})(w)+ N_{K_{x,v}}(w)\quad \mbox{for all}\;\; w\in \X.
$$
This implies that 
\begin{equation}\label{fgp}
D(\partial f)(x,v)(w)=T_{x,v}(w)+ K_{x,v}^\perp \quad \mbox{for all}\;\; w\in K_{x,v}. 
\end{equation} 
We claim now that $T_{x,v}$ is strongly monotone. To prove it, take $w\in K_{x,v}$ and $u\in D(\partial f)(x,v)(w)$ and conclude from (c) that 
$$
\la T_{x,v}(w),w\ra= \la u,w\ra\ge \frac{1}{\kappa}\|w\|^2,
$$
which clearly shows that the linear operator $T_{x,v}$ is strongly monotone. By \eqref{fgp}, one can easily conclude that 
$D(\partial f)(x,v)$ is strongly monotone. Appealing now to  \cite[Proposition 5.7]{pr96} tells us that $\sub f$ is 
strongly monotone locally around $(\ox,\ov)$ with constant $1/\kappa$. Thus, we find  $s>0$ and a neighborhood $V$ of $\ov$ such that 
the mapping $h(v):= (\sub f)^{-1}(v)\cap \B_s(\ox)$ is at most single-valued on $V$ and $|h(v)-h(v')|\le \kappa \|v-v'\|$   for any $v,v'\in V$ with $h(v)$,  $h(v')\not=\varnothing$. 
 Using the same argument as the one in the last paragraph of the proof of Theorem~1.3 in \cite[page~295]{pr98} which is based on the strong convexity of $\sub f$ locally around $(\ox,\ov)$ and \cite[Proposition~5.5]{pr96}, we can find some $l>0$ such that 
 $f(x)\ge l\|x\|^2$ for $x\in \B_s(\ox)$. 
Shrinking $V$ and $s>0$ if necessary, we can assume without loss of generality that $\mbox{argmin}\{f(x)-\la v,x\ra\;|\; x\in \B_s(\ox)\}$
is nonempty and is inside of the interior of $\B_s(\ox)$ for any $v\in V$, which leads us to
$\mbox{argmin}\{f(x)-\la v,x\ra\;|\; x\in \B_s(\ox)\}\subset h(v)=(\sub f)^{-1}(v)\cap \B_s(\ox)$ whenever $v\in V$.
Since $h$ is at most single-valued on $V$, we  conclude that $\ox$ is a tilt-stable local  minimizer of $f$ with constant $\kappa$. 
\end{proof}

The following characterization of tilt-stability via second subderivative can be regarded as a direct application of Proposition \ref{tiltch}.

\begin{Theorem}[characterization of tilt-stability via second subderivative]\label{ssti}
Let $f:\X\to\oR$ be prox-regular  and subdifferentially continuous at
$\ox$ for $\ov=0$ and $\kappa>0$. 
Then, the following properties are equivalent.
\begin{itemize}[noitemsep]
\item [\rm{(a)}]  The point $\ox$ is a tilt-stable local  minimizer of $f$ with constant $\kappa$.
\item [\rm{(b)}]  There is a constant $\eta>0$ such that for all $w\in \X$ we have
\begin{equation*}
\d^2 f(x, v)(w)\ge \frac{1}{\kappa}\|w\|^2
\end{equation*}
for all $(x,v)\in \big(\gph \sub f\big) \cap \B_\eta(\ox,0)$
such that  $f$ is twice epi-differentiable at $x$ for $v$.

\item [\rm{(c)}]  There is a constant $\eta>0$ such that for all $w\in \X$ we have
\begin{equation*}
\d^2 f(x, v)(w)\ge \frac{1}{\kappa}\|w\|^2
\end{equation*}
for all $(x,v)\in \big(\gph \sub f\big) \cap \B_\eta(\ox,0)$
at which $\d^2 f(x,v)$ is a GQF.
 
\end{itemize}

\end{Theorem}

\begin{proof} Clearly, (b) implies (c). We are going to show that (c) yields (a).To this end, take $\eta$ from (c) and $(x,v)\in \big(\gph \sub f\big) \cap \B_\eta(\ox,0)$ such that 
$\d^2 f(x,v)$ is a GQF. 
Choosing a smaller $\eta$ if necessary, we can assume via Remark~\ref{proxr} that $f$ is prox-regular and subdifferentially continuous at $x$ for $v$.
Note also that $f$ is twice epi-differentiable at $x$ for $v$ according to the last part of Remark~\ref{lip-man} and the fact proto-differentiablity and 
twice epi-differentiability are equivalent for $f$ by \cite[Theorem~13.40]{rw}. Let $(w,u)\in \gph D(\partial f)(x,v)$.
  Since $f$ is twice epi-differentiable at $x$ for $v$, we conclude from \cite[Theorem~13.40]{rw}
that $u\in \sub\big(\frac{1}{2} \d^2 f(x, v)\big)(w)$. By  \cite[Lemma 3.6]{chn2}, we have  $\d^2 f(x, v)(w)=\la u,w\ra$. Appealing now to  Proposition~\ref{tiltch}(c) tells us that   $\ox$ is a tilt-stable local  minimizer of $f$ with constant $\kappa$, which proves (a).

To prove the implication (a)$\implies$(b), we  conclude from the uniform growth condition from Proposition~\ref{usogc2} and the validity of (a) that $\Delta^2_t f(x,v)(w)\ge \|w\|^2/\kappa$ for any $w\in \X$, $t>0$, and $(x,v)\in \gph \sub f$ sufficiently close to $(\ox,0)$, which proves (c) and hence completes the proof. 
\end{proof}

Note that a neighborhood characterization of tilt-stable local minimizers of convex functions was recently established in \cite[Corollary~2.5]{Nghia}. The latter, however, doesn't require for the points  to be taken from set of points at which the function is either twice epi-differentiable or its second subderivative is a GQF
as Theorem~\ref{ssti}. We should also add here that the authors in the recent preprint \cite{kmpv} studied the relationship between the second subderivative and tilt-stable local minimizers of prox-regular functions. Indeed, they showed in \cite[Corollary~2.6]{kmpv} that the variational s-convexity (cf. \cite[Definition~2.2]{kmpv}), which encompasses the concept of tilt stability, implies a similar inequality for the second subderivative as the one in Theorem~\ref{ssti}(b) without restricting the points from a neighborhood at which the function is twice epi-differentiable. 
Note that such a result is a direct consequence of the uniform second-order growth condition. However, the opposite direction of such a result, which was not achieved in \cite{kmpv}, is far more challenging and requires a different approach. 
Theorem~\ref{ssti} goes further and provides a characterization of tilt stability using second subderivatives, which has no counterpart in \cite{kmpv}. 

 The following definition, motivated mainly by Remark~\ref{lip-man},  appeared recently in  \cite[equation~(4.9)]{r23} for convex functions.

\begin{Definition}\label{def:quadbddefi} Suppose that  $f:\X\to \oR$ is a prox-regular and subdifferentially continuous function at $\ox\in \X$ with $f(\ox)$ finite for $\ov\in \sub f(\ox)$. Given the set $\Tilde{\cal O}$ from Remark~{\rm\ref{lip-man}} on which $\d^2 f(x,v)$ is a GQF whenever $(x,v)\in \gph \sub f$, the quadratic bundle of $f$ at $\ox$ for $\ov$, denoted $\quadr f(\ox,\ov)$, is defined by 
$$
q\in \quadr f(\ox,\ov) \iff 
\begin{cases}
q\;\mbox{is a GQF on \X, and }\;\; \exists\, (x^k,v^k)\in (\gph \sub f)\cap\cal \Tilde{\cal O}\\
\mbox{such that}\;\; (x^k,v^k)\to (\ox,\ov), \;\; \d^2 f(x^k,v^k)\xrightarrow{e} q. 

\end{cases}
$$
\end{Definition}

A similar object can be defined using graphical derivatives of $\sub f$ for any $(x,v)\in (\gph \sub f)\cap \Tilde{\cal O}$. 
To that end, define the set-valued mapping $R_{\ox,\ov}:\X\tto \X$ with $(\ox,\ov)\in \gph \sub f$ so that 
\begin{equation}\label{mapr}
\gph R_{\ox,\ov} = \limsup_{(x,v)\xrightarrow{({\ss \gph} \sub f)\cap \Tilde{\cal O}} (\ox,\ov)} \gph D(\sub f)(x, v) =  \limsup_{(x,v)\xrightarrow{({\ss \gph} \sub f)\cap \Tilde{\cal O}} (\ox,\ov)} T_{\ss \gph \sub f}(x,v).
\end{equation}
It is important to mention here that this construction was first defined in the proof of the implication (b)$\implies$(a) in \cite[Theorem~2.1]{pr98}. One could alternatively define 
the mapping $R_{\ox,\ov}$ as 
\begin{equation}\label{mapr2}
\gph R_{\ox,\ov} =  \bigcup\big\{ L\big|\;   L\in S_{\ox,\ov}\big\},
\end{equation}
where   $S_{\ox,\ov}$ is defined by 
\begin{equation}\label{mapr3}
S_{\ox,\ov}:= \big\{ L\in {\cal L}_n\big|\;  \exists\, (x^k,v^k)\xrightarrow{(\ss \gph \sub f)\cap\Tilde{\cal O}} (\ox,\ov)\; \mbox{such that}\; T_{\ss \gph \sub f}(x^k,v^k)\to L\big\}.
\end{equation}
Here ${\cal L}_n$ stands for the set of all  linear subspaces of $\X\times \X$ of dimension $n$ with $n=\dim \X$. 
The latter set in the right-hand side of \eqref{mapr2} was recently introduced in \cite[Definition~3.3]{go} using a different method, which is  equivalent to the given definition above. 
To see
why \eqref{mapr2} holds,  recall first  from \cite[page~116]{rw} that a sequence of sets $\{C^k\b$ in $\X$ is called to {\em escape to the horizon} if $\limsup_{k\to \infty} C^k= \varnothing$. 
An important example of such a sequence of sets in our framework in this paper is when each $C^k$ is a linear subspace of $X$. Since we clearly have 
$0\in \limsup_{k\to \infty} C^k$, we can conclude that the sequence $\{C^k\b$  doesn't escape to the horizon. In such a case, one can use \cite[Theorem~4.18]{rw}
to conclude that the sequence $\{C^k\b$ always has a convergent subsequence. 

Turing back to the proof of \eqref{mapr2}, it is easy to see that the set in the right-hand side of \eqref{mapr2} is always included in $\gph R_{\ox,\ov}$.
To get the opposite inclusion, take $(w,u)\in \gph R_{\ox,\ov}$. This gives us  sequences $\{(x^k,v^k)\b\subset ( \gph \sub f)\cap \Tilde{\cal O}$ and $(w^k,u^k)\in   L^k:= T_{\ss \gph \sub f}(x^k,v^k) $ such that 
$(x^k,v^k)\to (\ox,\ov)$ and $(w^k,u^k)\to (w,u)$ as $k\to \infty$. Since each $L^k$ is a linear subspace, we infer from the discussion above that the sequence $\{L^k\b$ doesn't escape to the horizon.  
Appealing now to \cite[Theorem~4.18]{rw}, we can assume without loss of generality that   the sequence $\{L^k\b$   is  convergent  to a set $L\subset \X\times \X$. 
Since each $L^k$ is a linear subspace of dimension $n$, one can easily see that $L$ enjoys the same property and thus belongs to  $S_{\ox,\ov}$
and $(w,u)\in L$. This proves the inclusion `$\subset$' in \eqref{mapr2} and finishes the proof. 

To proceed, we need to present the following result that provides a direct relationship between the quadratic bundle and the mapping $R_{\ox,\ov}$ in \eqref{mapr}.
Such a relationship appeared first in \cite[Proposition~3.33]{go}, but we supply a different proof below that is more compatible with our approach in this paper. Moreover, it was assumed in the latter result that $f$ is subdifferentially continuous in a neighborhood of the point in question, which is unnecessary due to Remark~\ref{proxr}.
\begin{Proposition} \label{quadn} 
Assume that  $f:\X\to\oR$ is prox-regular  and subdifferentially continuous at
$\ox$ for $\ov=0$.  Then, we have 
$$
\gph R_{\ox,\ov}=  \bigcup\big\{  \gph \sub (\sm  q )\big|\; q\in \quadr f(\ox,\ov)\big\}.
$$
\end{Proposition}
\begin{proof} The proof of the inclusion `$\supset$' falls directly out of 
Attouch's theorem from \cite[Theorem~12.35]{rw} as argued in the proof of \cite[Proposition~3.33]{go}. Indeed, if $q\in \quadr f(\ox,\ov)$, we find 
$(x^k,v^k)\in \gph \sub f$ converging to $(\ox,\ov)$ such that $\d^2 f(x^k,v^k)\xrightarrow{e} q$ and $\d^2 f(x^k,v^k)$ is a GQF for each $k$.
Using Remark~\ref{monotone}, Proposition~\ref{quadform}, and \eqref{newq}, 
we arrive at $\gph \sub (\sm  q )\subset \gph R_{\ox,\ov}$.

To prove the opposite inclusion, we 
proceed differently. By \eqref{mapr2}, take $L\subset \gph R_{\ox,\ov}$. So, we find a  sequence $\{(x^k,v^k)\b\subset ( \gph \sub f)\cap \Tilde{\cal O}$
such that $L^k:= T_{\ss \gph \sub f}(x^k,v^k)\to L$. Suppose  $f$ is prox-regular at $\ox$ for $\ov$ with constants $\ve $ and $r$ satisfying \eqref{prox}. It follows from  Remark~\ref{proxr} that  $f$ is prox-regular and subdifferentially continuous at $x^k$ for $v^k$ with the same constant $r$ and perhaps a smaller $\ve$ for any $k$ sufficiently large. By Remark~\ref{monotone}, the latter tells us that $\sub f +rI$ is monotone, which coupled again with Remark~\ref{monotone}
allows us to conclude that  the mapping $w\mapsto \frac{1}{2} \d^2 f(x^k, v^k)(w)+ \frac{r}{2}\|w\|^2$ is convex. This implies   
 for any $k$ sufficiently large that $H^k:= D(\sub f)(x^k,v^k)+rI$ is cyclically maximal monotone; see \cite[Definition~12.24]{rw} for its definition. Moreover, we have 
 $$
 \gph H^k =M( L^k)
 \quad \mbox{with}\;\;    M:=\begin{bmatrix}
 I & 0\\
 rI& I
 \end{bmatrix}.
 $$
 This, coupled with \cite[Theorem~4.27]{rw} and the convergence of $L^k$ to $L$, demonstrates 
 that the sequence $\{\gph H^k\b$ converges to  $M(L)$. Since $H^k$ 
 are  cyclically  maximal monotone, $M(L)$ is the graph of a cyclically  maximal monotone mapping due to \cite[Theorem~12.32]{rw}.
 Appealing now to \cite[Theorem~12.45]{rw} tells us that we can find a proper lsc convex function $h:\X\to \oR$ such that $M(L)=\sub h$. 
 Moreover, we can conclude that $\sub h$ is generalized linear, since $L^k$ is a linear subspace. 
Consequently, it follows from Attouch's theorem (cf. \cite[Theorem~12.35]{rw}) that $\frac{1}{2} \d^2 f(x^k, v^k)+ \frac{r}{2}\|\cdot\|^2\xrightarrow{e}  h$ and hence $\d^2 f(x^k, v^k)\xrightarrow{e}  2h- r\|\cdot\|^2:= q$. 
We know that  $M$ is invertible and thus arrive at 
$$
L= M^{-1}(\gph \sub h) =\begin{bmatrix}
 I & 0\\
- rI& I
 \end{bmatrix}(\gph \sub h)= \gph \sub (\sm q).
$$
Since $q\in \quadr f(\ox,\ov)$, we get the inclusion `$\subset$' in the claimed equality in the theorem, which  completes the proof.
\end{proof}

We proceed with justifying the sequential compactness of the quadratic bundle for prox-regular functions in the epi-convergence topology, which has  important applications in our results at the end of  this section.  To this end, suppose that $f:\X\to\oR$ is prox-regular  and subdifferentially continuous at $\ox$ for $\ov$ and that the constants $r$ and $\ve$ satisfy \eqref{prox}. According to Remark~\ref{proxr}, $f$
is prox-regular and subdifferentially continuous at $x$ for $v$ whenever 
$(x,v)\in \gph \sub f$ is sufficiently close to $(\ox,\ov)$ with the same constant $r$. For any such a pair $(x,v)$, we can conclude from \cite[Theorem~4.4]{pr96} that for any $\gamma\in (0,1/r)$ and any $z$ close to $x+\gamma v$ we have
$$
\mbox{prox}_{\gamma f}(z)= (I+\gamma T)^{-1}(z),
$$
where $T$ is a graphical localization of $\sub f$. A similar argument as \cite[Exercise~12.64]{rw} brings us to 
\begin{equation}\label{gd_prox}
 D \mbox{prox}_{\gamma f}(x+\gamma v)= (I+\gamma D(\sub f)(x,v))^{-1}.
\end{equation}
In addition, if $(x,v)\in \Tilde{\cal O}$ with the set $\Tilde{\cal O}$ taken from Remark~\ref{lip-man}, then $D(\sub f)(x,v)$ is generalized linear, which together with \eqref{gd_prox} implies that $ D \mbox{prox}_{\gamma f}(x+\gamma v)$ is generalized linear. Since $ \mbox{prox}_{\gamma f}$
is Lipschitz continuous around $x+\gamma v$, we deduce from \cite[Proposition~3.1]{r85} that the latter property of $ \mbox{prox}_{\gamma f}$ is equivalent to its differentiability at $x+\gamma v$. Given a   mapping  $\Phi\colon\X\to\Y$ that is  Lipschitz continuous around $x\in \X$, define the collection of its limiting Jacobian matrices by 
\begin{equation*}\label{radem}
\Bar\nabla \Phi(x):=\Big\{\lim_{k\to\infty}\nabla \Phi(x^k)\Big|\;x^k\to\ox,\;x^k\in\O_\Phi\Big\},
\end{equation*}
where $\O_\Phi$ stands for the set on which $\Phi$ is differentiable.

\begin{Proposition}\label{quad_top}
Assume that  $f:\X\to\oR$ is prox-regular  and subdifferentially continuous at
$\ox$ for $\ov$. Then the quadratic bundle of $f$ at $\ox$ for $\ov$ enjoys the following properties.
\begin{itemize}[noitemsep]
\item [\rm{(a)}] There is a positive constant $r$ such that $q+r\|\cdot\|^2$ is a proper lsc convex function for any $q\in \quadr f(\ox,\ov)$.

\item [\rm{(b)}] Any sequence $\{q^k\b\subset \quadr f(\ox,\ov)$ has a subsequence that epi-converges to a GQF in  $\quadr f(\ox,\ov)$.

\end{itemize}
\end{Proposition}

\begin{proof} To prove (a), take $q\in \quadr f(\ox,\ov)$. By definition, we find a sequence $\{(x^k,v^k)\b\subset \gph \sub f$, converging to $(\ox,\ov)$, such that $\d^2 f(x^k,v^k)\xrightarrow{e} q$ and $\d^2 f(x^k,v^k)$ is a GQF for each $k$. Appealing to Remark~\ref{monotone}, we find a positive constant $r$ so that the mapping $w\mapsto \d^2 f(x^k, v^k)(w)+ r\|w\|^2$ is convex for any $k$ sufficiently large. Since $\d^2 f(x^k, v^k)+ r\|\cdot\|^2\xrightarrow{e} q + r\|\cdot\|^2$, we deduce from \cite[Theorem~7.17]{rw} that $q + r\|\cdot\|^2$ is convex. By \cite[Proposition~7.4(a)]{rw}, the latter function is lsc. Its properness results from prox-regularity of $f$ at $x^k$ for $v^k$ and the fact that 
$(0,0)\in \epi \d^2 f(x^k,v^k)$ for any $k$ sufficiently large. 

To justify (b), suppose that  the sequence $\{q^k\b\subset \quadr f(\ox,\ov)$.
Set $L^k:=\gph \sub (\sm q^k)$ and observe from  the proof of the inclusion  `$\supset$' in Proposition~\ref{quadn}  that 
 $L^k$  are linear subspaces, belonging to $S_{\ox,\ov}$, which is defined by \eqref{mapr2}. 
 On the other hand, \eqref{gd_prox} tells us that the graphs of $D(\sub f)(\ox,\ov)$ and $ D \mbox{prox}_{\gamma f}(\ox+\gamma \ov)$ coincide locally up to a change of coordinate. Since the graphical convergence and pointwise convergence of a sequence of bounded linear mappings are equivalent (cf. \cite[Theorem~5.40]{rw}), we can conclude from \eqref{gd_prox} that an element  $L\in S_{\ox,\ov}$ corresponds to an element $A\in \Bar\nabla\mbox{prox}_{\gamma f}(\ox+\gamma \ov)$ such that 
 $$
 M L= \gph A\quad \mbox{with}\quad M:=\begin{bmatrix}
     I & \gamma I\\
     I & 0
 \end{bmatrix} .
 $$
It is well known that $\Bar\nabla\mbox{prox}_{\gamma f}(\ox+\gamma \ov)$
is sequentially compact (cf. \cite[Theorem~9.62]{rw}). Since $M$ is invertible, the latter property implies that 
 the sequence $\{L^k\b$  has a subsequence, convergent to some  element $L\in S_{\ox,\ov}$. By Proposition~\ref{quadn}, we find $q\in \quadr f(\ox,\ov)$ so that $L=\gph \sub (\sm q)$. Thus we arrive at 
$L^k=\gph \sub (\sm q^k)\to \gph \sub (\sm q)=L$. 
Therefore, it follows from (a) and Attouch's theorem in \cite[Theorem~12.35]{rw} that $q^k\xrightarrow{e} q $, which completes the proof of (b). 
\end{proof}

Below, we present our first characterization of tilt-stable local minimizers of a function via its quadratic bundle. 
\begin{Theorem}[characterization of tilt-stability via quadratic bundle]\label{tsqb}
Assume that  $f:\X\to\oR$ is prox-regular  and subdifferentially continuous at
$\ox$ for $\ov=0$ and   that $\kappa$ and $\ell$ are two positive constants. 
Consider the following properties:
\begin{itemize}[noitemsep]
\item [\rm{(a)}]  The point $\ox$ is a tilt-stable local  minimizer of $f$ with constant $\kappa$.
 
\item [\rm{(b)}] The following  condition  holds: $q(w)\ge \ell \|w\|^2$ for all $w\in \X$ and all $q\in \quadr f(\ox,\ov)$. 
 \end{itemize}
 The implication {\rm(}a{\rm)}$\implies${\rm(}b{\rm)} holds for $\ell=1/\kappa$. The opposite implication is satisfied for any $\ell>1/\kappa$.

\end{Theorem}
\begin{proof}
Assume now (a) holds and take $q\in \quadr f(\ox,\ov)$. Thus, we find a sequence $\{(x^k,v^k)\b\subset  {\cal O}\cap (\gph \sub f)$, converging to $ (\ox,\ov)$, such that 
$ \d^2 f(x^k,v^k)\xrightarrow{e} q$.    This, coupled with Theorem~\ref{ssti}(c),  implies that 
$$
q(w)=(\limi_{k\to \infty} \d^2 f(x^k, v^k))(w)=  \liminf_{k\to \infty,\, w'\to w} \d^2 f(x^k, v^k)(w')\ge \frac{1}{\kappa}\|w\|^2,
$$
where the second equality comes from \cite[equation~7(6)]{rw}. This 
 proves (b) for $\ell=1/\kappa$. Turning now to the proof of the opposite implication (b)$\implies$(a),  pick $\ell>0$ such that $\ell>1/\kappa$.
We claim that there exists $\eta>0$ such that for any $(x,v)\in \big(\gph \sub f\big) \cap \B_\eta(\ox,0)$
at which $\d^2 f(x,v)$ is a GQF and any $w\in \X$, the inequality 
\begin{equation}\label{cso}
\d^2 f(x, v)(w)\ge \frac{1}{\kappa}\|w\|^2
\end{equation}
 is satisfied. Suppose to the contrary that the latter claim fails. So, we find  sequences  $(x^k,v^k)\in \gph \sub f$,   
 converging to $(\ox,\ov)$, and $\{w^k\b\subset \X$ for which $\d^2 f(x^k,v^k)$ is a GQF and the inequality 
\begin{equation}\label{cso2}
 \d^2 f(x^k, v^k)(w^k)< \frac{1}{\kappa}\|w^k\|^2
\end{equation}
 holds for all $k$. Without loss of generality, we can assume that $\|w^k\|=1$ and $w^k\to w$ for some $w\in \X\setminus \{0\}$ and 
that   $f$ is prox-regular and subdifferentially continuous at $x^k$ for $v^k$  for any $k$ sufficiently large due to  Remark~\ref{proxr}.
 It follows from \cite[Theorem~13.40]{rw} that $D(\sub f)(x^k, v^k)= \sub\big(\frac{1}{2} \d^2 f(x^k, v^k)\big)$ for all sufficiently large $k$.
 Since $\d^2 f(x^k,v^k)$ is a GQF, we conclude from Lemma~\ref{quadform} that $D(\sub f)(x^k, v^k)$
 is  generalized linear, which implies  that $\gph D(\sub f)(x^k, v^k)=T_{\ss \gph \sub f}(x^k,v^k)$ is a linear subspace for all sufficiently large $k$.
Since each $L^k:=T_{\ss \gph \sub f}(x^k,v^k)$ is a linear subspace, we have $0\in \limsup_{k\to \infty} L^k$ and thus  the sequence $\{L^k\b$ doesn't escape to the horizon.  
Appealing now to \cite[Theorem~4.18]{rw}, we can assume without loss of generality that   the sequence $\{L^k\b$   is  convergent  to a linear subspace $L$ in  $ \X\times \X$. 
Using a similar argument as the proof of Proposition~\ref{quadn}, we find $q\in \quadr f(\ox,\ov)$ such that $L=  \gph \sub (\sm q)$
and $\d^2 f(x^k, v^k)\xrightarrow{e}    q$.  This, coupled with \eqref{cso2}, brings us to 
$$
q(w)=(\limi_{k\to \infty}  \d^2 f(x^k, v^k))(w)\le  \liminf_{k\to \infty} \d^2 f(x^k, v^k)(w^k)\le \frac{1}{\kappa} < \ell.
$$
On the other hand, it follows from $q\in \quadr f(\ox,\ov)$, (b), and $\|w\|=1$ that $q(w)\ge \ell$, a contradiction. This proves the implication (b)$\implies$(a)
and hence completes the proof.
\end{proof}

The theorem above provides a quantitative result on the characterization of tilt-stable local minimizers of a function via its quadratic bundle. In particular, it gives us a relationship between the constant of tilt stability and $\ell$, Theorem~\ref{tsqb}(b). Moreover, the condition in 
Theorem~\ref{tsqb}(b)  is desired to be expressed without the appearance of constant $\ell$ therein. Below, we record our final result in which all these considerations will be taken into account to present a simpler characterization of tilt-stable minimizers using the concept of quadratic bundle. 

\begin{Theorem}\label{pos-tilt}
Assume that  $f:\X\to\oR$ is prox-regular  and subdifferentially continuous at
$\ox$ for $\ov=0$. 
Then the following properties are equivalent:
\begin{itemize}[noitemsep]
\item [\rm{(a)}]  The point $\ox$ is a tilt-stable local  minimizer of $f$.
 
\item [\rm{(b)}] For any $w\in \X\setminus \{0\}$ and any $q\in \quadr f(\ox,\ov)$, we have $q(w) >0$.
 \end{itemize}
\end{Theorem}

\begin{proof} The implication (a)$\implies$(b) results from Theorem~\ref{tsqb}. Assume that (b) is satisfied.
We claim that there exists $\ell>0$ for which we have $q(w)\ge \ell \|w\|^2$ for all $w\in \X$ and all $q\in \quadr f(\ox,\ov)$.
Suppose for contradiction that there are sequences $\{q^k\b\subset \quadr f(\ox,\ov)$ and $\{w^k\b\subset \X$ such that for any $k$ the estimate 
$$
q^k(w^k) < \frac{1}{k} \|w^k\|^2
$$
is satisfied. 
Without of loss of generality, we can assume that $\|w^k\|=1$ and that $w^k\to w$, where $w\in \X$ with $\|w\|=1$. 
Appealing now to Proposition~\ref{quad_top}(b), we can assume 
by passing to a subsequence if necessary that $q^k \xrightarrow{e} q$ with  $q\in  \quadr f(\ox,\ov)$, 
which brings us to 
$$
q(w)=(\limi_{k\to \infty}q^k)(w)\le  \liminf_{k\to \infty} q^k(w^k)\le 0.
$$
This clearly  is a contradiction with (b), since $w\neq 0$ and hence completes the proof.
\end{proof}

Note that tilt stability of local minimizers was first characterized in \cite{pr98} via the concept of coderivative of the subgradient mappings of prox-regular and subdifferentially continuous functions.
A neighborhood characterization via the concept of regular coderivative was established in \cite{mn14} for tilt stability of local minimizers. Quite recently, a new characterization of tilt-stable local minimizers was achieved in \cite[Theorem~7.9]{go} using the concept of the subspace contained derivative (SCD). Note that the latter concept has a direct relationship with the quadratic bundle of a function according to Proposition~\ref{quadn}. While the approach in \cite{go} relies on the characterization of tilt stability via the coderivative, our results try to avoid any dual constructions in the characterization of tilt stability of local minimizers.  Finally, we should add here that quite recently, a characterization of tilt-stable local minimizers via the concept of quadratic bundle was achieved in \cite[Theorem~5.2]{kmpv}, which is similar to Theorem~\ref{tsqb}. The approach used there, however, is different from the one  in this section. 
It is worth noting that the characterization of tilt-stable local minimizers presented in Theorem~\ref{pos-tilt}, which is widely applied in practice, also differs from the results obtained by~\cite{kmpv}. In particular, the characterization in Theorem~\ref{pos-tilt} relies essentially on Proposition~\ref{quad_top}, which was not addressed in~\cite{kmpv}. 

We close this section by pointing out that the subdifferential continuity, used throughout this section, can be dropped with no harm.
In doing so, one needs to replace the subgradient mapping $\sub f$ with its $f$-attentive  graphical  localization of $\sub f$ (see \cite{pr96} for more details) in the presented results in this section. While there are not many functions which don't enjoy subdifferential continuity, it was recently observed in \cite[Example~2.5]{kmpv} that the latter condition fails for the $l_{0}$ pseudo-norm; see \cite{kmpv} for more discussion.

\section{Quadratic Bundle of Spectral Functions}\label{sec:qbdofspfunc}
The results in Theorems~\ref{tsqb} and \ref{pos-tilt} open a new door to characterize tilt-stable local minimizers of a function via its quadratic bundle. The question  then boils down  to compute 
the quadratic bundle for different classes of functions. While this does not seem to be an easy target in general, the following discussion indicates that in many situations, we 
would probably get away with finding only one quadratic bundle that  in some sense is {\em minimal}. To motivate our discussion further, 
recall that for a function $f:\X \to \oR$ and a point $\ox\in\X$ with $f(\ox)$ finite,   the subderivative function $\d f(\ox)\colon\X\to\oR$ is defined for any $w\in \X$ by
\begin{equation*}\label{fsud}
\d f(\ox)(w)=\liminf_{\substack{
   t\searrow 0 \\
  w'\to w
  }} {\frac{f(\ox+tw')-f(\ox)}{t}};
\end{equation*}
its critical cone   at $\ox$ for $\bar v$ with $\bar v\in   \sub f(\ox)$ is defined by 
\begin{equation*}\label{cricone}
{{\cal C}_f}(\ox,\bar v)=\big\{w\in \R^n\,\big|\,\la\bar v,w\ra=\d f(\ox)(w)\big\}.
\end{equation*}

Note that the critical cone of a function has a close relationship with its second subderivative. Indeed, given $(\ox,\ov)\in \gph \sub f$ such that $\d^2 f(\ox,\ov)$ is proper, it follows from \cite[Proposition~13.5]{rw} that 
$\dom \d^2 f(\ox,\ov)\subset {{\cal C}_f}(\ox,\bar v) $. Equality, however, requires some additional assumptions. Interested readers can find  some sufficient conditions  in \cite[Proposition~3.4]{ms20} using the concept of the parabolic subderivative. 

The following result
shed more light on the idea of a minimal quadratic bundle that we are going to pursue later in this paper. While we provide a short proof for readers' connivance,
we should mention that it can be gleaned from the proof of   \cite[Theorem~2]{r23}, which uses a different approach.  Recall that the function $\th:\R^n\to \oR$ is called polyhedral if   $\epi \th$ is a polyhedral convex set.  Recall also that   a closed face $F$ of a  polyhedral convex cone $C\subset \R^d$ is defined by 
$$
F=C\cap [v]^\perp\quad \mbox{for some}\;\; v\in C^*.
$$

\begin{Proposition} \label{qbpol} Assume that $\th:\R^n\to \oR$ is a   polyhedral function and $(\ox,\ov)\in \gph \sub \th$. Then we have 
\begin{equation}\label{poly}
q(w)\ge \dd_{{\cal C}_\th  (\ox,\ov  )-{\cal C}_\th  (\ox,\ov)}(w)\quad \mbox{for all}\;\;w\in \R^n\;\;\mbox{and all}\; q\in \quadr \th(\ox,\ov).
\end{equation}
Moreover, we have $ \dd_{{\cal C}_\th  (\ox,\ov)-{\cal C}_\th  (\ox,\ov)}\in \quadr \th(\ox,\ov)$.
\end{Proposition}
\begin{proof} According to \cite[Proposition~3.3]{hjs}, there is a sequence $\{(x^k,v^k)\b\subset \gph \sub \th$, converging to $(\ox,\ov)$,
such that ${\cal C}_\th(x^k,v^k)={\cal C}_\th  (\ox,\ov  )-{\cal C}_\th  (\ox,\ov)$, since ${\cal C}_\th  (\ox,\ov)$ is a face of the polyhedral set ${\cal C}_\th  (\ox,\ov)$. Moreover, 
it follows from \cite[Proposition~13.9]{rw} that $\d^2 \th(x^k, v^k)=\dd_{{\cal C}_\th(x^k,v^k)}$. Since ${\cal C}_\th(x^k,v^k)$ is a linear subspace, the second 
subderivative $\d^2 \th(x^k, v^k)$ is a GQF. Clearly, $\d^2 \th(x^k, v^k)\xrightarrow{e}  \dd_{{\cal C}_\th  (\ox,\ov  )-{\cal C}_\th  (\ox,\ov)}$, which proves that 
$\dd_{{\cal C}_\th  (\ox,\ov  )-{\cal C}_\th  (\ox,\ov)}\in \quadr \th(\ox,\ov)$. To prove \eqref{poly}, pick $q\in \quadr \th(\ox,\ov)$ and find 
a sequence $(x^k,v^k)\in \gph \sub \th$ such that $(x^k,v^k)\to (\ox,\ov)$ and $\d^2 \th(x^k, v^k)=\dd_{{\cal C}_\th(x^k,v^k)}\xrightarrow{e}    q$ with $\d^2 \th(x^k, v^k)$ being a GQF for each $k$.
It follows again from \cite[Proposition~3.3]{hjs} that ${\cal C}_\th(x^k,v^k)\subset {\cal C}_\th  (\ox,\ov  )-{\cal C}_\th (\ox,\ov)$ for all $k$ sufficiently large, which leads 
us to 
$$
\epi \dd_{{\cal C}_\th(x^k,v^k)} ={\cal C}_\th(x^k,v^k)\times [0,\infty)\subset \big({\cal C}_\th  (\ox,\ov  )-{\cal C}_\th  (\ox,\ov)\big)\times [0,\infty).
$$
Thus, we arrive at the inclusion $\epi q\subset  \big({\cal C}_\th  (\ox,\ov  )-{\cal C}_\th  (\ox,\ov)\big)\times [0,\infty)$, which clearly proves \eqref{poly}.
\end{proof}

 The result above suggests that we may expect the quadratic bundle for some functions consists of a minimal element in certain sense,
 which can be leveraged to simplify the characterization of tilt-stable local minimizers for some classes of optimization problems. This motivates
 the following definition. 
 
 \begin{Definition} \label{mquad} Suppose that  $f:\X\to \oR$ is a prox-regular function at $\ox\in \X$ with $f(\ox)$ finite for $\ov\in \sub f(\ox)$.
 We say that $f$  enjoys the minimal  quadratic bundle property at $\ox$ for $\ov$ if there exists $\bar q\in  \quadr f(\ox,\ov)$ such that 
 $q(w)\ge \bar q(w)$ for any $q\in  \quadr f(\ox,\ov)$ and any $w\in \X$. In this case, we refer to $\overline{q}$ as a minimal quadratic bundle of $f$ at $\ox$ for $\ov$. 
 \end{Definition}

According to Proposition~\ref{qbpol}, polyhedral functions enjoy the minimal  quadratic bundle property. We should add here that the concept of the minimal quadratic bundle was explored for some classes of functions  in \cite {r23} without explicitly providing a definition. Note that  \cite[Question~1]{r23} presents a convex piecewise linear quadratic function that doesn't enjoy the minimal quadratic bundle property. Our main goals in this section are first to demonstrate that a subclass of spectral functions enjoys the latter property as well and then  to compute a minimal element in the quadratic bundle 
of this class of functions. 

\begin{Corollary}\label{tilt_minimal}
Assume that  $f:\X\to\oR$ be prox-regular and subdifferentially continuous at
$\ox$ for $\ov=0$,   and that 
$f$ enjoys the minimal  quadratic bundle property at $\ox$ for $\ov$. 
Then the following properties are equivalent:
\begin{itemize}[noitemsep]
\item [\rm{(a)}]  The point $\ox$ is a tilt-stable local  minimizer of $f$.
 
\item [\rm{(b)}] There exists {a minimal quadratic bundle} $\bar q\in  \quadr f(\ox,\ov)$ such that for any $w\in \X\setminus \{0\}$  we have $\bar q(w) >0$.
 \end{itemize}
\end{Corollary}

\begin{proof} The claimed equivalence is an immediate consequence of Theorem~\ref{pos-tilt} and Definition~\ref{mquad}.
\end{proof}

In the rest of this section, our main goal  is to show that  an important subclass of spectral functions has a minimal quadratic bundle and then to calculate such an element. To this end, we begin by recalling some notation. 
In what follows, $\S^n$ stands for the linear space of all $n\times n$ real symmetric matrices equipped with the usual Frobenius inner product and its induced norm. 
For a matrix $X\in\S^n$, we use $\lambda_1(X)\geq\lambda_2(X)\geq\dots\geq\lambda_n(X)$ to denote the eigenvalues of $X$ (all real and counting multiplicity) arranging in nonincreasing order and use $\lambda(X)$ to denote the vector of the ordered eigenvalues of $X$. Let $\Lambda(X)=\mbox{Diag}(\lambda(X))$. 
Suppose $X\in\S^n$ has the following eigenvalue decomposition:
\begin{equation}\label{eq:eig-decomp1}
X= \left[\begin{array}{ccc}
{P}_{\alpha^1} & \cdots & {P}_{\alpha^{r}}
\end{array}\right] \left[\begin{array}{ccc}
\Lambda(X)_{\alpha^1\alpha^1} &  & 0 \\
  & \ddots &  \\
0 &   & \Lambda(X)_{\alpha^{r}\alpha^{r}}
\end{array}\right]\left[\begin{array}{c}
{P}_{\alpha^1}^{\top} \\
 \vdots \\
 {P}_{\alpha^{r}}^{\top}
\end{array}\right],
\end{equation}
where $\mu_1>\cdots>\mu_r$ are the distinct eigenvalues of $X$ and 
the index sets $\al^l$ are defined by 
\begin{equation}\label{eq:def-alpha}
	\alpha^l:=\{1\le i\le n \mid {\lambda}_i(X)=\mu_l\}, \quad l=1,\dots, r.
\end{equation} 
Denote also by ${\bf O}^n$ as the set of $n-$dimensional orthogonal matrices. Suppose that $g:\S^n\to \oR$  has the composite representation  
 \begin{equation}\label{spec}
g(X)=(\th\circ\lm)(X), \;\; X\in \S^n,
\end{equation}
where $\th:\R^n\rightarrow\oR$ is a symmetric polyhedral function. 
A symmetric function means that for all permutation matrix $U\in{\cal P}^{n}$, where ${\cal P}^n$ standing for the set of n-dimensional permutation matrices,  we have $\theta({x})=\theta(U{x})$.  
Given $({X}, {Y})\in {\rm gph}\,\partial g$ and  $P \in {\bf O}^n({X}) \cap {\bf O}^n({Y})$, define ${\bf O}^n(X)$ as the set of all matrices satisfying \eqref{eq:eig-decomp1}. As a result of \cite[page 164]{Lewis96}, $g$ should be convex, which implies via   \cite[Example 13.30]{rw} that the spectral function $g$   is prox-regular and subdifferentially continuous  at every point of its domain.

Before presenting our  main result of this section, we need the following lemma about  
the explicit form of the second subderivative, which is crucial for the derivation of the main result.

\begin{Proposition}\label{sec_subd}
 Assume that $g$ has the spectral representation in \eqref{spec} and  that $({X}, {Y}) \in \operatorname{gph} \partial g$. Then the following properties hold:
 \begin{itemize}[noitemsep]
\item [\rm{(a)}] 
 If  $X$ has the eigenvalue decomposition \eqref{eq:eig-decomp1} with $P\in {\bf O}^n({X}) \cap {\bf O}^n({Y})$, then $g$ is twice epi-differentiable at
 $X$ for $Y$ and 
 its second subderivative at $X$ for $Y$ can be calculated for any $H\in \S^n$ by 
 \begin{equation}\label{second_sub}
 \d^2g(X,Y)(H)= \varUpsilon_{{X},{Y}}(H)+ \dd_{{\cal C}_g(X,Y)}(H),    
 \end{equation}
 where 
 \begin{eqnarray}
\varUpsilon_{{X},{Y}}(H)&=& 2\sum_{l=1}^{r}  \big\la \Lm(Y)_{\al^l\al^l}, P_{\al_m}^\top H ( \mu_{l}I- X)^{\dagger} H  P_{\al^l}\big\ra\nonumber\\
&=&2\sum\limits_{1\leq l< l'\leq r}\sum\limits_{i\in\alpha^l}\sum\limits_{j\in\alpha^{l'}}\frac{\lambda_i({Y})-\lambda_j({Y})}{\lambda_i(X)-\lambda_j(X)}(P^{\top}HP)^2_{ij},\label{eq:sigtsdp}
\end{eqnarray}
and  where   $\al^l$, $l=1,\ldots,r$, come from \eqref{eq:def-alpha}, and 
$P\in {\bf O}^n({X}) \cap {\bf O}^n({Y})$. 
\item [\rm{(b)}] We have ${\cal C}_g(X,Y)=N_{\sub g(X)}(Y)$. Moreover, the second subderivative $\d^2g(X,Y)$ is a GQF if and only if $Y\in \ri \sub g(X)$. 
\end{itemize}
\end{Proposition}

\begin{proof} The first claim in (a) was taken from \cite[Corollary 5.9]{MSarabi23}; see also  \cite[Propositions 6 and 10]{cuiding} for a similar result. 
The first claim in (b) results from \cite[Theorem~8.30]{rw}.
To prove the second claim in (b), we deduce from  \eqref{second_sub} that  the second subderivative $\d^2g(X,Y)$ is a GQF
 if and only if ${\cal C}_g(X,Y)$ is a linear subspace of $\S^n$. It follows from    ${\cal C}_g(X,Y)=N_{\sub g(X)}(Y)$ 
 that the latter property amounts to  $Y\in \ri \sub g(X)$.
\end{proof}

Recall that according to \cite[Theorem 2.49]{rw}, the polyhedral function $\th$ has a representation of the form 
\begin{equation}\label{eq:def_phi12}
\begin{cases}
\th=\theta_1+\theta_2\;\;\mbox{with}\\
\th_1(x):=\max\limits_{1\leq \nu\leq p}\{\langle {a}^\nu, {x}\rangle-c_\nu\}\;\;\mbox{and }\;\; \theta_2(x):=\delta_{{\rm dom}\,\theta}(x),    
\end{cases}
\end{equation}
where $\{({a}^\nu, c_\nu)\}_{\nu=1}^p\subset \R^n\times\R$ for a positive integer $p$ and where
\begin{equation*}\label{eq:def_psi}
{\rm dom}\,\th=\{{ x}\in\R^n \mid \psi({x}):=\max\limits_{1\leq \mu\leq q}\{\langle {b}^\mu, {x}\rangle-d_\mu\}\leq0\}
\end{equation*}
for some $\{({b}^\mu, d_\mu)\}_{\mu=1}^u\subset (\R^n\times\R)\setminus \{(0,0)\}$ with a positive integer $u$. It is important 
to emphasize here that we demand each pair $({b}^\mu, d_\mu)\neq (0,0)$, since such a case can be simply dismissed with no harm in the representation of $\dom \th$. The case  $\th_1=0$ can be covered by letting $p=1$ and $(a^1,c_1)=(0,0)\in \R^n\times \R$. Similarly, $\th_2=0$ can be covered by setting $u=1$ and $(b^1,d_1)=(0,1)\in \R^n\times \R$.
For each $\nu\in\{1,\ldots,p\}$, define $\mathcal{D}_\nu=\{{x}\in\R^n\mid \langle {a}^j,{x}\rangle-c_j\leq\langle {a}^\nu,{x}\rangle-c_\nu, \  j=1,\dots,p\}$. 
Take  ${x}\in\dom \th$ and define the active index sets  
\begin{equation*}\label{eq:iota1}
\iota_1({x}):=\{1\leq \nu\leq p \mid {x}\in\mathcal{D}_\nu\}\quad \mbox{and}\quad \iota_2({x}):=\{1\leq \mu\leq q \mid \langle {b}^\mu,{x}\rangle-d_\mu=0\}.
\end{equation*}
Using these index sets, one can express the subdifferential of $\th_1$
and $\th_2$ from \eqref{eq:def_phi12}, respectively,  by 
\begin{equation}
\partial \th_1({x})={\rm conv}\,\big\{{a}^\nu|\;\nu\in\iota_1({x})\big\}\quad {\rm and} \quad \sub \th_2(x)={N}_{\ss\dom \th}({x})={\rm cone}\,\{{b}^\mu|\;\mu\in\iota_2({x})\}.\label{partialdiffphi}
\end{equation}

The following lemma gives a simple representation of the critical cone of $\th$ for an important case that will be used extensively in the proofs of Theorem~\ref{prop:minquadbdcha} and Proposition~\ref{quad_minimal}.  
\begin{Lemma} \label{crit_gqf} Assume that $\th:\R^n\to\oR$ is a polyhedral function and $y\in \ri \sub \th(x)$. Then the critical cone of $\th$ at $x$
for $y$ has a representation of the form 
$$
{\cal C}_{\theta}(x,y)=\left\{w\in\R^n\Bigg | \begin{array}{l}
\langle d,{a}^{\nu'}-{a}^\nu\rangle=0\;\;\;\mbox{for all}\;\;\;\nu,\nu'\in{\iota}_1(x),\\
\langle w,{b}^\mu\rangle=0\;\;\;\mbox{for all}\;\;\;\mu\in{\iota}_2(x)\\
\end{array}\right\}.
$$
    
\end{Lemma}
\begin{proof} Since $y\in \ri \sub \th(x)$ and since ${\cal C}_{\theta}(x,y)=N_{\sub \th(x)}(y)$, it is not hard to see that ${\cal C}_{\theta}(x,y)=\para \{\sub \th(x)\}^\perp$, where $\para\{\sub \th(x)\}$ stands for the subspace parallel to the affine hull of $\sub \th(x)$. Since 
$$
\para \{\sub \th(x)\}=\span\big\{{a}^{\nu'}-{a}^\nu|\; \nu,\nu'\in{\iota}_1(x)\big\}+\span\big\{{b}^\mu|\; \mu\in{\iota}_2(x)\big\},
$$
where $\span\big\{{b}^\mu|\; \mu\in{\iota}_2(x)\big\}$ signifies the subspace generated by the set $\big\{{b}^\mu|\; \mu\in{\iota}_2(x)\big\}$, one can obtain the claimed formula for the critical cone ${\cal C}_{\theta}(x,y)$ by applying the Farkas Lemma from \cite[Lemma~6.45]{rw}. 
\end{proof}
For  the polyhedral function $\th$   from \eqref{eq:def_phi12}, suppose $(x,{y})\in \gph\partial \th$. We assume in what follows that 
$y\notin \sub\th_1(x)$ and $y\notin\sub \th_2(x)$ and  
 define the index set of positive coefficients in a representation of a subgradient by 
\begin{equation}\label{eq:eta}
\eta({x},y):=
\left\{(\nu,\mu)\in \iota_1({x})\times\iota_2({x})\Bigg | \begin{array}{l}
 \displaystyle\sum_{\nu\in \iota_1({x})}u_\nu{a}^\nu+\sum_{\mu\in \iota_2({x})}v_\mu{b}^\mu={y},\\
\displaystyle\sum_{\nu\in \iota_1({x})}u_\nu=1,\ 0<u_\nu\leq1,\  v_\mu>0\\
\end{array}\right\}.
\end{equation}
When either $y\in\sub \th_1(x)$ or $y\in\sub \th_2(x)$ holds, our approach can be significantly simplified by dropping the indices related to the part of $\th$ that is zero and so we will  proceed to analyze the main case that both $\th_1$ and $\th_2$ are present in our analysis of the quadratic bundles of the spectral function $g$ in
\eqref{spec}.
The index set in \eqref{eq:eta}  plays a major role in various proofs in this section and allows us to provide a simple description of the critical cone of spectral functions. We begin with the following result, which is a direct consequence of \cite[Proposition 1]{cuiding}.
\begin{Proposition}\label{cor:inv} Assume that the polyhedral function $\th$ has the representation in \eqref{eq:def_phi12} and that $(x,{y})\in \gph\partial \th$. Then, the following two statements hold.
\begin{itemize}[noitemsep]
\item [\rm{(a)}] For any $\nu\in \iota_1(x)$, $\mu\in\iota_2(x)$ and  $Q\in\mathcal{P}^n_{x}$ (i.e., $Q\in\mathcal{P}^n$ and $Qx=x$),  there exist  $\nu'\in  \iota_1(x)$ and $\mu'\in  \iota_2(x)$ such that ${a}^{\nu'}=Q{a}^\nu$ and ${b}^{\mu'}=Q{b}^\mu$, respectively. 
\item [\rm{(b)}] 
	 For any $(\nu,\mu)\in \eta(x,y)$ and $Q\in\mathcal{P}^n_{x}\cap\mathcal{P}^n_{y}$,   there exist  $(\nu',\mu')\in  \eta(x, y)$ such that ${a}^{\nu'}=Q\,{a}^\nu$ and ${b}^{\mu'}=Q\,{b}^\mu$, respectively. 
    \end{itemize}
\end{Proposition}

\begin{Remark}\label{part}{\rm In this section, we will be using multiple partitions of eigenvalues corresponding to a pair $(X,Y)\in \gph \sub g$. To facilitate the presentation and make it easier for the readers to follow our proofs, we briefly list them below. Figure~\ref{fig1} shows an example of all these different partitions for the $\al^l\times \al^l$ block. 
 \begin{figure}[t]
\includegraphics[width=8cm]{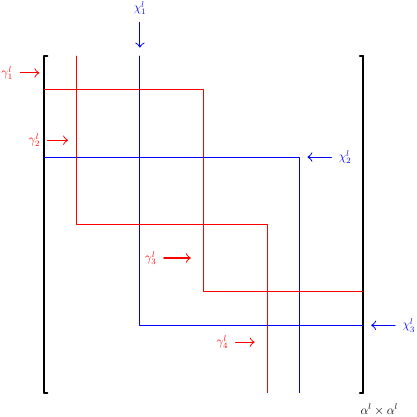}
\centering
\caption{Different partitions of eigenvalues}\label{fig1}
\end{figure}
\begin{itemize}[noitemsep]
\item [\bf{(P1)}] The index sets $\al^l$, $l\in \{1,\ldots,r\}$, from \eqref{eq:def-alpha}, where $r$ is the number of distinct eigenvalues of $X$.

\item [\bf{(P2)}]  Suppose that $\{(X^k,Y^k)\b\subset \gph \sub g$ is a sequence, converging to $(X,Y)$. For any $l\in \{1,\ldots,r\}$, define the index sets $\{\chi_t^l\}_{t=1}^{z^l}$ with $z^l\in \N$ to further partition the set $\alpha^l$ based on $\lambda(X^k)$ as
\begin{equation}\label{eq:partonxk}
\begin{cases}
\lambda_i(X^k)=\lambda_j(X^k),&{\rm if}\; i,j\in\chi_t^l\;{\rm and}\;t\in\{1,\dots,z^l\},\\
\lambda_i(X^k)>\lambda_j(X^k),&{\rm if}\, i\in\chi_t^l, j\in\chi_{t'}^l\,{\rm and}\,t,t'\in\{1,\dots,z^l\}\,{\rm with}\,t<t'.
\end{cases}
\end{equation}
Note that $\{\chi_t^l\}_{t=1}^{z^l}$ depends on $k$. By taking a subsequence if necessary, we can assume that the index sets $\chi_t^l$ remain constant for all $k$. This is because $\cup_{l=1}^r\cup_{t=1}^{z^l}\chi_t^l$ can be regarded as a partition of $\{1,\dots,n\}$ and the total number of different partitions of $\{1,\dots,n\}$ is finite. Thus, we can take a subsequnce of $\{(X^k,Y^k)\b$ to ensure that the partition  $\{\chi_t^l\}$ for each $(X^k,Y^k)$ in the subsequence remains the same.

\item [\bf{(P3)}]For any $l\in \{1,\ldots,r\}$ and  $\alpha^l$ from  \eqref{eq:def-alpha}, we define the index sets $\{\gamma_s^l\}_{s=1}^{u^l}$  with $u^l\in \N$  to further partition the index set $\alpha^l$  based on $\lambda(Y)$ as
\begin{equation}\label{eq:partonY}
\begin{cases}
\lambda_i(Y)=\lambda_j(Y),&{\rm if}\; i,j\in\gamma_s^l\;{\rm and}\;s\in\{1,\dots,u^l\},\\
\lambda_i(Y)>\lambda_j(Y),&{\rm if}\, i\in\gamma_s^l, j\in\gamma_{s'}^l\,{\rm and}\,s,s'\in\{1,\dots,u^l\}\,{\rm with}\,s<s'.
\end{cases}
\end{equation}
 \end{itemize}
}    
\end{Remark}

The following lemma provides necessary technical tools for the characterization of critical cone and  the caluculation of quadratic bundle of spectral functions. To present it, suppose that $({X}, {Y}) \in \operatorname{gph} \partial g$ and $X$ has the eigenvalue decomposition \eqref{eq:eig-decomp1} with $P\in {\bf O}^n({X}) \cap {\bf O}^n({Y})$. 
Denote 
\begin{equation}\label{eq:Eldiff}
{\cal E}^l:=\left\{s\in\{1,\dots,u^l\}\Bigg | \begin{array}{l}
\exists\, i, j \in \gamma_s^l \text { such that }\left({a}^\nu\right)_i \neq\left({a}^{\nu}\right)_j {\rm or} \left({b}^{\mu}\right)_i \neq\left({b}^\mu\right)_j \\[3pt]
\text {for some } (\nu,\mu)\in \eta(\lambda({X}), \lambda({Y}))
\end{array}  \right\}
\end{equation}

\begin{Lemma}\label{lem:wIdforg}
Suppose that $({X}, {Y}) \in \operatorname{gph} \partial g$ and $X$ has the eigenvalue decomposition \eqref{eq:eig-decomp1} with $P\in {\bf O}^n({X}) \cap {\bf O}^n({Y})$.   For any $w\in\R^n$, satisfying  for any $(\nu,\mu)$ and $(\nu',\mu')\in\eta(\lambda({X}), \lambda({Y}))$,
$$
\langle w,{a}^\nu-{a}^{\nu'}\rangle=0 \quad {\rm and} \quad \langle w,{b}^\mu\rangle=0,
$$
there exists $\rho_s^l$ such that for all $s\in{\cal E}^l$, we have
\begin{equation*}
w_{\gamma_s^l}=\rho^l_s{\bf 1}_{|\gamma_s^l|},
\end{equation*} 
where ${\bf 1}$ denotes the vector whose components are all 1.
\end{Lemma}
\begin{proof}
While the proof is similar to that of \cite[Proposition 4 (iii)]{cuiding}, we provide a proof for the readers' convenience.  We proceed with the  proof of  the case where $y\notin\partial\theta_1(\lambda(X))$ and $y\notin\partial\theta_2(\lambda(X))$. The other cases $y\in\partial\theta_1(\lambda(X))$ or $y\in\partial\theta_2(\lambda(X))$ can be argued similarly. 
Recall \eqref{eq:partonY}.  For each $l\in\{1,\ldots,r\}$, if $s\in {\cal E}^l$, suppose the first case that there exist $i,j\in \gamma^l_{s}$ such that $({a}^{t})_i\neq({a}^{t})_j$ for some $(t,k)\in{\eta}(\lambda(X),\lambda(Y))$. 
Consider the $n\times n$ permutation matrix $Q^{i,j}$  satisfying
	\[
	(Q^{i,j}{a}^{t})_z=\left\{\begin{array}{ll}
	({a}^{t})_j & \mbox{if $z=i$,} \\ [3pt]
	({a}^{t})_i & \mbox{if $z=j$,} \\ [3pt]
	({a}^{t})_z & \mbox{otherwise,}
	\end{array} \right. \quad z=1,\ldots,n.
	\]  
        
	Since $\lambda_i({X})=\lambda_j({X})$ and $\lambda_i({Y})=\lambda_j(Y)$, it is clear that $Q^{i,j}\lambda({X})=\lambda({X})$ and $Q^{i,j}\lambda({Y})=\lambda({Y})$. 
    It then follows from the condition and Proposition \ref{cor:inv} that there exists $(t',k')\in \eta(\lambda(X),\lambda(Y))$ such that ${a}^{t'}=Q^{i,j}{a}^{t}$. Therefore, we have  
	\[
	\langle w,{a}^{t}-{a}^{t'}\rangle=\langle w,{a}^{t}-{a}^{t'}\rangle=(w_i-w_j)(({a}^{t})_i-({a}^{t})_j)=0,
	\] 
	which implies that $
	w_i=w_j.
	$
	For any $i'\in \gamma^l_{s}$ with $i'\neq i$ and $i'\neq j$, if $({a}^{t})_{i'}\neq({a}^{t})_i$, by replacing $i$ by $i'$ and $j$ by $i$ in the above argument, we obtain   
	$$
	w_{i'}=w_i=w_j;
	$$
	otherwise if $({a}^{t})_{i'}=({a}^{t})_i$,  then by replacing $i$ by $i'$ in the above argument, we can also obtain the above equality.
	Consequently, we know that for any $s\in {\cal E}^l$, there exists some $\rho^l_{s}\in\R$ such that for any $i\in \gamma^l_{s}$,
	\[
	w_i=\rho^l_{s}. 
	\]
 The other case where $({b}^{k})_i\neq({b}^{k})_j$ can be obtained similarly. Thus, the desired result has been verified. 
\end{proof}

It is well-known (see e.g.,  \cite[Theorem 7]{Lancaster} and \cite[Proposition 1.4]{Torki}) that the eigenvalue function $\lambda$ is directional differentiable everywhere and for any $X\in\S^n$, the directional derivative $\lambda'(X,H)$ at $X$ along the direction $H\in\S^n$ is given by
\begin{equation}\label{eq:egdd}\lambda'(X,H)=\big(\lambda(P^{\top}_{\alpha^1}HP_{\alpha^1}),\dots,\lambda(P^{\top}_{\alpha^r}HP_{\alpha^r})\big)^{\top}, \end{equation}
where $P\in{\bf O}^n(X)$ and the index sets $\alpha^l$, $l=1,\ldots,r$ are given by \eqref{eq:def-alpha}. 
The following lemma provides a complete characterization of the critical cone of $g$, which is inspired by  \cite[Propositions 4 and 8]{cuiding}, where a similar result was achieved by splitting the polyhedral function $\th$ into two parts. Below, we show that it can be done without such a decomposition. 

\begin{Lemma}\label{lemma:characofcc}
Suppose that $({X}, {Y}) \in \operatorname{gph} \partial g$ and $X$ has the eigenvalue decomposition \eqref{eq:eig-decomp1} with $P\in {\bf O}^n({X}) \cap {\bf O}^n({Y})$. 
 If $H \in \mathcal{C}_g\left({X},{Y}\right)$, then the following properties (a)-(c) hold.
 \begin{itemize}
\item[\rm(a)] For each $l \in\{1, \ldots, r\}$, $P_{\alpha^l}^{\top} H P_{\alpha^l}$ has the following block diagonal structure:
$$
P_{\alpha^l}^{\top} H P_{\alpha^l}=\operatorname{Diag}\big((P^{\top} H P)_{\gamma_1^l \gamma_1^l}, \ldots,(P^{\top} H P)_{\gamma_{u^l}^l \gamma_{u^l}^l}\big).
$$
\item[\rm(b)]  For any $(\nu,\mu)$ and $(\nu',\mu')\in\eta(\lambda({X}), \lambda({Y}))$, $$\left\langle\lambda'({X} , H), {a}^\nu\right\rangle=\left\langle\lambda'({X} , H), {a}^{\nu'}\right\rangle=\max _{\nu \in \iota_1(\lambda({X}))}\left\langle\lambda'({X} , H), {a}^\nu\right\rangle;$$
 $$\left\langle\lambda'({X} , H), {b}^\mu\right\rangle=\left\langle\lambda'({X} , H), {b}^{\mu'}\right\rangle=\max _{\mu \in \iota_2(\lambda({X}))}\left\langle\lambda'({X} , H), {b}^\mu\right\rangle=0. $$
\item[\rm(c)] For each $l \in\{1, \ldots, r\}$ and $k \in \mathcal{E}^l$, there exists a scalar $\rho_k^l$ such that $\left(P^{\top} H P\right)_{\gamma_k^l \gamma_k^l}=\rho_k^l I_{\left|\gamma_k^l\right|}$. 
\end{itemize}
Consequently, 
we can conclude that  $H \in \mathcal{C}_g\left({X},{Y}\right)$ if and only if for any $(\nu,\mu)$ and $(\nu',\mu')\in\eta(\lambda({X}), \lambda({Y}))$,
$$
\langle\operatorname{diag}(P^{\top} H P), {a}^\nu\rangle=\langle\operatorname{diag}(P^{\top} H P), {a}^{\nu'}\rangle=\max _{\nu \in \iota_1(\lambda({X}))}\langle\lambda'({X} ; H), {a}^\nu\rangle
$$
and
$$\langle\operatorname{diag}(P^{\top} H P), {b}^\mu\rangle=\langle\operatorname{diag}(P^{\top} H P), {b}^{\mu'}\rangle=\max _{\mu \in \iota_2(\lambda({X}))}\langle\lambda'({X} ; H), {b}^\mu\rangle=0. $$

\end{Lemma}
\begin{proof}
The proofs of (a)-(b) and the claimed characterization of $\mathcal{C}_g\left({X},{Y}\right)$ are similar to that of \cite[Proposition 4]{cuiding} combined with \cite[Proposition 3.2]{MSarabi18} and \cite[Proposition 5.4]{MSarabi23}. Part (c) follows from Lemma \ref{lem:wIdforg}. 
\end{proof}

In what follows, for any vector $x\in \R^n$, denote by $x^{\dn}$ the vector with the same components permuted in nonincreasing order. 
Set $\R^n_{\dn}:=\{x^{\dn}|\; x\in \R^n\}$.

\begin{Remark}\label{crit_charc}{\rm
Although we often will use the characterization of the critical cone
of spectral functions of Lemma~\ref{lemma:characofcc}, it is important to remind the readers of another characterization, obtained recently in \cite[Proposition~5.4]{MSarabi23}. To do so, suppose $(X,Y)\in \gph \sub g$
with $g$ taken from \eqref{spec} and $P\in {\bf O}^n({X}) \cap {\bf O}^n({Y})$. Then $H\in {\cal C}_g(X,Y)$ if and only if the following properties are satisfied:
\begin{itemize}[noitemsep]
\item [\rm{(i)}]  $\lm'(X;H)\in {\cal C}_\th(\lm(X),\lm(Y))$;
 
\item [\rm{(ii)}] for any $l\in \{1,\ldots,r\}$, the matrices $\Lm(Y)_{\al^l\al^l}$ and $P_{\al^l}^\top HP_{\al^l}$ have a simultaneous ordered spectral decomposition, where the index set $\al^l$ is taken from \eqref{eq:def-alpha}. 
 \end{itemize}
 Observe that for any $l\in \{1,\ldots,r\}$, one can decompose the vector $\lm(P_{\al^l}^\top  H P_{\al^l})\in \R^{|\al^l|}_{\dn}$ as
\begin{equation}\label{crit_equi3}
\lm(P_{\al^l}^\top  H P_{\al^l})=\Big(\lm(P_{\al^l}^\top  H P_{\al^l})_{\gamma_1^l}, \ldots, \lm(P_{\al^l}^\top  H P_{\al^l})_{\gamma_{u^l}^l}\Big),    
\end{equation}
where $\lm(P_{\al^l}^\top  H P_{\al^l})_{\gamma_j^l}\in \R^{|\gamma_j^l|}_{\dn}$ for any $j\in \{1,\ldots,u^l\}$, and   
where $u^l$ and the index sets $\gamma_j^l$ are taken from \eqref{eq:partonY}. Bearing this in mind, 
note that condition (ii) above  is equivalent to the following statement:
\begin{itemize}[noitemsep] 
\item [\rm{(ii)'}] For any $l\in \{1,\ldots,r\}$, 
$P_{\alpha^l}^{\top} H P_{\alpha^l}$ has the  block diagonal representation 
$$
P_{\alpha^l}^{\top} H P_{\alpha^l}=\operatorname{Diag}\big((P^{\top} H P)_{\gamma_1^l \gamma_1^l}, \ldots,(P^{\top} H P)_{\gamma_{u^l}^l \gamma_{u^l}^l}\big)
$$
with  $(P^{\top} H P)_{\gamma_j^l \gamma_j^l}= Q^l_j\operatorname{Diag}\big(\lm(P_{\al^l}^\top  H P_{\al^l})_{\gamma_j^l}\big)  (Q^l_j)^\top$ for some orthogonal matrix  $Q^l_j\in  {\bf O}^{|\gamma_j^l|}$ whenever $j\in \{1,\ldots,u^l\}$. Moreover, 
\begin{equation}\label{crit_equi2}
    \lm_i(P_{\al^l}^\top  H P_{\al^l})\ge \lm_j(P_{\al^l}^\top  H P_{\al^l}), \;\;\mbox{if}\; i\in \gamma_s^l, j\in \gamma_{s'}^l\,{\rm and}\,s,s'\in\{1,\dots,u^l\}\,{\rm with}\,s<s'.
\end{equation}
 \end{itemize}

To prove this claim, suppose first that condition (ii) holds. 
 Given $l\in \{1,\ldots,r\}$, the matrices $\Lm(Y)_{\al^l\al^l}$ and $P_{\al^l}^\top HP_{\al^l}$ have a simultaneous ordered spectral decomposition. Thus,  
we find an orthogonal matrix $  Q^l\in  {\bf O}^{|\al^l|}$ such that 
\begin{equation}\label{spcr4}
\Lm(Y)_{\al^l\al^l}=  Q^l\Lm(Y)_{\al^l\al^l}  (Q^l)^\top\quad \mbox{and}\quad  P_{\al^l}^{\top} H P_{\al^l}=  Q^l\Lambda(P_{\al^l}^\top  H P_{\al^l})  (Q^l)^\top.
\end{equation}
It is not hard to see from the first equality above   that $Q^l$ has a block diagonal representation as 
$$
Q^l=\operatorname{Diag}(Q^l_1,\ldots, Q^l_{u^l})
\quad \mbox{with}\;\; Q^l_j\in  {\bf O}^{|\gamma_j^l|}, \;\; j\in \{1,\ldots,u^l\},
$$
where $u^l$ and the index sets $\gamma_j^l$ are taken from \eqref{eq:partonY}.
This, coupled with the second equality in \eqref{spcr4}, implies  that $P_{\al^l}^{\top} H P_{\al^l}$ has the claimed block diagonal representation. Observe also that \eqref{crit_equi2} results from the fact that $\Lambda(P_{\al^l}^\top  H P_{\al^l})=\operatorname{Diag}\big(\lm(P_{\al^l}^\top  H P_{\al^l})\big)$.

Assume now that (ii)' is satisfied. Given $l\in \{1,\ldots,r\}$ and the matrices $Q^l_j\in  {\bf O}^{|\gamma_j^l|}$ with $j\in \{1,\ldots,u^l\}$, set 
$Q^l:=\operatorname{Diag}(Q^l_1,\ldots, Q^l_{u^l})$.  Thus, we have 
$$
\Lm(Y)_{\al^l\al^l}=  Q^l\Lm(Y)_{\al^l\al^l}  (Q^l)^\top\quad \mbox{and}\quad  P_{\al^l}^{\top} H P_{\al^l}=  Q^l\operatorname{Diag}\Big(\lm(P_{\al^l}^\top  H P_{\al^l})_{\gamma_1^l}, \ldots, \lm(P_{\al^l}^\top  H P_{\al^l})_{\gamma_{u^l}^l}\Big)  (Q^l)^\top,
$$
where the first equality results from the fact that $\lambda_i(Y)=\lambda_j(Y)$,  if $i,j\in\gamma_s^l$  and $s\in\{1,\dots,u^l\}$, and where the second one follows from the assumed block diagonal representation of  $P_{\al^l}^\top HP_{\al^l}$. According to \eqref{crit_equi3}, the second equality above is the same as the second equality in \eqref{spcr4}. Therefore, we can conclude from \eqref{crit_equi2} that the matrices $\Lm(Y)_{\al^l\al^l}$ and $P_{\al^l}^\top HP_{\al^l}$ have a simultaneous ordered spectral decomposition, which proves (ii). 

}
\end{Remark}

\begin{Remark}\label{remark:invariant-y1-y2}{\rm
    By using \cite[Proposition~4.4]{MSarabi16} and \cite[Proposition~3.2 and page 612]{MSarabi18}, it can be checked directly that the property in Lemma \ref{lemma:characofcc}(b) is equivalent to that of Remark \ref{crit_charc}(i).  
To see this, denote $d=\lambda'(X,H)$. For any $d$, satisfying the property in Lemma \ref{lemma:characofcc}(b), $d\in{\cal C}_{\theta}(\lambda(X),\lambda(Y))$ holds trivially. 
Conversely, for all $d\in{\cal C}_{\theta}(\lambda(X),\lambda(Y))$, it follows from $\lm(Y)\in \sub\th (\lm(X))$ that for any representation $\lambda(Y)=y_1+y_2$ with $y_1\in\partial\theta_1(\lambda(X))$ and $y_2\in\partial\theta_2(\lambda(X))$, one has 
\begin{equation}\label{eq:C-theta-temp}
{\cal C}_{\theta}(\lambda(X),\lambda(Y))=
\begin{aligned}
\left\{d\Bigg |\left.
\begin{array}{l}
\langle d,{a}^{\nu'}-{a}^\nu\rangle=0, \quad \nu, \nu'\in\eta_1(\lambda(X),y_1) \\[3pt]
\langle d,{b}^\mu\rangle=0, \quad\quad\quad\;\mu\in\eta_2(\lambda(X),y_2)\\[3pt]
    \langle d,{a}^{\nu'}-{a}^\nu\rangle\leq0, \quad \nu'\in\eta_1(\lambda(X),y_1), \nu\in\iota_1(\lambda(X))\backslash\eta_1(\lambda(X),y_1)\\[3pt]
     \langle d,{b}^{\mu}\rangle\leq0, \quad\quad\quad\;\mu\in\iota_2(\lambda(X))\backslash\in\eta_2(\lambda(X),y_2) 
\end{array}\right.\right\}, 
\end{aligned}
\end{equation}
where $\eta_1(\lambda(X),y_1)=\big\{\nu\in\iota_1({x})\mid \sum_{\nu\in \iota_1({x})}u_\nu{a}^\nu={y}, \sum_{\nu\in \iota_1({x})}u_\nu=1,\ 0<u_\nu\leq1\big\}$ and $\eta_2(\lambda(X),y_2)=\big\{\mu\in\iota_2(x)\mid  \sum_{\mu\in \iota_2({x})}v_\mu{b}^\mu={y}_1, v_\mu>0\big\}$. 
Note that the above representation of ${\cal C}_{\theta}(\lambda(X),\lambda(Y))$  is independent from the decomposition of $\lambda(Y)=y_1+y_2$ according to \cite[Proposition~4.4]{MSarabi16} and \cite[Proposition~3.2]{MSarabi18} (see also \cite[page 612]{MSarabi18}). 
Denote \begin{equation}\label{eqM}{\cal M}=\{(y_1,y_2)\mid\lambda(Y)=y_1+y_2,\; y_1\in\partial\theta_1(\lambda(X)),\; y_2\in\partial\theta_2(\lambda(X))\}.\end{equation}  
Therefore, the property in Lemma \ref{lemma:characofcc}(b) holds if we show $d$ satisfies the conditions 
\begin{equation}\label{eq:Ccon}
\left\{\begin{array}{l}
\begin{array}{l}
\langle d,{a}^{\nu'}-{a}^\nu\rangle=0, \\[3pt]
\langle d,{b}^\mu\rangle=0,
\end{array}
 \ \ (\nu,\mu), (\nu',\mu')\in\eta(\lambda(X),\lambda(Y)),\\ [3pt]
\begin{array}{l}
    \langle d,{a}^{\nu'}-{a}^\nu\rangle\leq0, \\[3pt]
     \langle d,{b}^{\mu}\rangle\leq0, 
\end{array}
\ (\nu,\mu)\in\iota_1(\lambda(X))\times\iota_2(\lambda(X))\backslash\eta(\lambda(X),\lambda(Y)).\end{array}\right.
\end{equation}
Indeed, we know from \eqref{eq:C-theta-temp} that for all $(\nu,\mu), (\nu',\mu')\in\iota_1(\lambda(X))\times\iota_2(\lambda(X))$, 
$$\begin{array}{l}
\langle d,{a}^{{\nu}}-{a}^{\nu'}\rangle\leq0\quad\mbox{and}\quad
\langle d,{b}^{\mu}\rangle\leq0. 
\end{array}$$
It can be checked directly that 
$$\eta(\lambda(X),\lambda(Y))=\bigcup_{(y_1,y_2)\in{\cal M}}\eta_1(\lambda(X),y_1)\times \eta_2(\lambda(X),y_2).$$ 
Let $(\nu,\mu), (\nu',\mu')\in\eta(\lambda(X),\lambda(Y))$ be arbitrarily given. Then, there exists $(y_1,y_2)$ and $(y_1',y_2')\in{\cal M}$ such that $(\nu,\mu)\in \eta_1(\lambda(X),y_1)\times \eta_2(\lambda(X),y_2)$ and $(\nu',\mu')\in \eta_1(\lambda(X),y_1')\times \eta_2(\lambda(X),y_2')$. 
Let $\widetilde{y}_1=\ ({y_1+y_1'})/{2}$ and $\widetilde{y}_2=({y_2+y_2'})/{2}$. It is clear that $(\widetilde{y}_1,\widetilde{y}_2)\in{\cal M}$. 
We deduce from the definition of $\widetilde{y}_1$ that $$\eta_1(\lambda(X),\widetilde{y}_1)\supset\eta_1(\lambda(X),y_1)\cup\eta_1(\lambda(X),y_1').$$ It then follows from \cite[Proposition~4.4]{MSarabi16} and \eqref{eq:C-theta-temp} that $\langle d,{a}^{\nu'}-{a}^\nu\rangle=0$ for all $\nu$, $\nu'\in\eta_1(\lambda(X),\widetilde{y}_1)$. This also implies $\langle d,{a}^{\nu'}-{a}^\nu\rangle=0$ for all $\nu, \nu'\in\eta_1(\lambda(X),y_1)\cup\eta_1(\lambda(X),y_1')$. Moreover, it is easy to see  that for all $\mu\in\eta_2(\lambda(X),y_2)\cup\eta_2(\lambda(X),y_2')$, $\langle d, b^{\mu}\rangle=0$. This shows that  \eqref{eq:Ccon} holds.

}
\end{Remark}

The following lemma portraits the explicit form of the affine of the critical cone of $g$, which can be obtained directly from Lemma \ref{lemma:characofcc}. 

\begin{Lemma}\label{lemma:charofaffcc}
Suppose that $({X}, {Y}) \in \operatorname{gph} \partial g$ and $X$ has the eigenvalue decomposition \eqref{eq:eig-decomp1} with $P\in {\bf O}^n({X}) \cap {\bf O}^n({Y})$. 
Then we have $H \in {\rm aff}\,\mathcal{C}_g\left({X},{Y}\right)$ if and only if the conditions (a) and (c) of Lemma \ref{lemma:characofcc} hold and for any $(\nu,\mu)$ and $(\nu',\mu')\in\eta(\lambda({X}), \lambda({Y}))$, 
\begin{equation}\label{eq:affcha=01}
\langle\operatorname{diag}(P^{\top} H P), {a}^\nu\rangle=\langle\operatorname{diag}(P^{\top} H P), {a}^{\nu'}\rangle,
\end{equation}
and 
\begin{equation}\label{eq:affcha=02}\langle\operatorname{diag}(P^{\top} H P), {b}^\mu\rangle=\langle\operatorname{diag}(P^{\top} H P), {b}^{\mu'}\rangle=0.\end{equation}
\end{Lemma}
\begin{proof}
Assume first that  $H\in{\rm aff}\,\mathcal{C}_g\left({X},{Y}\right)$.  This is equivalent to saying that there exist $H_1, H_2\in\mathcal{C}_g\left({X},{Y}\right)$ and a constant $\sigma\in\R$ such that $H=\sigma H_1+(1-\sigma)H_2$.  
It follows from Lemma \ref{lemma:characofcc} that \eqref{eq:affcha=01}
and \eqref{eq:affcha=02} hold. 

To prove the opposite implication,  we can argue as the proofs of \cite[the first equation on page 10]{cuiding} and  Lemma \ref{lemma:characofcc}(c) to ensure that for any $(\nu,\mu)\in\eta(\lambda({X}), \lambda({Y}))$,  we have $\langle\operatorname{diag}(P^{\top} H P), {a}^\nu\rangle=\langle\lambda'(X,H), {a}^{\nu}\rangle$ and $\langle\operatorname{diag}(P^{\top} H P), {b}^\mu\rangle=\langle\lambda'(X,H), {b}^\mu\rangle$. Together with \eqref{eq:affcha=01} and \eqref{eq:affcha=02}, we get $\lambda'(X,H)\in\mbox{aff}\,{\cal C}_{\theta}(\lambda(X),\lambda(Y))$. 
Hence, there exist $h_1, h_2\in{\cal C}_\theta(\lambda(X),\lambda(Y))$ and  $\sigma\in\R$ such that  $\lambda'(X,H)=\sigma h_1+(1-\sigma)h_2$. For each $l \in\{1, \ldots, r\}$ and $k \in \mathcal{E}^l$, we conclude  from Lemma \ref{lem:wIdforg} that $(h_1)_{\gamma^l_k}=\widehat{\rho}_k^l{\bf 1}$ and $(h_2)_{\gamma^l_k}=\widetilde{\rho}_k^l{\bf 1}$ with $\sigma\widehat{\rho}_k^l+(1-\sigma)\widetilde{\rho}_k^l=\rho_k^l$. Pick $H_i$, $i=1,2$ such that $(P^{\top}H_iP)_{\gamma_k^l\gamma_k^l}=U_{kl}\mbox{Diag}((h_i)_{\gamma^l_k})U_{kl}^{\top}$ with $U_{kl}\in{\bf O}^{|\gamma_k^l|}((P^{\top}HP)_{\gamma_k^l\gamma_k^l})$, and for all $w\in\gamma_k^l, j\in\gamma_{k'}^{l'}$ with $\gamma_k^l\neq\gamma_{k'}^{l'}$, 
\begin{equation}\label{eq:od0}
(P^{\top}H_1P)_{wj}=\sigma (P^{\top}HP)_{wj} \quad \mbox{and}\quad (P^{\top}H_2P)_{wj}=(1-\sigma)(P^{\top}HP)_{wj},
\end{equation}
where the index sets $\gamma^l_k$ are taken from Remark~\ref{part}(P3). We claim that $\lambda'(X,H_i)\in{\cal C}_{\theta}(\lambda(X),\lambda(Y))$, $i=1,2$. To justify it, one can  check directly that there exists a block permutation matrix $Q=\mbox{Diag}\,(Q_1,\dots,Q_r)$ with $Q_l\in{\cal P}^{|\alpha^l|}$ such that $\lambda'(X,H_i)=Qh_i$, $i=1,2$, where the index sets $\al^l$ are taken from Remark~\ref{part}(P1). We infer from $h_i\in{\cal C}_{\theta}(\lambda(X),\lambda(Y))$ that 
\begin{equation}\label{eq:cc}
  \begin{aligned}  \theta'(\lambda(X),h_i)&=\sup_{z\in\partial\theta(\lambda(X))}\langle z,h_i\rangle=\langle h_i,\lambda(Y)\rangle\leq\langle \lambda'(X,H_i),\lambda(Y)\rangle\leq\sup_{z\in\partial\theta(\lambda(X))}\langle z,\lambda'(X,H_i)\rangle\\
&=\sup_{z\in\partial\theta(\lambda(X))}\langle Qz,h_i\rangle=\sup_{Qz\in\partial\theta(\lambda(X))}\langle Qz,h_i\rangle=\sup_{z\in\partial\theta(\lambda(X))}\langle z,h_i\rangle, \quad i=1,2,
  \end{aligned}
\end{equation}
where the penultimate equality comes from the fact that for all $z\in\partial\theta(\lambda(X))$, we have $Qz\in\partial\theta(\lambda(X)) $, which is a consequence of $\th$ being symmetric. 
Thus, the inequalities in \eqref{eq:cc}   become equalities, which implies 
$$\langle \lambda'(X,H_i),\lambda(Y)\rangle=\sup_{z\in\partial\theta(\lambda(X))}\langle z,\lambda'(X,H_i)\rangle=\theta'(\lambda(X),\lambda'(X,H_i)),\quad i=1,2.
$$
This proves our claim above, namely $\lambda'(X,H_i)\in{\cal C}_{\theta}(\lambda(X),\lambda(Y))$, $i=1,2$.  Since $H$ satisfies the condition in Lemma \ref{lemma:characofcc}(a), it  follows from \eqref{eq:od0}  that $H_1$ and $H_2$ also satisfy the latter condition.  
Therefore, we infer from  \eqref{eq:cc} and \eqref{eq:egdd}  that 
$$\begin{aligned}
&\sum_{l=1}^r\langle (P^{\top}H_iP)_{\alpha^l\alpha^l},\Lambda(Y)_{\alpha^l\alpha^l}\rangle=\sum_{l=1}^r\sum_{k=1}^{u^l}\langle (P^{\top}H_iP)_{\gamma^l_k\gamma^l_k},\Lambda(Y)_{\gamma^l_k\gamma^l_k}\rangle=\sum_{l=1}^r\sum_{k=1}^{u^l}\langle (h_i)_{\gamma^l_k},(\lambda(Y))_{\gamma^l_k}\rangle\\
&=\langle h_i,\lambda(Y)\rangle=\langle \lambda'(X,H_i),\lambda(Y)\rangle=\sum_{l=1}^r\langle \Lambda((P^{\top}H_iP)_{\alpha^l\alpha^l}),\Lambda(Y)_{\alpha^l\alpha^l}\rangle,\quad i=1,2.\end{aligned} $$
It follows from Fan's inequality \cite{Fan49} that for any $l=1,\dots,r$,  the matrices $\Lm(Y)_{\al^l\al^l}$ and $P_{\al^l}^\top H_iP_{\al^l}$, $i=1, 2$, have a simultaneous ordered spectral decomposition. Combining this with $\lambda'(X,H_i)\in{\cal C}_{\theta}(\lambda(X),\lambda(Y))$, $i=1, 2$ and applying Remark \ref{crit_charc}, we deduce that   $H_1$, $H_2\in{\cal C}_g(X,Y)$. Since $H=\sigma H_1+(1-\sigma)H_2$,
we arrive at $H \in {\rm aff}\,\mathcal{C}_g\left({X},{Y}\right)$,
which completes the proof. 
\end{proof}

Suppose that  $(X,Y)\in \gph \sub g$ and define the set 
$$\mathcal{R}^n_{\gtrsim}(X):=\{{d}\in\R^n \mid {d}_1\geq\dots\geq{d}_{|\alpha^1|};\dots;{d}_{n-|\alpha^{r}|+1}\geq\cdots\geq{d}_{n} \},$$
where the index sets $\al^l$, $l=1,\ldots,r$ are defined by \eqref{eq:def-alpha}. 
Using the index set $\eta(\lambda(X),\lambda(Y))$ from \eqref{eq:eta}, define the index sets  
$$\eta_1(\lambda(X),\lambda(Y)):=
\Big\{\nu\in\iota_1(\lambda(X))\big |\; \exists\,\mu\in\iota_2(\lambda(X))\;\mbox{such that}\;(\nu,\mu)\in\eta(\lambda(X),\lambda(Y))\Big\}
$$ 
and 
$$\eta_2(\lambda(X),\lambda(Y)):=
\Big\{\mu\in\iota_2(\lambda(X))\big|\; \exists\,\nu\in\iota_1(\lambda(X))\;\mbox{such that}\;(\nu,\mu)\in\eta(\lambda(X),\lambda(Y))\Big\}.
$$ 
It is not hard to see that 
\begin{equation}\label{eta_index}
\eta(\lambda(X),\lambda(Y))=\eta_1(\lambda(X),\lambda(Y))\times \eta_2(\lambda(X),\lambda(Y)).
\end{equation}
Taking into account  these index sets, define the set $K(X,{Y})$ as
\begin{equation}\label{eq:defintset}
\begin{aligned}
\left\{d\in\mathcal{R}^n
\Bigg |\left.
\begin{array}{l}
\begin{array}{l}
\langle d,{a}^{\nu'}-{a}^\nu\rangle=0, 
\quad\nu, \nu'\in \eta_1(\lambda(X),\lambda(Y)),\\[3pt]
\langle d,{b}^\mu\rangle=0, \quad\quad\quad\;\mu\in \eta_2(\lambda(X),\lambda(Y)),
\end{array}
\\ [3pt]
\begin{array}{l}
    \langle d,{a}^{\nu'}-{a}^\nu\rangle<0, \quad\nu'\in\iota_1(\lambda(X))\backslash\eta_1(\lambda(X),\lambda(Y)), \nu\in\eta_1(\lambda(X),\lambda(Y)),\\[3pt]
     \langle d,{b}^{\mu}\rangle<0, \quad\quad\quad\;\mu\in\iota_2(\lambda(X))\backslash\eta_2(\lambda(X),\lambda(Y))
\end{array}
\end{array}\right.\right\}. 
\end{aligned}
\end{equation}

In what follows, we often are going to assume that $K(X,Y)\neq\varnothing$ for a given $(X,Y)\in \gph \sub g$. It is not hard to see that  this condition can be  checked when $\theta$ is a polyhedral function. Indeed, we will demonstrate in Examples \ref{exp:le}-\ref{exp:combine} that  this condition  automatically holds for important instances of the spectral function $g$. 

The following lemma provides a simple sufficient condition, holding automatically  in important examples of the spectral function in \eqref{spec}, under which  $K(X,Y)\neq\varnothing$ is satisfied.  To do so, consider the following system of equations:
    \begin{equation}\label{eqs1}\left[\begin{array}{c}
b_1^{\top} \\
 \vdots \\
 b_{|\eta_2(\lambda(X),\lambda(Y))|}^{\top}\\
 b_{|\eta_2(\lambda(X),\lambda(Y))|+1}^{\top}\\
 \vdots\\
 b_{|\iota_2(\lambda(X))|}^{\top}
\end{array}\right]d=\left[\begin{array}{c}
0 \\
 \vdots \\
 0\\
 -1\\
 \vdots \\
 -1
\end{array}\right]
\quad \mbox{and}\quad 
\left[\begin{array}{ccccc}
1,-1,0,\dots,0\\
1,0,-1,\dots,0\\
\vdots\\
1,0,0\dots,-1
\end{array}\right]\left[\begin{array}{c}
a_1^{\top} \\
 \vdots \\
 a_{|\eta_1(\lambda(X),\lambda(Y))|}^{\top}\\
 a_{|\eta_1(\lambda(X),\lambda(Y))|+1}^{\top}\\
 \vdots\\
 a_{|\iota_1(\lambda(X))|}^{\top}
\end{array}\right]d=\left[\begin{array}{c}
0 \\
 \vdots \\
 0\\
 -1\\
 \vdots \\
 -1
\end{array}\right].\end{equation}

\begin{Lemma}\label{suff_k}
Suppose that $({X}, {Y}) \in \operatorname{gph} \partial g$ and $X$ has the eigenvalue decomposition \eqref{eq:eig-decomp1} with $P\in {\bf O}^n({X}) \cap {\bf O}^n({Y})$. If the sets $\{b^i\}_{i\in\iota_2(\lambda(X))}$ and  $\{a^i\}_{i\in\iota_1(\lambda(X))}$ are linear independent and there exists a common solution to the systems of equations in   \eqref{eqs1}, then we have  $K(X,Y)\neq\varnothing$.  
\end{Lemma}
\begin{proof}
We first consider the case where $\theta=\theta_2$ in \eqref{eq:def_phi12}. 
 It is easy to see that when $\{b^i\}_{i\in\iota_2(\lambda(X))}$ are chosen from orthogonal basis of $\mathbb{R}^n$, $d\in K(X,Y)$ can be chosen directly by letting $d_{\eta_2(\lambda(X),\lambda(Y))}=0$ and $d_{\iota_2(\lambda(X))\backslash\eta_2(\lambda(X),\lambda(Y))}=-1$. When $\{b^i\}_{i\in\iota_2(\lambda(X))}$ are linear independent,  the left-hand side system of equations in \eqref{eqs1} has a solution $d$.
Thus,  $d\in K(X,Y)$, which implies $K(X,Y)\neq\varnothing$. 

For the case where where $\theta=\theta_1$ in \eqref{eq:def_phi12}, 
it follows from our assumption that the right-hand side system of equations in \eqref{eqs1} has a solution $d$. Thus, we again arrive at
$d\in K(X,Y)$, which implies $K(X,Y)\neq\varnothing$. 

When $\theta=\theta_1+\theta_2$ in \eqref{eq:def_phi12}, the proof is  similar to the above two cases, so we omit it here for simplicity. 
\end{proof}

Note in Lemma~\ref{suff_k} that in many application we often have $\th=\th_1$ or $\th=\th_2$. In such a case the second assumption therein on the existence of a common solution to \eqref{eqs1} is superfluous.

Now we have been equipped with the necessary tools for the derivation of the minimal quadratic bundle of this specific spectral function. The following proposition is the first step.

\begin{Proposition}\label{quad_minimal}
Suppose that $({X}, {Y}) \in \operatorname{gph} \partial g$ and $X$ has the eigenvalue decomposition \eqref{eq:eig-decomp1} with $P\in {\bf O}^n({X}) \cap {\bf O}^n({Y})$.   
Suppose that $K(X,Y)\neq\varnothing$. 
Then,  the proper convex function $q:\S^n\to \oR$,  defined for any $H\in\S^n$ by 
	\begin{eqnarray}
	q(H)=\varUpsilon_{X,Y}\big(H\big)+\delta_{{\rm aff}\,{\cal C}_{g}(X,Y)}(H),\label{eq:qeqsgF}
	\end{eqnarray}
is a GQF and  $q\in {\rm quad}\,g(X,{Y})$, where $\varUpsilon_{X,Y}$ is taken from \eqref{eq:sigtsdp}.
\end{Proposition}
\begin{proof}
Since $Y\in \sub g(X)=\sub (\th\circ \lm)(X)$, 
Set $\overline{\iota}_1:=\iota_1(\lambda(X))$, $\overline{\iota}_2:=\iota_2(\lambda(X))$, $\overline{\eta}:=\eta(\lambda(X),\lambda(Y))$.   
We break the proof into two steps. 
\vskip 5pt
\noindent{\em Step 1: There is  a sequence $\{(X^k,Y^k)\b\subset \gph \sub g$ such that 
$$\iota_1(\lambda(X^k))\times\iota_2(\lambda(X^k))=\eta(\lm(X^k),\lm(Y^k))=\overline{\eta}. 
$$}

To prove it,  take  $w\in K(X,Y)$. 
Define the sequence $\{{w}^k\b$ by $w^k:={w}/k$ for each $k$. Clearly, $\{{w}^k\b$ converges to 0 and for each $k$, ${w}^k$ belongs to the right hand side of \eqref{eq:defintset}. 
It follows from Lemma \ref{lem:wIdforg} that there exists $\rho_s^l$ such that for all $s\in{\cal E}^l$, 
\begin{equation}\label{eq:equconst}
w_{\gamma_s^l}=\rho^l_s{\bf 1}_{|\gamma_s^l|}.
\end{equation} 
For each $k$, define ${x}^k:=\lambda({X})+w^k$ and $X^k:={P}{\rm Diag}({x}^k){P}^{\top}$ with $P\in {\bf O}^n({X}) \cap {\bf O}^n({Y})$ and $Y^k={Y}$. Then, we have $Qx^k=\lambda(X^k)$ for some $Q\in{\cal P}^n$ and   $\iota_1(x^k)\times\iota_2(x^k)=\overline{\eta}$ for any $k$ sufficiently large. Observe that $Q$ can be chosen independent of $k$ by passing to a subsequence if necessary. Indeed, $\iota_1(x^k)\times\iota_2(x^k)\supset\overline{\eta}$ follows directly from \eqref{eta_index} and 
\cite[Theorem 2.1]{MSarabi18}. Suppose by contradiction that there exists $(\nu',\mu')\in\iota_1(x^k)\times\iota_2(x^k)$ but $(\nu',\mu')\notin\overline{\eta}$. By \eqref{eta_index}, we have $\nu'\in\iota_1(x^k)\backslash\eta_1(\lambda(X),\lambda(Y))$ or $\mu'\in\iota_2(\lambda(X))\backslash\eta_2(\lambda(X),\lambda(Y))$. It follows from the construction of $X^k$ that we should have for all $(\nu,\mu)\in\overline{\eta}$, 
\begin{equation}\label{eq:x2l}\langle x^k,a^{\nu'}\rangle-c_{\nu'}<\langle x^k,a^{\nu}\rangle-c_{\nu}\;\mbox{or}\;\langle x^k,b^{\mu'}\rangle-d_{\mu'}<\langle x^k,b^{\mu}\rangle-d_{\mu},
\end{equation}
which implies $\nu'\notin\iota_1(x^k)$ or $\mu'\notin\iota_2(x^k)$. Thus, we have $(\nu',\mu')\notin\iota_1(x^k)\times\iota_2(x^k)$, which leads to a contradiction.  Therefore, we have verified $\iota_1(x^k)\times\iota_2(x^k)=\overline{\eta}$. 
Now, we claim that $\iota_1(\lambda(X^k))\times\iota_2(\lambda(X^k))=\overline{\eta}$. Still from \eqref{eta_index} and 
\cite[Theorem 2.1]{MSarabi18}, we get $\iota_1(\lambda(X^k))\times\iota_2(\lambda(X^k))\supset\overline{\eta}$. 
We conclude for all $\nu\in\iota_1(\lambda(X^k))$,  $\mu\in\iota_2(\lambda(X^k))$, and for all $\nu'\in \iota_1(\lambda(X^k))^c$ and $\mu'\in\iota_2(\lambda(X^k))^c$ that 
$$\langle \lambda(X^k),a^{\nu'}\rangle-c_{\nu'}<\langle \lambda(X^k),a^{\nu}\rangle-c_{\nu}\quad\mbox{and}\quad\langle \lambda(X^k),b^{\mu'}\rangle-d_{\mu'}<\langle \lambda(X^k),b^{\mu}\rangle-d_{\mu},$$
which lead us to  
$$\langle x^k,Qa^{\nu'}\rangle-c_{\nu'}<\langle x^k,Qa^{\nu}\rangle-c_{\nu}\quad \mbox{and}\quad \langle x^k,Qb^{\mu'}\rangle-d_{\mu'}<\langle x^k,Qb^{\mu}\rangle-d_{\mu}.$$
Therefore, it follows from \cite[Proposition 1]{cuiding} that there exists $i$ such that 
$$\langle x^k,Qa^{\nu}\rangle-c_{\nu}=\langle x^k,a^i\rangle-c_i$$
and exists $j$ such that 
$$\langle x^k,Qb^{\mu}\rangle-d_{\mu}=\langle x^k,b^j\rangle-d_j, $$
meaning that  $i\in\iota_1(x^k)$ and $j\in\iota_2(x^k)$. Combining this with $$|\iota_1(x^k)|+|\iota_1(x^k)^c|=|\iota_1(\lambda(X^k))|+|\iota_1(\lambda(X^k))^c|,$$
we have  $|\iota_1(\lambda(X^k))||\iota_2(\lambda(X^k))|\leq|\iota_1(x^k)||\iota_2(x^k)|=|\bar{\eta}|$. Since  $\iota_1(\lambda(X^k))\times\iota_2(\lambda(X^k))\supset\overline{\eta}$, we have $\iota_1(\lambda(X^k))\times\iota_2(\lambda(X^k))=\overline{\eta}$.  
Since $w^k \in {\cal C}_{\theta}(\lambda({X}),\lambda(Y))$, 
we arrive at 
$$
(x^k, \lm(Y))- (\lm(X), \lm(Y))=(w^k,0)\in \gph N_{{\cal C}_{\theta}(\lambda({X}),\lambda(Y))}.
$$
Employing now the reduction lemma for polyhedral functions
in \cite[Theorem~3.1]{hjs}, we obtain $\lm(Y)\in \sub \th\big(x^k\big)$. By using \eqref{partialdiffphi} and the symmetric property of $\theta$, we also have $\lm(Y)\in \sub \th\big(\lambda(X^k)\big)$. 
It can be seen directly that $\overline{\eta}=\eta(\lm(X^k),\lm(Y^k))$.  

\vskip 5pt
\noindent{\em Step 2: The sequence $\{\d^2g(X^k, Y^k)\b$ epi-converges to   the function on the right hand-side of \eqref{eq:qeqsgF}. Moreover, $\d^2g(X^k, Y^k)$ is a GQF for each $k$. }

To prove the claim, define the index set $\{\chi_t^l\}$ 
as the one in Remark~\ref{part}(P1) for   the sequence $\{(X^k,Y^k)\b$ defined above. 
Since $\theta$ is symmetric and polyhedral, we know from   Lemma \ref{sec_subd} that for any $H\in\S^n$, 

\begin{equation*}
\begin{aligned}
&\d^2g(X^k, Y^k)(H)
=2\sum\limits_{1\leq l\leq l'\leq r}\sum\limits_{1\leq t\leq t'\leq z^l\atop (l,t)\neq(l',t')}\sum\limits_{i\in\chi^l_t}\sum\limits_{j\in\chi^{l'}_{t'}}\frac{\lambda_i({Y})-\lambda_j({Y})}{\lambda_i(X^k)-\lambda_j(X^k)}(P^{\top}HP)^2_{ij}
+\delta_{{\cal C}_{g}(X^k,Y^k)}(H)\\
=&2\sum\limits_{1\leq l< l'\leq r}\sum\limits_{i\in\alpha^l}\sum\limits_{j\in\alpha^{l'}}\frac{\lambda_i({Y})-\lambda_j({Y})}{\lambda_i(X^k)-\lambda_j(X^k)}(P^{\top}HP)^2_{ij}\\
&+2\sum\limits_{1\leq l\leq r}\sum\limits_{1\leq t<t'\leq z^l}\sum\limits_{i\in\chi^l_t}\sum\limits_{j\in\chi^{l}_{t'}}\frac{\lambda_i({Y})-\lambda_j({Y})}{\lambda_i(X^k)-\lambda_j(X^k)}({P}^{\top}H{P})^2_{ij}
+\delta_{{\cal C}_{g}(X^k,Y^k)}(H).
\end{aligned}
\end{equation*}	
According to Lemma~\ref{lemma:characofcc} and the observation in {\em Step 1}, we can conclude that   $\d^2g(X^k, Y^k)$ is a GQF for each $k$. 
For any $H\in \S^n$, set 
$$f_1^k(H):=2\sum\limits_{1\leq l< l'\leq r}\sum\limits_{i\in\alpha^l}\sum\limits_{j\in\alpha^{l'}}\frac{\lambda_i({Y})-\lambda_j({Y})}{\lambda_i(X^k)-\lambda_j(X^k)}({P}^{\top}H{P})^2_{ij},$$
$$f_2^k(H):=2\sum\limits_{1\leq l\leq r}\sum\limits_{1\leq t<t'\leq z^l}\sum\limits_{i\in\chi^l_t}\sum\limits_{j\in\chi^{l}_{t'}}\frac{\lambda_i({Y})-\lambda_j({Y})}{\lambda_i(X^k)-\lambda_j(X^k)}({P}^{\top}H{P})^2_{ij}.$$ 
Indeed, $f_1$ consists of terms in $\d^2g(X^k, Y^k)$ where the indices $i$ and $j$ belong to two different index sets $\al^l$ and $\al^{l'}$. On the other hand, $f_2$ captures those terms in $\d^2g(X^k, Y^k)$ where the indices $i$ and $j$ belong to the same $\al^l$. 
 By the definition of continuous convergence from \cite[page 250]{rw}, we conclude that $f_1^k$ converges continuously to $\varUpsilon_{X,Y}$ defined by  \eqref{eq:sigtsdp}. 
Next, we are going to show that $f_2^k+\delta_{{\cal C}_{g}(X^k,Y^k)}$ epi-converges to $\delta_{\mbox{aff}\,{\cal C}_{g}({X},{Y})}$.  To prove it,  define for each $k$
\begin{equation}\label{eq:CCk}{\cal C}_1^k=
\left\{H\in\S^n\Big |\begin{array}{l}
 \langle\lambda'(X^k,H),{a}^\nu-{a}^{\nu'}\rangle=0\;\;\mbox{for all}\;\; \nu,\nu'\in \iota_1(\lm(X^k)),\\
\langle\lambda'(X^k,H),{b}^\mu\rangle=0\;\;\mbox{for all}\;\;  \mu\in \iota_2(\lm(X^k))\\
\end{array}\right\},\end{equation}
and
$${\cal C}_2^k=\left\{ H\in\S^n\Big | \begin{array}{l}
\Lambda(Y)_{\chi_t^l\chi_t^l}\;\mbox{and}\;P^{\top}_{\chi_t^l}HP_{\chi_t^l}\;\mbox{have a simultaneous ordered spectral decomposition}
\end{array}\right\}.$$
Applying Remark~\ref{crit_charc}, Remark~\ref{remark:invariant-y1-y2},
and Lemma~\ref{crit_gqf} for $(X^k,Y^k)$ shows that ${\cal C}_{g}(X^k,Y^k)={\cal C}_1^k\cap{\cal C}_2^k$. 
\vskip 5pt
\noindent{\em Step 2.1: For any $H\in \S^n$ and every $H^k\rightarrow H$, 
we have 
\begin{equation}\label{epi_limit}
\liminf_{k\rightarrow\infty}(f_2^k+\delta_{{\cal C}_{g}(X^k,Y^k)})(H^k)\geq\delta_{{\rm aff}\,{\cal C}_{g}(X,Y)}(H).
\end{equation}}

It is easy to see that $f_2^k+\delta_{{\cal C}_{g}(X^k,Y^k)}\geq0$.
Pick $\widetilde H\in \S^n$ and and assume that  $H^k\rightarrow \widetilde H$. If $\liminf\limits_{k\rightarrow\infty}(f_2^k+\delta_{{\cal C}_{g}(X^k,Y^k)})(H^k)=\infty$, the claimed inequality clearly holds. 
Assume that $\liminf\limits_{k\rightarrow\infty}(f_2^k+\delta_{{\cal C}_{g}(X^k,Y^k)})(H^k)<\infty$. Thus, we can assume by passing to a subsequence, if necessary, that $H^k\in {\cal C}_{g}(X^k,Y^k)\cap \dom f^k_2$ and for any $k$.
 Our goal is to show that $\widetilde H\in {\rm aff}\,{\cal C}_{g}(X,Y)$
via the equivalent description of the latter set from Lemma~\ref{lemma:charofaffcc}. To this end, observe from 
 the definition of ${w}^k$ that  the sequence  $\{f_2^k\}$ is nondecreasing, since 
$$
\lambda_i(X^k)-\lambda_j(X^k)=\frac{w_i-w_j}{k} \ge \frac{w_i-w_j}{k+1}=\lambda_i(X^{k+1})-\lambda_j(X^{k+1})
$$
for any $i\in \chi^l_t$ and $j\in \chi^l_{t'}$ with $l\in \{1,\ldots,r\}$
and $1\le t<t'\le z^l$. Thus,  \cite[Proposition 7.4]{rw} implies that $f_2^k$ epi-converges to $\mbox{sup}_k\{\mbox{cl}f_2^k\}$, where $\mbox{cl}f_2^k$ is the closure of $f_2^k$ (cf. \cite[equation~1(6)]{rw}).  It is easy to see that $\mbox{sup}_k\{\mbox{cl}f_2^k\}=\mbox{sup}_k\{f_2^k\}=\delta_{{\cal D}_1}$,
 where 
\begin{equation}\label{d1}
 {\cal D}_1:=\big\{H\in\S^n\mid (P^{\top}HP)_{ij}=0,\,\mbox{for all}\,i,j\,\mbox{not in the same}\, \chi_t^l\,\mbox{and not in the same}\,\gamma_s^l \big\}   
\end{equation}
 with the index sets $\gamma^l_s$ for $s\in \{1,\ldots,u^l\}$ taken from \eqref{eq:partonY}.
 Moreover, it follows from Remark \ref{crit_charc} that ${\cal C}_2^k\subset
 {\cal D}_2:=\big\{H\in\S^n\mid (P^{\top}HP)_{ij}=0,\,\mbox{for all}\,i,j\in\chi_t^l\,\mbox{and}\,i,j\,\mbox{not in the same}\, \gamma_s^l\big\}$; see Figure~\ref{fig2} in which the blocks of $P^\top HP$
 that are zero for a matrix $H$ in  ${\cal D}_1$ or $ {\cal D}_2$ are hatched.
 \begin{figure*}[h!]
\centering 
\subfigure[for $H$ in ${\cal D}_1$]{\includegraphics[width=.45\linewidth]{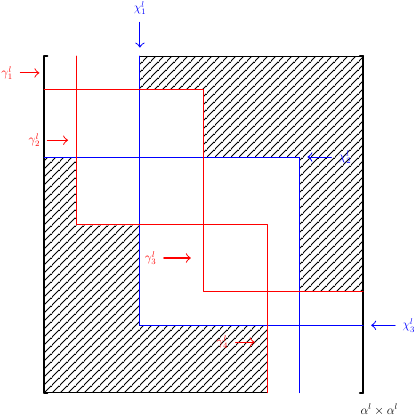}}
\subfigure[for $H$ in ${\cal D}_2$ ]{\includegraphics[width=.45\linewidth]{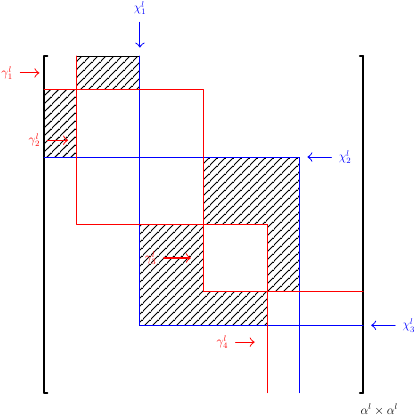}}
\caption{$\al^l\times \al^l$ block of $P^\top HP$}
\label{fig2}
\end{figure*}

Thus, we get  
 $$
 {\cal D}_1\cap {\cal D}_2 =\big\{H\in\S^n\mid (P^{\top}HP)_{\gamma^l_s\gamma^l_{s'}}=0\; \mbox{for all}\; l=1\dots,r \;\mbox{and}\; s\neq s'\big\}.
 $$ 
Because $f_2^k\xrightarrow{e} \dd_{{\cal D}_1}$ and $\liminf_{k\to \infty}f_2^k(H^k)<\infty$, it follows from \cite[Proposition~7.2]{rw} that 
$\dd_{{\cal D}_1}(\widetilde H)\le \liminf_{k\to \infty}f_2^k(H^k)$, implying that
$\widetilde H\in {\cal D}_1$. Moreover, since $H^k\in {\cal C}_2^k\subset{\cal D}_2$, we get   $\widetilde H\in {\cal D}_1\cap{\cal D}_2$. This implies that condition (a) in Lemma \ref{lemma:characofcc} holds for the matrix $\widetilde H$. 

To verify condition (c)  in Lemma \ref{lemma:characofcc} for the matrix $\widetilde H$, we need to use Lemma~\ref{lemma:characofcc}(c) for  $H^k\in {\cal C}_{g}(X^k,Y^k)$. That requires further partitioning of the index sets 
$\gamma_s^l$. To do so, define the index sets $\beta^l_s(t)$, $t\in \{1,\dots,v^{l,s}\}$ with $v^{l,s}\in \N$, to further partition the index set $\gamma_s^l$ from Remark~\ref{part}(P3) based on $x^k=\lambda(X^k)$ as
\begin{equation}\label{eq:partonXpY2}
\begin{cases}
\lm_i(X^k)=\lm_j(X^k),&{\rm if}\; i,j\in\beta^l_s(t)\;{\rm and}\;t\in\{1,\dots,v^{l,s}\},\\
\lm_i(X^k)>\lm_j(X^k),&{\rm if}\, i\in\beta^l_s(t), j\in\beta^l_s(t')\,{\rm and}\,t,t'\in\{1,\dots,v^{l,s}\}\,{\rm with}\,t<t'.
\end{cases}
\end{equation}
We can conclude from \eqref{eq:equconst} and the definition of $\lm(X^k)$ that for any  $s\in{\cal E}^l$, we must have $v^{l,s}=1$,
 which means  that  $\gamma_s^l$ will not be further partitioned by $\lm(X^k)$ whenever $s\in{\cal E}^l$. Since $\widetilde H\in {\cal D}_1\cap {\cal D}_2$,   each $(P^{\top} \widetilde HP)_{\chi^l_t\chi^l_t}$ is a block diagonal matrix, whose   diagonal blocks are composed of some $(P^{\top} \widetilde HP)_{\gamma^l_s\gamma^l_s}$; see Figure~\ref{fig2}.  
Since $H^k\in {\cal C}_{g}(X^k,Y^k)$, we know from Lemma \ref{lemma:characofcc}(c) that for all $s\in{\cal E}^l$,  $(P^{\top}H^k P)_{\gamma^l_s\gamma^l_s}=\rho_s^{l}(k)I_{\left|\gamma_s^l\right|}$ for some $\rho_s^{l}(k)\in\R$.
Since the sequence $\{H^k\b$ is convergent, so is the sequence $\{\rho_s^{l}(k)\b$.  
 Passing to the limit shows the validity of condition (c)  in Lemma \ref{lemma:characofcc} for the matrix $\widetilde H$.
We now claim that 
\begin{equation}\label{eq:CCk2}
\begin{aligned}
{\cal C}_1^k=\left\{ H\in\S^n\Big | \begin{array}{l}
\forall\,\nu,\nu'\in\iota_1^k,\,\mu\in\iota_2^k, \langle\mbox{diag}(P^{\top}HP),{a}^\nu-{a}^{\nu'}\rangle=0,\;
\langle\mbox{diag}(P^{\top}HP),{b}^{\mu}\rangle=0
\end{array}\right\}:={\cal C}.
\end{aligned}
\end{equation}
To justify it, set
$${\cal E}^l_1=\big\{k\in{\cal E}^l\mid \exists\, i, j \in \gamma_k^l \text { such that }\left({a}^\nu\right)_i \neq\left({a}^\nu\right)_j \text { for some } (\nu,\mu) \in \overline{\eta}\big\}.$$ 
For all $(\nu,\mu)\in\overline{\eta}$, we have from \eqref{eq:egdd} and $H^k\in{\cal C}_g(X^k,Y^k)$ that 
\begin{equation}\label{eq:S28P4}
\begin{aligned}
\langle\lambda'(X^k,H^k),{a}^\nu\rangle&=\sum_{l=1}^r\sum_{s=1}^{u^l}\sum_{t=1}^{v^{l,s}}\langle \lambda(P^{\top}_{\beta^l_s(t)}H^kP_{\beta^l_s(t)}),(a^\nu)_{\beta^l_s(t)}\rangle\\
&=\sum_{l=1}^r\big(\sum_{s\in{\cal E}^l_1}\sum_{t=1}^{v^{l,s}}\langle \lambda(P^{\top}_{\beta^l_s(t)}H^kP_{\beta^l_s(t)}),(a^\nu)_{\beta^l_s(t)}\rangle+\sum_{s\notin{\cal E}^l_1}\sum_{t=1}^{v^{l,s}}\langle \lambda(P^{\top}_{\beta^l_s(t)}H^kP_{\beta^l_s(t)}),(a^\nu)_{\beta^l_s(t)}\rangle\big)\\
&=\sum_{l=1}^r\big(\sum_{s\in{\cal E}^l_1}\langle \lambda(P^{\top}_{\gamma^l_s}H^kP_{\gamma^l_s}),(a^\nu)_{\gamma^l_s}\rangle+\sum_{s\notin{\cal E}^l_1}\sum_{t=1}^{v^{l,s}}\langle \lambda(P^{\top}_{\beta^l_s(t)}H^kP_{\beta^l_s(t)}),(a^\nu)_{\beta^l_s(t)}\rangle\big)\\
&=\sum_{l=1}^r\big(\sum_{s\in{\cal E}^l_1}\langle \rho_s^l{\bf 1}_{|\gamma_s^l|},(a^\nu)_{\gamma^l_s}\rangle+\sum_{s\notin{\cal E}^l_1}\sum_{t=1}^{v^{l,s}}\langle P^{\top}_{\beta^l_s(t)}H^kP_{\beta^l_s(t)},\tilde{\rho}_s^l{\bf 1}_{\beta^l_s(t)}\rangle\big)\\
&=\langle \mbox{diag}(P^{\top}H^kP),{a}^\nu\rangle, 
\end{aligned}\end{equation}
where $\rho_s^l$ and $\tilde{\rho}_s^l$ are some constants in $\R$. 
Similarly, we can   prove for all $(\nu,\mu)\in\overline{\eta}$ that 
\begin{equation}\label{eq:S28P4b}\langle\lambda'(X^k,H^k),{b}^\mu\rangle=\langle \mbox{diag}(P^{\top}H^kP),{b}^\mu\rangle. \end{equation}
Combining this, \eqref{eq:S28P4}, and \eqref{eq:CCk}
confirms \eqref{eq:CCk2}. Recall that $H^k\in {\cal C}_1^k={\cal C}$ for all $k$. Therefore, $\widetilde H\in {\cal C}$. Appealing now to Lemma~\ref{lemma:charofaffcc}
tells us that $\widetilde H\in \mbox{aff}\,{\cal C}_{g}({X},{Y})$ and so the right-hand side of \eqref{epi_limit} becomes zero. This proves \eqref{epi_limit}. 

\vskip 5pt
\noindent{\em Step 2.2: We have $f_2^k+\delta_{{\cal C}_{g}(X^k,Y^k)}\xrightarrow{e} \delta_{{\rm aff}\,{\cal C}_{g}(X,Y)}$.}

To prove this claim, we begin by showing that for any $H\in\S^n$, there exists  $H^k\rightarrow H$ such that 
\begin{equation}\label{epi_lim2}
\limsup\limits_{k\rightarrow\infty}(f_2^k+\delta_{{\cal C}_{g}(X^k,Y^k)})(H^k)\leq\delta_{{\rm aff}\,{\cal C}_{g}(X,Y)}(H).
\end{equation}
For any $H\notin{\rm aff}\,{\cal C}_{g}(X,Y)$, this inequality trivially holds. 
Pick any $H\in{\rm aff}\,{\cal C}_{g}(X,Y)$ and let $H^k=H$. Then we have 
$$\big(f_2^k+\delta_{{\cal C}_{g}(X^k,Y^k)}\big)(H)=f_2^k(H)+\delta_{{\cal C}_1^k}(H)+\delta_{{\cal C}_2^k}(H).$$
It follows from Lemma \ref{lemma:characofcc}(a) that $(P^{\top}HP)_{ij}=0$ for all $i,j$ not in the same $\gamma_s^l$, which implies $f_2^k(H)=0$. 
We know from \eqref{eq:affcha=01}, Lemma \ref{lemma:characofcc}(c) and a similar argument as \eqref{eq:S28P4} that $H\in {{\cal C}_1^k}$. Since $\lambda(Y)=\displaystyle\sum_{\nu}w_\nu a^\nu+\sum_{\mu}u_\mu b^\mu$ with $\displaystyle\sum_{\nu}w_\nu=1$, $w_\nu>0$, $u_\mu>0$ with $(\nu,\mu)\in\overline{\eta}$, we know again from \eqref{eq:egdd}, \eqref{eq:S28P4} and \eqref{eq:S28P4b} that as $k\rightarrow\infty$,
$$
\begin{aligned}
 \sum_{l=1}^{r}\sum_{t=1}^{z^l}\langle \Lm(P_{\chi^l_t}^{\top}HP_{\chi^l_t}),\Lambda(Y)_{\chi^l_t\chi^l_t}\rangle &=\langle \lambda'(X^k,H),\lambda(Y)\rangle=\langle \mbox{diag}(P^{\top}HP),\lambda(Y)\rangle\\
 &=\langle P^{\top}HP,\Lambda(Y)\rangle=\sum_{l=1}^{r}\sum_{t=1}^{z^l}\langle \Lambda(Y)_{\chi^l_t\chi^l_t},P_{\chi^l_t}^{\top}HP_{\chi^l_t}\rangle. 
\end{aligned}
$$
This, together with Fan's inequality \cite{Fan49}, shows that  $\Lambda(Y)_{\chi_t^l\chi_t^l}$ and $P^{\top}_{\chi_t^l}HP_{\chi_t^l}$ have a simultaneous ordered spectral decomposition, which implies $H\in {{\cal C}_2^k}$. Therefore, for any $H\in{\rm aff}\,{\cal C}_{g}(X,Y)$, $\limsup\limits_{k\rightarrow\infty}(f_2^k+\delta_{{\cal C}_{g}(X^k,Y^k)})(H^k)=0=\delta_{{\rm aff}\,{\cal C}_{g}(X,Y)}(H)$, which confirms \eqref{epi_lim2}. Combining \cite[Proposition~7.2]{rw} with the inequalities in \eqref{epi_limit} and \eqref{epi_lim2} demonstrates that $f_2^k+\delta_{{\cal C}_{g}(X^k,Y^k)}\xrightarrow{e} \delta_{{\rm aff}\,{\cal C}_{g}(X,Y)}$.

To finish the proof, recall that  $f_1^k$ converges continuously to $\varUpsilon_{X,Y}$.
Therefore, we obtain from \cite[Theorem 7.46(b)]{rw} and the observation in Step 2.2 that $\d^2g(X^k, Y^k)=f_1^k+f_2^k+\delta_{{\cal C}_{g}(X^k,Y^k)}$ epi-converges to $\varUpsilon_{X,Y}+\delta_{\mbox{aff}\,{\cal C}_{g}({X},{Y})}$,
	which completes the proof. 
\end{proof}

Next, we are going to show that the GQF in \eqref{eq:qeqsgF} is a minimal element in the quadratic bundle of the spectral function $g$ from \eqref{spec}, which is the main result of this section. 

\begin{Theorem}\label{prop:minquadbdcha} 
Suppose that $({X}, {Y}) \in \operatorname{gph} \partial g$ and $X$ has the eigenvalue decomposition \eqref{eq:eig-decomp1} with $P\in {\bf O}^n({X}) \cap {\bf O}^n({Y})$. 
Take the set $K(X,Y)$ from \eqref{eq:defintset} and assume that $K(X,Y)\neq\varnothing$.  
Then,   
the minimal element in  ${\rm quad}\,g(X,{Y})$ is $\varUpsilon_{X,Y}+\delta_{{\rm aff}\,{\cal C}_g(X,Y)}$, where 
$\varUpsilon_{X,Y}$ is taken from \eqref{eq:qeqsgF}. 
\end{Theorem}
\begin{proof} Suppose  $(X,Y)\in \gph \sub g$. 
We define  the function $\Xi_{X,Y}:\S^n\to \oR$ for any $H\in \S^n$ by 
\begin{equation*}
\Xi_{X,Y}(H):= \liminf_{\substack{
(X', Y') \to (X,Y)\\
 Y'\in \ri \sub g(X'), \; H'\to H}} 
\d^2 g(X' , Y')(H').
\end{equation*}
We proceed with dividing the proof into three steps.
\vskip 5pt
\noindent{\em Step 1. For any $H\in \S^n$, we always have  $\Xi_{X,Y}(H)\leq\varUpsilon_{X,Y}\big(H\big)+\dd_{{\rm aff}\,{\cal C}_{g}(X,Y)}(H)$.}

 To prove it, it follows from Proposition~\ref{sec_subd}(b) that 
$\d^2g(X,Y)$ is a GQF if and only if $Y\in \ri \sub g(X)$.
Taking $q\in{\rm quad}\,g(X,{Y})$, we find  $\{(X^k,Y^k)\b\subset \gph \sub g$ such that $Y^k\in \ri \sub g(X^k)$ and $\d^2 g(X^k,Y^k)\xrightarrow{e} q$. By \cite[Proposition~7.2]{rw}, for any $H\in \S^n$, there is a sequence $H^k\to H$ such that $\d^2 g(X^k,Y^k)(H^k)\to q(H)$. 
This clearly tells us that $\Xi_{X,Y}(H)\leq q(H)$ for any $H\in \S^n$.
According to  Proposition~\ref{quad_minimal}, $\varUpsilon_{X,Y}+\delta_{{\rm aff}\,{\cal C}_g(X,Y)}\in {\rm quad}\,g(X,{Y})$. Combining these confirms our first claim. 

To proceed with the next step, take $H\in \S^n$. There are sequences  $\{(X^k,Y^k)\b$,  $\{H^k\b$
such that 
\begin{equation}\label{eq:dsfinite}
\begin{cases}
(X^k,Y^k)\to (X,Y)\;\;\mbox{with}\;\; Y^k\in \ri \sub g(X^k),\; H^k\rightarrow H\\
\mbox{satisfying}\;\; \d^2g(X^k,Y^k)(H^k)\to \Xi_{X,Y}(H).    
\end{cases}    
\end{equation}
Thus, we get  $\Lambda(X)=\lim_{k\rightarrow\infty}\Lambda(X^k)$ and $\Lambda(Y)=\lim_{k\rightarrow\infty}\Lambda(Y^k)$. Since $\{P^k\b$ is uniformly bounded, we  may assume by taking a subsequence, if necessary,  that
$\{P^k\b$ converges to an orthogonal matrix  $\widehat{P}\in{\bf O}^n({X}) \cap {\bf O}^n({Y})$. As stated in \cite{SSun03} and essentially proved in the derivation of \cite[Lemma 4.12]{SSun02}, we can write  $\widehat{P}=PQ$, where 
$$
Q\in{\cal D}:=\Big\{Q\in{\bf O}^n\mid Q={\rm Diag}(Q_1^1,\dots,Q_l^{u^l})\;\;\mbox{with}\;\; Q_t^l\in{\bf O}^{|\gamma_t^l|}\Big\}.
$$  
Take the index sets $\alpha^l$  with $l=1,\ldots,r$ from Remark~\ref{part}(P1).
Moreover, for any $l\in \{1,\ldots,r\}$, take the index sets $\{\chi_t^l\}_{t=1}^{z^l}$ from Remark~\ref{part}(P2).

\vskip 5pt
\noindent{\em Step 2. For any $H\in\S^n$, we always have 
\begin{equation}\label{eq:woutdom}
\Xi_{X,Y}(H)\geq\varUpsilon_{X,Y}\big(H\big)+\delta_{{\rm aff}{\cal C}_{\theta}(\lambda(X),\lambda(Y))}(w), 
\end{equation}
where \begin{equation}\label{eq:woutdom2}w:=\left(\lambda((Q^{\top}P^{\top}HPQ)_{\chi^1_1\chi^1_1}),\dots,\lambda((Q^{\top}P^{\top}HPQ)_{\chi^r_{u^r}\chi^r_{u^r}})\right)\end{equation} with $Q\in{\cal D}$. Consequently, if $\Xi_{X,Y}(H)<\infty$, then $w\in {\rm aff}{\cal C}_{\theta}(\lambda(X),\lambda(Y))$.}

To justify this claim, take $H\in \S^n$ and the   sequences  $\{(X^k,Y^k)\b$ and  $\{H^k\b$ satisfying \eqref{eq:dsfinite}. 
It follows from \cite[Proposition 5.3]{MSarabi23} that 
{\small{
\begin{equation}\label{eq:21P5.3}
\mathrm{d}^2 g(X^k, Y^k)(H^k) \geq \mathrm{d}^2 \theta(\lambda(X^k), \lambda(Y^k))\left(\lambda'(X^k ; H^k)\right)+2 \sum_{l=1}^r\sum_{t=1}^{z^l}\left\langle\Lambda(Y^k)_{\chi_t^l \chi_t^l}, (P^k_{\chi_t^l})^{\top} H^k\left(\mu_l^k I-X^k\right)^{\dagger} H^k P^k_{\chi_t^l}\right\rangle,
\end{equation}}}
where $z^l$ is taken from \eqref{eq:partonxk}. 
It can be checked directly that 
\begin{equation}\label{eq:sigtgeq}
    \begin{aligned}
    & 2\sum_{l=1}^r\sum_{t=1}^{z^l}\left\langle\Lambda(Y^k)_{\chi_t^l \chi_t^l}, (P^k_{\chi_t^l})^{\top} H^k\left(\mu_l^k I-X^k\right)^{\dagger} H^k P^k_{\chi_t^l}\right\rangle\\
    &=2\sum\limits_{1\leq l< l'\leq r}\sum\limits_{i\in\alpha^l}\sum\limits_{j\in\alpha^{l'}}\frac{\lambda_i({Y}^k)-\lambda_j({Y}^k)}{\lambda_i(X^k)-\lambda_j(X^k)}(({P}^k)^{\top}H^k{P}^k)^2_{ij}\\
&+2\sum\limits_{1\leq l\leq r}\sum\limits_{1\leq t<t'\leq z^l}\sum\limits_{i\in\chi_t^l}\sum\limits_{j\in\chi_{t'}^l}\frac{\lambda_i({Y}^k)-\lambda_j({Y}^k)}{\lambda_i(X^k)-\lambda_j(X^k)}(({P}^k)^{\top}H^k{P}^k)^2_{ij}\\
&\geq 2\sum\limits_{1\leq l< l'\leq r}\sum\limits_{i\in\alpha^l}\sum\limits_{j\in\alpha^{l'}}\frac{\lambda_i({Y}^k)-\lambda_j({Y}^k)}{\lambda_i(X^k)-\lambda_j(X^k)}(({P}^k)^{\top}H^k{P}^k)^2_{ij},
    \end{aligned}
\end{equation}
where the last inequality follows from   
$$
\frac{\lambda_i({Y}^k)-\lambda_j({Y}^k)}{\lambda_i(X^k)-\lambda_j(X^k)} \geq0\quad \mbox{for all}\;\; i,j\in\{1,\dots,n\}.
$$
Passing to the limit  brings us to  
\begin{equation}\label{eq:sigtgeq2}\liminf_{k\rightarrow\infty}2\sum_{l=1}^r\sum_{t=1}^{z^l}\left\langle\Lambda(Y^k)_{\chi_t^l \chi_t^l}, (P^k_{\chi_t^l})^{\top} H^k\left(\mu_l^k I-X^k\right)^{\dagger} H^k P^k_{\chi_t^l}\right\rangle\geq\varUpsilon_{X,Y}\big(H\big).\end{equation}
On the other hand, since $Y^k\in\mbox{ri}\,\partial g(X^k)$, we know from Lemma    \ref{crit_gqf} that  
${\cal C}_{g}(X^k,Y^k)$ has a representation of the form 
$${\cal C}_{\theta}(\lambda(X^k),\lambda(Y^k))=\left\{d\in\R^n\Bigg | \begin{array}{l}
\langle d,{a}^{\nu'}-{a}^\nu\rangle=0,\;\mbox{if}\;\nu,\nu'\in{\iota}_1^k,\\
\langle d,{b}^\mu\rangle=0,\;\mbox{if}\;\mu\in{\iota}_2^k,\\
\end{array}\right\}.$$
It results from \cite[Proposition~3.3]{hjs} that there are faces $K_1$ and $K_2$ of the critical cone ${\cal C}_{\theta}(\lambda(X),\lambda(Y))$ for which we have 
$$
\mathrm{d}^2 \theta(\lambda(X^k), \lambda(Y^k))=\dd_{{\cal C}_{\theta}(\lambda(X^k),\lambda(Y^k))}=\dd_{K_1-K_2}\ge \dd_{{\rm aff}{\cal C}_{\theta}(\lambda(X),\lambda(Y))},
$$
where the last inequality comes from \cite[Corollary~3.4]{hjs}. 
Observe from \eqref{eq:egdd} 
that  
$$
 \lambda' ( X^k ; H^k )=\Big(\lambda \big((P^k_{\chi^1_1}\big)^{\top} H^k P^k_{\chi^1_1}\big),\ldots,\lambda \big((P^k_{\chi^r_{u^r}}\big)^{\top} H^k P^k_{\chi^r_{u^r}}\big) \Big)^{\top}, 
 $$
which leads us to $\lambda'(X^k,H^k)\rightarrow w$ with $w$ taken from \eqref{eq:woutdom2}. If $\lm'(X^k ; H^k)\in
{\cal C}_{\theta}(\lambda(X^k),\lambda(Y^k))\subset {\rm aff}{\cal C}_{\theta}(\lambda(X),\lambda(Y))$ for infinitely many $k$, we 
can conclude   that $w\in {\rm aff}{\cal C}_{\theta}(\lambda(X),\lambda(Y))$ and that 
$$
\liminf_{k\rightarrow\infty}\mathrm{d}^2 \theta(\lambda(X^k), \lambda(Y^k))\big(\lambda'(X^k ; H^k)\big)=0\ge 0=\dd_{{\rm aff}{\cal C}_{\theta}(\lambda(X),\lambda(Y))}(w).
$$
Otherwise, $\lm'(X^k ; H^k)\in
{\cal C}_{\theta}(\lambda(X^k),\lambda(Y^k))$ for only finitely many $k$, which tells us that 
$$
\liminf_{k\rightarrow\infty}\mathrm{d}^2 \theta(\lambda(X^k), \lambda(Y^k))\big(\lambda'(X^k ; H^k)\big)=\infty\ge \dd_{{\rm aff}{\cal C}_{\theta}(\lambda(X),\lambda(Y))}(w).
$$
Combining these with \eqref{eq:21P5.3} and  \eqref{eq:sigtgeq2} leads us to  \eqref{eq:woutdom}.

\vskip 5pt
\noindent{\em Step 3. We have $\dom \Xi_{X,Y} ={\rm aff}\,{\cal C}_{g}(X,Y)$. } 

Observe first that the inclusion ${\rm aff}\,{\cal C}_{g}(X,Y)\subset \dom \Xi_{X,Y}$ falls immediately from the observation in Step 1. To prove the opposite inclusion, pick   $H\in \dom \Xi_{X,Y}$. We proceed to show that 
$H$ satisfies the characterization of ${\rm aff}\,{\cal C}_{g}(X,Y)$, obtained in Lemma~\ref{lemma:charofaffcc}.  We begin with proving that $H$ satisfies condition (a) in Lemma~\ref{lemma:charofaffcc}. 
It follows  from \eqref{eq:dsfinite}  that for any $k$ sufficiently large, 
\begin{equation}\label{eq:finite}
\d^2 g(X^k, Y^k)(H^k)\leq M:=\Xi_{X,Y}(H)+1<\infty. 
\end{equation}
For each $k$, pick the sequences $t_{k_m}\searrow 0$ and $H^{k_m}\rightarrow H^k$ such that 
\begin{equation}\label{eq:finite2}
\d^2 g(X^k, Y^k)(H^k)=\lim_{m\rightarrow\infty}\Delta_{t_{k_m}}^2g(X^k,Y^k)(H^{k_m}).
\end{equation}
We know from  \cite[Proposition 1.4]{Torki} that
$$
\lambda(X^k+t_{k_m}H^{k_m})=\lm(X^k)+t_{k_m}\lambda'(X^k,H^{k_m})+o(t_{k_m}).
$$
By the convexity of $\theta$ and the fact that $\lm(Y^k)\in \sub \th(\lm(X^k))$, we have 
$$
\begin{aligned}
\frac{g(X^k+t_{k_m}H^{k_m})-g(X^k)}{t_{k_m}}&=\frac{\theta(\lambda(X^k)+t_{k_m}\lambda'(X^k,H^{k_m})+o(t_{k_m}))-g(X^k)}{t_{k_m}}\\
&\ge \la \lm(Y^k), \lambda'(X^k, H^{k_m})\ra+ o(t_{k_m})/t_{k_m}\\
&=\sum_{l=1}^{r}\sum_{t=1}^{z^l}\langle \Lambda(Y^k)_{\chi^l_t\chi^l_t},\Lambda((P^k)^{\top}H^{k_m}P^k)_{\chi^l_t\chi^l_t}\rangle+o(t_{k_m})/t_{k_m}.
\end{aligned}$$ 
It follows from \eqref{eq:finite} and \eqref{eq:finite2} that for each sufficiently large $k$ and $m$,  
\begin{equation*}\label{eq:cuiding22inv}
\frac{g(X^k+t_{k_m}H^{k_m})-g(X^k)}{t_{k_m}}\leq t_{k_m}(M+1)+\langle Y^k,H^{k_m}\rangle.
\end{equation*}
Combining the estimates above leads us to 
$$t_{k_m}(M+1)+\langle Y^k,H^{k_m}\rangle\geq\sum_{l=1}^{r}\sum_{t=1}^{z^l}\langle \Lambda(Y^k)_{\chi^l_t\chi^l_t},\Lambda((P^k)^{\top}H^{k_m}P^k)_{\chi^l_t\chi^l_t}\rangle+o(t_{k_m})/t_{k_m}.$$
Letting first $m\rightarrow\infty$  and  then $k\rightarrow\infty$ brings us to 
$$\langle Y,H\rangle\geq\sum_{l=1}^{r}\sum_{t=1}^{z^l}\langle \Lambda(Y)_{\chi^l_t\chi^l_t},\Lm\big(((PQ)^{\top}HPQ)_{\chi^l_t\chi^l_t}\big)\rangle=\langle\lambda(Y),w\rangle,$$
where $w$ comes from \eqref{eq:woutdom2}.
On the other hand, a direct application of Fan's inequality, coupled with $PQ=\widehat{P}\in{\cal O}^{n}(X)\cap {\cal O}^{n}(Y)$,  ensures that  
$$
\begin{aligned}
\langle Y,H\rangle &= \langle \Lm(Y),\Hat{P}H\Hat{P}\rangle= \sum_{l=1}^{r}\sum_{t=1}^{z^l}\langle \Lambda(Y)_{\chi^l_t\chi^l_t},\big((PQ)^{\top}HPQ\big)_{\chi^l_t\chi^l_t}\rangle \\
& \le  \sum_{l=1}^{r}\sum_{t=1}^{z^l}\langle \Lambda(Y)_{\chi^l_t\chi^l_t},\Lm\big(((PQ)^{\top}HPQ)_{\chi^l_t\chi^l_t}\big)\rangle=\langle\lambda(Y),w\rangle.
\end{aligned}
$$
Combining these and using Fan's inequality again, we conclude  that $\Lambda(Y)_{\chi^l_t\chi^l_t}$ and  $\big((PQ)^{\top}HPQ\big)_{\chi^l_t\chi^l_t}$  have a simultaneous ordered spectral decomposition. Using the same argument as the one in Remark~\ref{crit_charc} ensures that $((PQ)^{\top}HPQ)_{ij}=0$
for all $i,j\in\chi_t^l$  and all $i,j$ that are not in the same $\gamma_s^l$;
see Figure~\ref{fig2}(b) for an illustration of such a matrix $H$.
On the other hand, we deduce from  \eqref{eq:sigtgeq} and \eqref{eq:finite}  that 
$$
\begin{aligned}
&\infty>\sum\limits_{1\leq l\leq r}\sum\limits_{1\leq t<t'\leq z^l}\sum\limits_{i\in\chi_t^l}\sum\limits_{j\in\chi_{t'}^l}\frac{\lambda_i({Y}^k)-\lambda_j({Y}^k)}{\lambda_i(X^k)-\lambda_j(X^k)}(({P}^k)^{\top}H^k{P}^k)^2_{ij}\\
&\geq\sum\limits_{1\leq l\leq r}\sum\limits_{1\leq s<s'\leq u^l}\sum\limits_{1\leq t<t'\leq z^l}\sum\limits_{\substack{i\in\gamma_s^l,j\in\gamma_{s'}^l\\
i,j\notin\,\mbox{the same}\,\chi_t^l}}\frac{\lambda_i({Y}^k)-\lambda_j({Y}^k)}{\lambda_i(X^k)-\lambda_j(X^k)}(({P}^k)^{\top}H^k{P}^k)^2_{ij}. 
\end{aligned}$$
Passing to the limit implies that  $((PQ)^{\top}HPQ)_{ij}=0$ for all $i,j$ that are not in the same  $\chi_t^l$ and the same  $\gamma_s^l$; see Figure~\ref{fig2}(a) for an illustration of such a matrix $H$.
Combining these illustrates that 
for each $l \in\{1, \ldots, r\}$, the matrix $(PQ)_{\alpha^l}^{\top} H (PQ)_{\alpha^l}$ has a block diagonal representation in the form 
\begin{equation}\label{eq:verL4.73}
(PQ)_{\alpha^l}^{\top} H (PQ)_{\alpha^l}=\operatorname{Diag}\big((Q^{\top}P^{\top} H PQ)_{\gamma_1^l \gamma_1^l}, \ldots,(Q^{\top}P^{\top} H PQ)_{\gamma_{u^l}^l \gamma_{u^l}^l}\big),
\end{equation}
which confirms that   $H$ satisfies condition (a) in Lemma~\ref{lemma:charofaffcc}.

According to  \cite[Lemma 2.1]{hjs}, we have for any $k$ sufficiently large, 
that 
\begin{equation}\label{inc_ind}
    {\eta}(\lm(X),\lm(Y))\subset \iota_1^k\times\iota_2^k:= \iota_1(\lm(X^k))\times \iota_2(\lm(X^k))\subset \overline{\iota}_1\times\overline{\iota}_2:={\iota}_1(\lm(X))\times{\iota}_2(\lm(X)).
\end{equation} 
Recalling \eqref{eq:def_phi12}, one can directly  check that for all $\nu,\nu'\in\iota_1^k$, 
$\langle \lambda(X),a^\nu-a^{\nu'}\rangle=c_\nu-c_{\nu'}$ and $\langle \lambda(X^k),a^\nu-a^{\nu'}\rangle=c_\nu-c_{\nu'}$, leading us to 
$\langle \lambda(X^k)-\lambda(X),a^\nu-a^{\nu'}\rangle=0$. Similarly, we also have for all $\mu\in\iota_2^k$, $\langle \lambda(X^k)-\lambda(X),b^\mu\rangle=0$. Denote $w^k:=\lambda(X^k)-\lambda(X)$. It follows from Lemma \ref{lem:wIdforg} that    that for all $s\in{\cal E}^l$, $w^k_{\gamma_s^l}=\mu^l_s(k){\bf 1}$ for some $\mu^l_s(k)\in\R$. Similar to \eqref{eq:partonXpY2}, define the index sets $\beta^l_s(t)$, $t\in \{1,\dots,v^{l,s}\}$ with $v^{l,s}\in \N$, to further partition the index set $\gamma_s^l$ from Remark~\ref{part}(P3) based on $\lambda(X^k)$. By the definition of $w^k$ and the observation about $w^k_{\gamma_s^l}$ for all $s\in{\cal E}^l$,  we can conclude that  
$v^{l,s}=1$ for any $s\in{\cal E}^l$. 
This implies that for all $s\in{\cal E}^l$, $\gamma_s^l$ would not further be partitioned by $\lm(X^k)$. 
Since $H^k\in{\cal C}_g(X^k,Y^k)$ and since $Y^k\in\mbox{ri}\,\partial g(X^k)$, we know from Lemma \ref{lemma:characofcc} and Lemma \ref{crit_gqf} that for all $\nu,\nu'\in\iota_1^k$, $\langle \lambda'(X^k,H^k),a^\nu-a^{\nu'}\rangle=0$ and  all $\mu\in\iota_2^k$, $\langle \lambda'(X^k,H^k),b^\mu\rangle=0$. 
For any $l\in \{1,\ldots,r\}$, define the index set 
$${\cal F}^l:=\left\{t\in\{1,\dots,z^l\}\Bigg | \begin{array}{l}
\exists\, i, j \in \chi_t^l \text { such that }\left({a}^\nu\right)_i \neq\left({a}^\nu\right)_j \text { for some } \nu \in \iota_1^k,\\
\mbox{or} \;\left({b}^\mu\right)_i \neq\left({b}^\mu\right)_j \text { for some } \mu \in \iota_2^k
\end{array}\right\}.$$
While  ${\cal F}^l$ depends on $k$ as well,  we can assume by passing to a subsequence, if necessary,  that it remains constant for all $k$. 
A similar argument as that of Lemma \ref{lem:wIdforg} implies for all $t\in{\cal F}^l$ that  $(\lambda'(X^k,H^k))_{\chi^l_t}=\rho^l_t(k){\bf 1}_{|\chi^l_t|}$ for some $\rho^l_t(k)\in\R$. Passing to the limit and using the fact that $\lambda'(X^k,H^k)\to w$ with $w$ taken from \eqref{eq:woutdom2}, we have 
$$
\lm\big((PQ)^{\top}HPQ)_{\chi^l_t\chi^l_t}\big)=w_{\chi^l_t}=\rho^l_t{\bf 1}_{|\chi^l_t|}$$
for some $\rho^l_t\in\R$. This yields   $\big((PQ)^{\top}HPQ\big)_{\chi^l_t\chi^l_t}=\rho_t^lI_{|\chi^l_t|}$ for any $t\in{\cal F}^l$. For any $s\in{\cal E}^l$,  there exists  $t$ 
such that $\gamma_s^l\subset\chi_t^l$, since  $\gamma_s^l$ will not be further partitioned by $\lm(X^k)$ whenever  $s\in{\cal E}^l$. We claim that $t\in{\cal F}^l$. Indeed, if $s\in{\cal E}^l$,  
there exist $i,j\in\gamma_s^l$ such that $(a^\nu)_i\neq(a^{\nu})_j$ or  $(b^\mu)_i\neq(b^{\mu})_j$ for some $(\nu,\mu)\in{\eta}(\lambda(X),\lambda(Y))$. If the former holds, it results from \eqref{inc_ind}
that  there exist $i,j\in\gamma_s^l\subset\chi_t^l$ such that $(a^\nu)_i\neq(a^\nu)_j$. If the latter holds, a similar argument can be used.   By the definition of ${\cal F}^l$, we have   $t\in{\cal F}^l$. 
Therefore, for any $s\in{\cal E}^l$, there exists  $t\in{\cal F}^l$ such that $((PQ)^{\top}HPQ)_{\gamma_s^l\gamma_s^l}=\rho_t^lI_{|\gamma_s^l|}$. This confirms that 
for any $l \in\{1, \ldots, r\}$ and any $s \in \mathcal{E}^l$, there exists a scalar $\rho_s^l$ such that 
\begin{equation}\label{eq:verL4.72}(Q^{\top}P^{\top} H PQ)_{\gamma_s^l \gamma_s^l}=\rho_s^l I_{\left|\gamma_s^l\right|},
\end{equation}
proving that   $H$ satisfies condition (c) in Lemma~\ref{lemma:charofaffcc}.

A similar argument as that of \eqref{eq:S28P4}, coupled with  \eqref{eq:verL4.72}, demonstrates that  for all $(\nu,\mu)\in\eta(\lambda(X),\lambda(Y))$, $\langle w,a^\nu\rangle=\langle\mbox{diag}((PQ)^{\top}HPQ),a^\nu\rangle$ and $\langle w,b^\mu\rangle=\langle\mbox{diag}((PQ)^{\top}HPQ),b^\mu\rangle$, where $w$ is taken from \eqref{eq:woutdom2}. These and the fact that $w\in \mbox{aff}\,C_{\theta}(\lambda(X),\lambda(Y))$ yields 
\begin{equation*}\label{eq:verL4.71}
\mbox{diag}\big((PQ)^{\top}H(PQ)\big)\in\mbox{aff}\,C_{\theta}(\lambda(X),\lambda(Y)), \end{equation*}
which together with  \eqref{eq:verL4.72}, \eqref{eq:verL4.73}, and Lemma~\ref{lemma:charofaffcc} 
shows that $H\in{\rm aff}\,{\cal C}_{g}(X,Y)$.
Consequently, we deduce from Steps 1-3 that   
$$ \Xi_{X,Y}=\varUpsilon_{X,Y}+\delta_{\mbox{aff}\,{\cal C}_g(X,Y)}\in \quadr g(X,Y)$$
is the minimal quadratic bundle of $g$ at $X$ for $Y$,  which  completes the proof. 
\end{proof}

We should add here that 
 \cite[Example 3]{r23} also provides the explicit form of the quadratic bundles for the indicator function of the second-order cone. As shown in \cite[page 3]{SSun08}, the second-order cone can be written as a spectral function with representation \eqref{spec} under the Euclidean Jordan algebra. Thus, we can also obtain the result presented in \cite[Example 3]{r23} by our approach. This suggests  that our results can be extended to the polyhedral spectral function under Euclidean Jordan algebra.

We close this section by applying our major result in Theorem~\ref{prop:minquadbdcha} for important classes of spectral functions, enjoying the representation in \eqref{spec} with $\th$ being a polyhedral function. We begin by looking at the maximum eigenvalue function, which has important applications in optimal control. 
\begin{example}\label{exp:le}{\rm
 Consider the following largest eigenvalue optimization problem
\begin{equation}\label{opt-eq-eig1}
\begin{array}{cll}
\displaystyle\operatornamewithlimits{minimize}&\; \lambda_1(X)& \displaystyle\operatornamewithlimits{subject \; to }\;\;{{X}\in \S^n},
\end{array}
\end{equation}
where $\lambda_1(X)$ denotes the largest eigenvalue of a symmetric matrix $X$.
This corresponds to a special case of \eqref{spec}, where 
\[
\theta({x})=\max_{1\le i\le n}\left\{\langle {a}^i,{x}\rangle \right\},\quad {x}\in \R^n
\]
with ${a}^i$ being the unit vector whose $i$-th component is 1 and others are zero. It follows from Theorem~\ref{prop:minquadbdcha}, \eqref{eq:sigtsdp} and Lemma \ref{lemma:charofaffcc} that the minimal quadratic bundle of $\lambda_1(\cdot)$ takes the following form,
\begin{equation*}
    q(H)=2\sum_{i\in \mu}\sum_{j\in \omega}\frac{\lambda_i({Y})}{\lambda_i({X})-\lambda_j({X})}(P^{\top} HP)_{ij}^2+\delta_{{\rm aff}\,{\cal C}_{\lambda_1}(X,{Y})}(H), \quad H\in\S^n,
\end{equation*}
with
\[
H\in {\rm aff}\,({{\cal C}_{\lambda_1}({X},{Y})}) \;\Longleftrightarrow \;   P^{\top}_{\alpha^1} HP_{\alpha^1}=\begin{bmatrix}
	\rho\, I_{|\tau|} & 0 \\ 
	0 & P^\top_{\sigma} HP_{\sigma}
\end{bmatrix}\;\; \mbox{for some $\rho \in \R$},
\]
where $P\in{\cal O}^n(X)\cap{\cal O}^n(Y)$,  $\alpha^1=\left\{1\le i\le n\mid \lambda_i({X})={v}_1 \right\}$, $\tau:=\{i\in\alpha^1\mid \lambda_i({Y})>0\}$, $\sigma:=\{i\in\alpha^1\mid \lambda_i({Y})=0\}$, $\omega:= \bigcup_{l=2}^r\alpha^l$. Moreover, we know from Lemma \ref{suff_k} that $K(X,Y)\neq\varnothing$.} It is worth noting that the above explicit formulas of the minimal quadratic bundle (called previously the sigma term)   and  ${\rm aff}\,({{\cal C}_{\lambda_1}({X},{Y})})$ were also obtained in \cite[(67) and (66)]{cuiding} without appealing the theory of the quadratic bundle, used in our approach. 
\end{example}

\begin{example}\label{exp:sdp}{\rm
The $n$-dimensional positive semidefinite (SDP) constraint $X\in{\S}^n_+$ also corresponds to a special case of \eqref{spec}, where 
\[
\theta({x})
=\delta_{\R_+^n}({x})\quad {\rm and}\quad  {\rm dom}\,\theta=\{{x}\in\R^n\mid\max_{1\le i\le n}\left\{\langle {b}^i,{x}\rangle\leq0 \right\}, 
\]
with ${b}^i$ being the unit vector whose $i$-th component is $-1$ and others are zero. Thus,   the minimal quadratic bundle for SDP at $Y\in N_{\S_+^n}(X)$ has a representation of the form  
\begin{eqnarray}
q(H)&=&2\sum\limits_{1\leq l< l'\leq r}\sum\limits_{i\in\alpha^l}\sum\limits_{j\in\alpha^{l'}}\frac{\lambda_i({Y})-\lambda_j({Y})}{\lambda_i(X)-\lambda_j(X)}\big(({P}^{\top}HP)_{ij}\big)^2+\delta_{{\rm aff}\,{\cal C}_{g}(X,{Y})}(H)\nonumber\\
&=&2\sum_{i\in\alpha,j\in\gamma}\frac{-\lambda_j(Y)}{\lambda_i(X)}(P^{\top} HP)_{ij}^2+\delta_{{\rm aff}({\cal C}_{\S^n_+}(X,Y))}(H),\quad H\in\S^n,\label{eq:mqbsdp} 
\end{eqnarray}
where $P\in{\cal O}^n(X)\cap{\cal O}^n(Y)$, $\alpha=\{i\mid \lambda(X+Y)>0\}$, $\gamma=\{i\mid \lambda(X+Y)<0\}$, $\beta=\{i\mid \lambda(X+Y)=0\}$.  To justify the second equality above, observe that  if $i\in\alpha\cup\beta$, we have $\lambda_i(Y)=0$, which implies  $\frac{\lambda_i({Y})-\lambda_j({Y})}{\lambda_i(X)-\lambda_j(X)}=0$ for all $i\in\alpha, j\in\alpha\cup\beta$. Thus for all $i\in\alpha$ and $j\in\gamma$, we have $\frac{\lambda_i({Y})-\lambda_j({Y})}{\lambda_i(X)-\lambda_j(X)}=\frac{-\lambda_j(Y)}{\lambda_i(X)}$. Moreover,  the affine hull of critical cone of $\th$ at $X$ for $Y$ can be obtained from Lemma \ref{lemma:charofaffcc} as
\begin{equation*}\label{eq:affcsdp}
	{\rm aff}\,{\cal C}_{\S_+^n}({X},{Y})=\{H\in\S^n\mid P^{\top}_{\beta}HP_{\gamma}=0,\;P^{\top}_{\gamma}HP_{\gamma}=0\}.
\end{equation*}
 Indeed, it follows from Lemma \ref{lemma:characofcc}(a) that $P^{\top}_{\beta}HP_{\gamma}=0$ and for all $i,j\in\gamma$ such that   $\lambda_i(X+Y)\neq\lambda_j(X+Y)$, one has $P^{\top}_iHP_j=0$.  It can be checked from the explicit form of $b^i$ that for all $l\in\{1,\dots,r\}$ and $s\in\{1,\dots,u^l\}$, we have $s\in{\cal E}^l$, which implies $P^{\top}_{\gamma_s^l}HP_{\gamma_s^l}=\rho_s^lI$ for some $\rho_s^l\in\R$.  
Since $\eta_2(\lambda(X),\lambda(Y))=\gamma$ and $\iota_2(\lambda(X))=\gamma\cup\beta$, we know from \eqref{eq:affcha=02} that $\big(\diag(P^{\top}HP)\big)_{\gamma}=0$. Combining the aforementioned discussion together tells us that   $P^{\top}_{\gamma}HP_{\gamma}=0$. Moreover, we know from Lemma \ref{suff_k} that $K(X,Y)\neq\varnothing$.  
Note that the explicit forms of the corresponding affine critical cone and the minimal quadratic bundle were obtain  in \cite[equation (17)-(20) and Lemma 3.1]{Sun06} via a  different approach. We should add the latter was called the sigma term in \cite{Sun06}. 
}
\end{example}

\begin{example}\label{exp:combine}{\rm
The fastest mixing Markov chain (FMMC) problem was first proposed and studied in \cite{BDParrilo, BDSXiao}.  Suppose ${\cal G}=\{{\cal V},{\cal E}\}$ is a connected graph with the vertex set ${\cal V}=\{1,\dots,n\}$ and edge set ${\cal E}\subset{\cal V}\times{\cal V}$. 
It was shown in \cite[(1.1)]{BDXiao} that the FMMC problem can be formulated as   
 \begin{equation*}
	\begin{array}{cl}
		\displaystyle\operatornamewithlimits{minimize} & \max\{\lambda_2(X),-\lambda_n(X)\} \\ [3pt]
		{\rm subject \;to} & X\geq0,\; X{\bf 1}={\bf 1},\\
		& X_{ij}=0, (i,j)\in{\cal E}, X\in \S^n, 
	\end{array}
\end{equation*}
where ${\bf 1}\in\R^n$ denotes the vector whose components are all 1. 
It can be checked directly that its objective function can be written as the Ky Fan 2-norm $\|\cdot\|_{(2)}$ in $\S^n$.   Meanwhile, the semidefinite constraint is of great importance in machine learning and statistic applications such as clustering and community detection \cite{ZS,YO,DH}. 
If we further suppose $X\in\S^n_+$, by using the eigenvalue property of doubly stochastic matrices, namely square matrices with all nonnegative entries and each row and column summing to one,  yields $\lambda_1(X)=1$. 
Therefore, the FMMC with the SDP constraint can be written in the  form
\begin{equation}\label{eq:probFMMC}
	\begin{array}{cl}
		\displaystyle\operatornamewithlimits{minimize} & \lambda_1(X)+\lambda_2(X) +\dd_{\S^n_+}(X) \\ [3pt]
		{\rm subject\; to} & X\geq0,\; X{\bf 1}=1,\\
		& X_{ij}=0, (i,j)\in{\cal E}. 
	\end{array}
\end{equation}
The objective function of this problem, namely $g(\cdot):=\delta_{\S^n_+}(\cdot)+\lambda_1(\cdot)+\lambda_2(\cdot)$, has the spectral representation in \eqref{spec} with  $\theta({x})= \th_1(x)+\th_2(x)$ and 
$$
\th_1(x):=\max_{1\leq i<j\leq n}\{\langle {e}^{ij},{x}\rangle\}\quad \mbox{and}\quad \th_2(x):=\delta_{\R_+^n}(x)
$$ for any $x\in \R^n$, where the $i$-th and $j$-th components of $e^{ij}$ are 1 and  the others are 0.

Suppose $Y\in\partial\theta(\lambda(X))$. We obtain the explicit form of the minimal quadratic bundle of $g$ via Theorem \ref{prop:minquadbdcha}. In fact, suppose $\lambda(Y)=y_1+y_2$ with $y_1\in\theta_1(\lambda(X))$, $y_2\in\theta_2(\lambda(X))$. Note that  finding $y_1$ and $y_2$ requires solving an inequality system. As   mentioned in Remark \ref{remark:invariant-y1-y2}, the minimal quadratic bundle is also invariant under $y_1$ and $y_2$. Let $\alpha^l$, $l=1,\ldots,r$ be the index sets defined by \eqref{eq:def-alpha} and $P\in{\bf O}^n({X}) \cap {\bf O}^n({Y})$. For each $l \in\{1, \ldots, r\}$, recall the index sets $\{\gamma_s^l\}_{s=1}^{u^l}$ given by \eqref{eq:partonY}. 
It follows that 
\[
q(H)=\varUpsilon_{X,Y}(H)+\delta_{\mbox{aff}\,{\cal C}_g(X,Y)}(H),\quad H\in\S^n
\]
with
$$\varUpsilon_{X,Y}(H)=2\sum\limits_{1\leq p< p'\leq r}\sum\limits_{i\in\alpha^p}\sum\limits_{j\in\alpha^{p'}}\frac{\lambda_i({Y})-\lambda_j({Y})}{\lambda_i(X)-\lambda_j(X)}(P^{\top}HP)^2_{ij}.$$ 
Then we only need to characterize $\mbox{aff}\,{\cal C}_g(X,Y)$ and verify $K(X,Y)\neq\varnothing$. It is worth mentioning that the assumption $K(X,Y)\neq\varnothing$ can be justified  by using Lemma \ref{suff_k}. However, in this example, we verify it in a more direct way by finding a vector  $d$ in $K(X,Y)$. Consider the following two cases for all $X$ in the feasible set of \eqref{eq:probFMMC}. 

{\bf Case 1:} $|\alpha^1|\geq2$. We have 
$(y_1)_{i}\geq0$ if $i\in\alpha^1$ and $(y_1)_{i}=0$ otherwise; $(y_2)_{i}=0$ if $i\in\alpha^1$ and $(y_2)_{i}\leq0$ otherwise. It can be checked that $y_1$ and $y_2$ are arranged in a non-increasing order. 
Denote $\omega=\{i\mid (y_2)_i\neq0\}$ and $\omega'=\{i\notin\alpha^1\mid (y_2)_i=0\}$. 
We can pick $d_i={\bf 1}$ such that for all $(y_1)_i\neq0$, $d_i=0$ for all $(y_1)_i=0$, and $d_{\omega}=0$, $d_{\omega'}={\bf 1}$, which indicates $d\in K(X,Y)\neq\varnothing$. 
Denote
\begin{equation*}
	\tau:=\{i\in\alpha^1\mid (y_1)_i>0\}  \quad {\rm and}\quad \sigma:=\{i\in\alpha^1\mid (y_1)_i=0\}.
\end{equation*}
Then, we have $H\in {\rm aff}\,{\cal C}_{\theta\circ\lambda}({X},{Y})$ if and only if the following three conditions are satisfied:
\begin{itemize}
    \item[(i)] $P^{\top}_{\alpha^1} HP_{\alpha^1}=\begin{bmatrix}
	\rho\, I_{|\tau|} & 0 \\ 
	0 & P^\top_{\sigma} HP_{\sigma}
\end{bmatrix}$ for some $\rho \in \R$;
    \item[(ii)] for each $l \in\{2, \ldots, r\}$, $P_{\alpha^l}^{\top} H P_{\alpha^l}=\operatorname{Diag}\big((P^{\top} H P)_{\gamma_1^l \gamma_1^l}, \ldots,(P^{\top} H P)_{\gamma_{u^l}^l \gamma_{u^l}^l}\big)$;
    \item[(iii)] $P^{\top}_{\omega} HP_{\omega}=0$ with $\omega=\{i\mid (y_2)_i\neq0\}$. 
\end{itemize}

{\bf Case 2:} $|\alpha^1|=1$. 
It follows that 
\begin{equation*}
\left\{\begin{array}{ll}
(y_1)_1=1, & \\[3pt]
	(y_1)_i\in[0,1] &\forall\,i\in \alpha^2\quad {\rm and}\quad   \displaystyle\sum_{i\in\alpha^2}(y_1)_i=1,   \\ [3pt]
	(y_1)_i=0 &\forall\,i\in \alpha^l,\quad  l=3,\ldots,r.
\end{array} \right.
\end{equation*}
Denote $\tau:=\{j\in\alpha^2\mid (y_1)_j>0\}$, $\sigma:=\{j\in\alpha^2\mid (y_1)_j=0\}$,
$\zeta=\{i\mid\lambda_i(X)=0,\ (y_2)_i<0\}$ and $\zeta':=\{i\mid\lambda_i(X)=0,\ (y_2)_i=0\}$.
If $\lambda_2(X)>0$, we have 
$(y_1)_{i}\geq0$ when $i\in\alpha^1\cup\alpha^2$ and $(y_1)_{i}=0$ otherwise; $(y_2)_{i}=0$ when $i\in\alpha^1\cup\alpha^2$ and $(y_2)_{i}\leq0$ otherwise. It can be checked that $y_1$ and $y_2$ are arranged in a non-increasing order. 
We can pick $d$ such that $d_{\mu}=0$, $d_{\nu}=-{\bf 1}$ and $d_{\zeta}=0$, $d_{\zeta'}={\bf 1}$. Therefore, 
 $d\in K(X,Y)$, which yields $K(X,Y)\neq\varnothing$. If $\lambda_2(X)=0$, denote $\alpha=\{i\mid\lambda_i(Y)>0,i\neq1\}$, $\beta=\{i\mid\lambda_i(Y)=0\}$ and $\gamma=\{i\mid\lambda_i(Y)<0\}$. In particular, we can pick 
 $$y_1=\big(\displaystyle\sum_{i\in\alpha}\lambda_i(Y),\lambda_{\alpha}(Y),0,\dots,0\big)^\top +\sum_{i=1}^n\frac{1-\sum_{i\in\alpha}\lambda_i(Y)}{n}e^{1i}$$ 
 and $y_2=\lambda(Y)-y_1$, which implies for all $(\nu,\mu)\in\iota_1(\lambda(X))\times \iota_2(\lambda(X))$ that $(\nu,\mu)\in\eta(\lm(X),\lm(Y))$. 
Then, we have $H\in {\rm aff}\,{\cal C}_{\theta\circ\lambda}({X},{Y})$ if and only if the following three conditions are satisfied:
\begin{itemize}
    \item[(i)]$P^{\top}_{\alpha^2} HP_{\alpha^2}=\begin{bmatrix}
	\rho\, I_{|\tau|} & 0 \\ 
	0 & P^\top_{\sigma} HP_{\sigma}
\end{bmatrix}$ for some $\rho \in \R$;

    \item[(ii)] for each $l \in\{3, \ldots, r\}$, $P_{\alpha^l}^{\top} H P_{\alpha^l}=\operatorname{Diag}\big((P^{\top} H P)_{\gamma_1^l \gamma_1^l}, \ldots,(P^{\top} H P)_{\gamma_{u^l}^l \gamma_{u^l}^l}\big)$; 
    \item[(iii)] $P^{\top}_{\omega} HP_{\omega}=0$ with $\omega=\{i\mid (y_2)_i\neq0\}$. 
\end{itemize}
}
\end{example}

\section{Tilt-Stability in Matrix Optimization Problems}\label{tilt-matrix}

In this section, we aim  to study tilt-stable local minimizers of the composite optimization problem 
\begin{equation}\label{comp2}
\mbox{minimize}\;\;f(X):=\ph(X)+g(X)\quad \mbox{subject to}\;\; X\in \S^n.
\end{equation}
where $\ph:\S^n\to \R$ is ${\cal C}^2$-smooth, and $g:\S^n\to \oR$ is a convex  spectral function
with representation \eqref{spec}.  Note that the function $f$ from  \eqref{comp2} is  prox-regular and subdifferentially continuous due to  \cite[Proposition 2.2]{LPR2000}. 
Below, we will provide a characterization of a tilt-stable local minimizer of this problem.
by applying our characterization of such minimizers from Corollary~\ref{tilt_minimal}.

\begin{Theorem}\label{thm:tiltequiv} Assume that $\overline{X}\in \S^n$ with $0\in \sub f(\overline{X})$, where $f$ is taken from \eqref{comp2}. Suppose that $K(\overline{X},\overline{Y})\neq\varnothing$ with $\overline{Y}=-\nabla\varphi(\overline{X})\in \S^n$.  Then the following properties are equivalent. 

\begin{itemize}[noitemsep]
\item [\rm{(a)}] $\overline{X}$ is a tilt-stable local minimizer of $f$.
\item [\rm{(b)}] The strong second-order sufficient condition (SSOSC)  
$$
\langle \nabla^2\ph(\overline{X})H,H\rangle+\varUpsilon_{\overline{X},-\nabla\varphi(\overline{X})}(H)>0 \quad\mbox{for all}\;\;  H\in{\rm aff}\,{\cal C}_g(\overline{X},-\nabla\varphi(\overline{X}))\setminus \{0\}
$$
holds. 
\end{itemize}
Moreover, both conditions above yield the kernel condition 
\begin{equation}\label{ker_cond}
\ker \nabla^2\ph(\overline{X})\cap \ker q=\{0\},
\end{equation}
where  $\ker q=\big\{H\in {\rm aff}\,{\cal C}_g(\overline{X},-\nabla\varphi(\overline{X}))\big|\; \varUpsilon_{\overline{X},-\nabla\varphi(\overline{X})}(H)=0\big\}.$ If, in addition,  $\ph$ is convex, then \eqref{ker_cond} and the conditions in {\rm(}a{\rm)} and {\rm(}b{\rm)} are equivalent. Furthermore, $\ker q$ can be equivalently described as 
$$\left\{H\in\S^n\Bigg | \begin{array}{l}
(P^{\top}HP)_{ij}=0, \forall i,j\;\mbox{such that}\;\lambda_i(\overline{X})\neq\lambda_j(\overline{X}), \lambda_i(-\nabla\ph(\overline{X}))\neq\lambda_j(-\nabla\ph(\overline{X})),\\
(P^{\top}HP)_{\alpha^l\alpha^l}={\rm diag}((P^{\top}HP)_{\gamma^l_1\gamma^l_1},\dots,(P^{\top}HP)_{\gamma^l_{u^l}\gamma^l_{u^l}}),\\
\begin{array}{l}
\langle \diag(P^{\top}HP), a^\nu\rangle=\langle \diag(P^{\top}HP), a^{\nu'}\rangle,\\ 
\langle \diag(P^{\top}HP), b^\mu\rangle=0,
\end{array}
\ (\nu,\mu), (\nu',\mu')\in\eta(\lambda(X),\lambda(Y)), \\
\mbox{for each}\; l\in\{1,...,r\}\;\mbox{and}\; k\in{\cal E}^l, \exists\,\rho_k^l\in\R\;\mbox{such that}\;(P^{\top}HP)_{\gamma_k^l\gamma_k^l}=\rho_k^lI_{|\gamma_k^l|}
\end{array}\right\},$$
where $P\in{\bf O}^n(\overline{X})\cap{\bf O}^n(-\nabla\ph(\overline{X}))$,  ${\cal E}^l$ and $\gamma_k^l$ are given in \eqref{eq:Eldiff} and \eqref{eq:partonY}. 
\end{Theorem}
\begin{proof}
Employing the sum rule for the epi-convergence from \cite[Theorem~7.46]{rw}, one can conclude directly that 
\begin{equation}\label{quadc}
\quadr f(\overline{X},0)= \big\{h+ q|\; q\in \quadr g(\overline{X},-\nabla\varphi(\overline{X}))\big\}. 
\end{equation}
The equivalence of (a) and (b) immediately results from Corollary~\ref{tilt_minimal}, Theorems~\ref{prop:minquadbdcha} and \eqref{quadc}. We proceed by showing that (b) yields  (c). Set $q:=\varUpsilon_{\overline{X},-\nabla\ph(\overline{X})}+\dd_{{\rm aff}\,{\cal C}_g(\overline{X},-\nabla\ph(\overline{X}))}$
and observe that the SSOSC in (b) amounts to 
\begin{equation}\label{ssosc3}
 \langle \nabla^2\ph(\overline{X})H,H\rangle+ q(H)>0\quad \mbox{for all}\;\; H\in \S^n\setminus\{0\}.   
\end{equation}
Moreover, we conclude from Proposition~\ref{quad_minimal} that $q\in \quadr g(\overline{X},-\nabla\ph(\overline{X}))$. By definition, we find $(X^k,Y^k)\in \gph \sub g$ such that $\d^2g(X^k,Y^k)\xrightarrow{e} q$ and $\d^2g(X^k,Y^k)$ is a GQF for each $k$. It follows from convexity of $g$ that $\d^2g(X^k,Y^k)(H)\ge 0$ for all $H\in \S^n$. Moreover, $\d^2g(X^k,Y^k)$ is a convex function for each $k$. Combining these tells us that  $q$ is a convex function on $\S^n$ and $q(H)\ge 0$ for all $H\in \S^n$.  This demonstrates that \eqref{ssosc3} implies the kernel condition in \eqref{ker_cond}. Assume now that $\ph$ is convex. We are going to argue that (c) implies (b). To this end, it is easy to deduce from the convexity of $\ph$ that the inequality in \eqref{ssosc3} is automatically satisfied when the strict inequality `$>$' therein is replaced with `$\ge$'. Suppose to the contrary that there is 
$H\in \S^n\setminus\{0\}$ for which we have $ \langle \nabla^2\ph(\overline{X})H,H\rangle+ q(H)=0$. Since $q(H)\ge 0$
and $\nabla^2\ph(\overline{X})H,H\rangle\ge 0$, we arrive at
$H\in \ker \nabla^2\ph(\overline{X})\cap \ker q$, a contradiction. 
This confirms that in the presence of the convexity of $\ph$
the conditions in (b) and (c) are equivalent. 

It remains to justify the equivalent description of $\ker q$. To this end, we know from \eqref{eq:sigtsdp} that 
\begin{eqnarray*}
\ker q&=&\big\{H\in \mbox{aff}\,{\cal C}_g(\overline{X},-\nabla\ph(\overline{X}))\big|\; \varUpsilon_{\overline{X},-\nabla\ph(\overline{X})}(H)=0\big\}\\
&=&\big\{H\in \mbox{aff}\,{\cal C}_g(\overline{X},-\nabla\ph(\overline{X}))\big|\; \sum\limits_{1\leq p< p'\leq r}\sum\limits_{i\in\alpha^p}\sum\limits_{j\in\alpha^{p'}}\frac{\lambda_i(-\nabla\ph(\overline{X}))-\lambda_j(-\nabla\ph(\overline{X}))}{\lambda_i(\overline{X})-\lambda_j(\overline{X})}(P^{\top}HP)^2_{ij}=0\big\}.
\end{eqnarray*}
Since for any $i,j$, the inequality 
$$\sum\limits_{1\leq p< p'\leq r}\sum\limits_{i\in\alpha^p}\sum\limits_{j\in\alpha^{p'}}\frac{\lambda_i(-\nabla\ph(\overline{X}))-\lambda_j(-\nabla\ph(\overline{X}))}{\lambda_i(\overline{X})-\lambda_j(\overline{X})}(P^{\top}HP)^2_{ij}\geq0
$$
holds, 
we must have $(P^{\top}HP)_{ij}=0$ for any  $ i,j$ such that $\lambda_i(\overline{X})\neq\lambda_j(\overline{X})$ and $\lambda_i(-\nabla\ph(\overline{X}))\neq\lambda_j(-\nabla\ph(\overline{X}))$. 
This ensures the claimed representation of $\ker q$ and hence completes the proof.
\end{proof}

\begin{Remark}{\rm
The set $\ker q$, defined in Theorem~\ref{thm:tiltequiv}, 
can have a more explicit form for certain concrete examples by applying the explicit form of the minimal quadratic bundle. 
For the SDP framework from Example \ref{exp:sdp}, $\ker q$ can be  simplified as $\ker q=\big\{H\in \S^n\big|\; (P^{\top}HP)_{\gamma}=0\big\}$ by directly applying \eqref{eq:mqbsdp}. For the largest eigenvalue optimization problem from Example \ref{exp:le}, it is not hard to see that  
$$\ker q=\big\{H\in\S^n\mid (P^{\top}HP)_{\tau\omega}=0, P^{\top}_{\alpha^1} HP_{\alpha^1}=\begin{bmatrix}
	\rho\, I_{|\tau|} & 0 \\ 
	0 & P^\top_{\sigma} HP_{\sigma}
\end{bmatrix}\;\; \mbox{for some $\rho \in \R$}\big\}.$$
}
\end{Remark}

We close this section by comparing  Theorem \ref{thm:tiltequiv} with the recent results established in \cite{LPS24, OA}.
A characterization of tilt-stable local minimizers was achieved  in \cite[Theorem~4.1]{OA}  under three major assumptions (B1)-(B3). Assumption (B1) therein requires in the setting of Theorem~\ref{thm:tiltequiv} that $f$ be ${\cal C}^2$-decomposable in the sense of \cite[Definition~3.1]{OA}, which can be ensured when the spectral function $g$ in \eqref{comp2} enjoys the same property. We do not know whether or not our spectral function $g$ is ${\cal C}^2$-decomposable. Establishing such a result does not seem to be an easy task and does require a rather nontrivial approach such as the one used in \cite[Proposition~3.2]{cdz17}, where it was shown that the spectral function $g$ in \eqref{spec} is ${\cal C}^2$-cone reducible, which means that the epigraphical set $\mbox{epi}\, g$ is a ${\cal C}^2$-cone reducible provided that $\th$ enjoys the same property. Applying \cite[Proposition~3.2]{cdz17} tells us that $\delta_{\ss\epi\, g}$ is ${\cal C}^2$-decomposable. Whether or not this can be used to show that $g$ in \eqref{comp2} enjoys the same property is unclear to us at this time. The authors in \cite{LPS24} studied tilt-stable local minimizers of the convex Ky-Fan $k$-norm regularized problem using the concept of second subderivative and without appealing to the concept of the quadratic bundle. The equivalence of (a) and \eqref{ker_cond} in Theorem \ref{thm:tiltequiv} resembles a similar characterization in \cite[Theorem~4.1]{LPS24}. Finally, we should add that \cite[Theorem 7.9]{go} also provides a characterization of tilt stability using the concept of the subspace containing derivative. While the latter concept is related to the quadratic bundles, as shown in Proposition~\ref{quadn}, the advantage of our approach is that we do not have to calculate all quadratic bundles in our setting and the minimal one can be used for our goal.

\section{Conclusion}

In this paper, we establish a novel characterization of tilt stability through the framework of quadratic bundles, creating a significant theoretical bridge in optimization theory. Our analysis derives the explicit form of the minimal quadratic bundle for spectral functions and demonstrates its equivalence to the strong second-order sufficient condition, thereby advancing our understanding of matrix optimization problems. This theoretical unification strengthens the foundation for analyzing optimization problems through multiplier methods while deepening our insights into perturbation properties. Looking forward, several challenging questions emerge: extending these results to general spectral functions and studying the relationship between SSOSC and tilt stability without the nondegeneracy condition remain important problems for future research.


\begin{thebibliography}{10}

\bibitem{br}
{\sc C. G. E. Boender and K. A. H. G. Rinnooy}, {\em Bayesian stopping rules for multistart global optimization methods}, Math. Program.  37 (1987), 59--80.

\bibitem{bs} {\sc J. F. Bonnans and A. Shapiro}, {\em Perturbation Analysis of Optimization Problems}, Springer, New York, 2000.

\bibitem{BDParrilo} {\sc S. \ Boyd, P. \ Diaconis, P. A. \ Parrilo, and L. \ Xiao}, {\em Fastest mixing Markov chain on graphs with symmetries}, SIAM J. Optim. 20 (2009), 792–819.

\bibitem{BDSXiao} {\sc S. \ Boyd, P. \ Diaconis, J. \ Sun, and L. \ Xiao}, {\em Fastest mixing Markov chain on a path}, The American Mathematical Monthly, 113 (2006), 70–74.

\bibitem{BDXiao} {\sc S. \ Boyd, P. \ Diaconis and L. \ Xiao}, {\em Fastest mixing Markov chain on a graph}, SIAM Review, 44 (2004), 667–689.

\bibitem{CLM} {\sc
E. J. Cand$\grave{\rm e}$s, X. Li, Y. Ma and J. Wright}, {\em Robust principal component analysis?}, J. ACM,  58 (2011), 1--37.

\bibitem{CP} {\sc
E. J. Cand$\grave{\rm e}$s, Y. Plan}, {\em Matrix completion with noise},  Proceedings of the IEEE, 98(2010), 925--936.

\bibitem{chn} {\sc N. H. Chieu, L. V. Hien and T. T. A. Nghia}, {\em Characterization of tilt stability via subgradient graphical derivative with application to nonlinear programming}, SIAM J. Optim. 28
(2018),  2246--2273.

\bibitem{chn2} {\sc  N. H. Chieu, L. V. Hien, T. T. A. Nghia, and H. A. Tuan}, {\em Quadratic growth and strong metric subregularity of the subdifferential via subgradient graphical derivative}, SIAM J. Optim. 31  (2021), 545--568.

\bibitem{cdz17} {\sc Y. Cui, C. Ding and X.Y. Zhao}, {\em Quadratic growth conditions for convex matrix optimization problems associated with spectral functions}, SIAM J. Optim., 27 (2017), 2332–2355.

\bibitem{CC} {\sc T. F. Coleman and A. R. Conn}, {\em On the local convergence of a quasi-Newton method for the nonlinear programming problem}, SIAM J. Numer. Anal. 21(1984), 755--769.

\bibitem{cuiding} {\sc Y.\ Cui and C.\ Ding}, {\em Nonsmooth composite matrix optimization: strong regularity, constraint nondegeneracy and beyond}, arXiv: 1907.13253.

\bibitem{DH} {\sc A. 
Douik, B. Hassibi}, {\em Low-rank Riemannian optimization on positive semidefinite stochastic matrices with applications to graph clustering}, PMLR, (2018), 1299--1308.

\bibitem{DLewis} {\sc D. \ Drusvyatskiy and A. S. \ Lewis}, 
{\em Tilt Stability, Uniform Quadratic Growth, and Strong Metric Regularity of the Subdifferential}, SIAM J. Optim.  23 (2013), {256--267}.



\bibitem{Fan49} {\sc K. \ Fan}, 
{\em On a theorem of Weyl concerning eigenvalues of linear transformations}, Proceedings of the National Academy of Sciences of U.S.A. 35 (1949), 652--655.

\bibitem{go}  {\sc H. Gfrerer and J. V. Outrata}, {\em On (local) analysis of multifunctions via subspaces contained in graphs of generalized derivatives}, J. Math. Anal. Appl. 508 (2022), 125895.

\bibitem{Ghahramani} {\sc Z. Ghahramani}, {\em  Unsupervised learning, Summer school on machine learning}. Berlin, Heidelberg: Springer Berlin Heidelberg, 2003: 72-112.

\bibitem{hjs} {\sc  N. T. V. Hang, W. Jung}, and {\sc  E. Sarabi}, {\em Role of subgradients in variational analysis of polyhedral functions},  J. Optim.
Theory Appl. 200 (2024), 1160–1192.


\bibitem{hs23} {\sc  N. T. V. Hang}, and {\sc  E. Sarabi}, {\em Smoothness of Subgradient Mappings and Its Applications in Parametric Optimization}, {arXiv:2311.06026}.


\bibitem{kmp} {\sc P. D. Khanh, B. S. Mordukhovich, V. T. Phat}, {\em Variational convexity of functions and variational sufficiency in optimization}, SIAM J. Optim.  33 (2023), {1121--1158}.

\bibitem{kmpv} {\sc P. D. Khanh, B. S. Mordukhovich, V. T. Phat, L. D. Viet}, {\em Characterizations of Variational Convexity and Tilt
Stability via Quadratic Bundles}, arXiv:2501.04629v1.

\bibitem{Higham} {\sc N. J. Higham}, {\em Accuracy and stability of numerical algorithms}, Society for industrial and applied mathematics, 2002.

\bibitem{LPT} {\sc 
G. Landi, E. L. Piccolomini, I. Tomba}, {\em A stopping criterion for iterative regularization methods}, Appl. Numer. Math. 106 (2016), 53--68.

\bibitem{Lancaster}
{\sc P. \ Lancaster}. {\em On eigenvalues of matrices dependent on a parameter}. Numer. Math. 6 (1964), 377--387.

\bibitem{Lewis96} {\sc A. S. \
Lewis}, {\em Convex analysis on the hermitian matrices}, SIAM J. Optim.  6(1996), 164--177.

\bibitem{LZhang13} {\sc A. S. \ Lewis and S. \ Zhang}, {\em Partial Smoothness, Tilt Stability, and Generalized Hessians}, SIAM J. Optim.  23(2013), 74--94.
 
\bibitem{LPR2000}
{\sc A. B. Levy, R. A. Poliquin and R. T. Rockafellar}, {\em Stability of locally optimal solutions}, SIAM J. Optim.  10 (2000), 580--604.


\bibitem{LPS24} {\sc Y. L. Liu, S. H. Pan and W. Song}, {\em Tilt stability of Ky-Fan $k$-norm composite optimization}, arxiv: 2406.10945. 

\bibitem{ms20} {\sc A. \ Mohammadi and M. E. \ Sarabi},  {\em Twice epi-differentiability of extended-real-valued functions with applications in
composite optimization}. SIAM J. Optim. 30 (2020), 2379--2409.

\bibitem{MSarabi23}
{\sc A.  Mohammadi and E.  Sarabi}, {\em Parabolic Regularity of Spectral Functions}, Math. Oper. Res. (2024) DOI:10.1287/moor.2023.0010.



\bibitem{mn14} {\sc B. S. Mordukhovich and T. T. A. Nghia}, {\em Second-order characterizations of tilt stability with applications to nonlinear programming},  Math. Program. 149 (2015), 83--104.

\bibitem{MNR15} {\sc B. S. Mordukhovich,  T. T. A. Nghia and R. T. Rockafellar}, {\em Full stability in finite-dimensional optimization},  {  Math. Oper. Res.} 40 (2015), 226--252.

\bibitem{MSarabi16} {\sc B. S. \ Mordukhovich and M. E. \ Sarabi}, {\em 
Generalized differentiation of piecewise linear functions in second-order variational analysis}, Nonlinear Anal. 132 (2016), 240--273. 

\bibitem{MSarabi18} {\sc  B. S. \ Mordukhovich and M. E. \ Sarabi}, {\em Critical multipliers in variational systems via second-order generalized differentiation},  Math. Program. 169 (2018), 605--64.

\bibitem{MSarabi21} {\sc  B. S. \ Mordukhovich and M. E. \ Sarabi}, {\em Generalized Newton algorithms for tilt-stable minimizers in nonsmooth optimization}, SIAM J. Optim.   31(2021), 1184--1214.

\bibitem{OA}{\sc W. \ Ouyang and A. \ Milzarek}, {\em Variational Properties of Decomposable Functions. Part II: Strong Second-Order Theory}. arXiv preprint arXiv:2311.07276.


\bibitem{Nghia} {\sc T. T. A. \ Nghia}, {\em Geometric characterizations of Lipschitz stability for convex optimization problems}, arXiv:2402.05215, 2024.


\bibitem{pr96} {\sc R. A. Poliquin and R. T. Rockafellar}, {\em Prox-regular functions in variational analysis } Trans. Amer. Math. Soc. 348 (1996), 1805--1838.

\bibitem{pr98} {\sc R. A. Poliquin and R. T. Rockafellar}, {\em Tilt stability of a local minimum}, SIAM J. Optim. 8 (1998), 287--299.


 \bibitem{rw} {\sc R. T. Rockafellar and R. J-B Wets}, {\em Variational Analysis}, Grundlehren Series (Fundamental Principles of
Mathematical Sciences),  Springer, Berlin, 2006.

\bibitem{r85} {\sc R. T. Rockafellar}, {\em Maximal monotone relations and the second derivatives of nonsmooth functions}, Ann. Inst. H. Poincar\'e  Analyse Non Lin\'eaire  2 (1985), 167--184. 

\bibitem{r23} {\sc R. T. Rockafellar}, {\em Augmented Lagrangians and hidden convexity in sufficient conditions for local optimality},  Math. Program.  198 (2023), 159--194. 

\bibitem{rockvtn}
{\sc R. T. Rockafellar}, {\em Variational convexity and the local monotonicity of subgradient mappings}, Vietnam Journal of Mathematics, 47 (2019),  547–561.

\bibitem{Rowan}
{\sc T. H. Rowan}, {\em Functional stability analysis of numerical algorithms}, The University of Texas at Austin, 1990.




\bibitem{Sun06}
{\sc D. F. Sun}, {\em
The strong second order sufficient condition and constraint nondegeneracy in nonlinear semidefinite programming and their implications}, 
Math. Oper. Res. 31 (2006), 761--776.

 \bibitem{SSun02}
  {\sc D. F. Sun and J. Sun}, {\em Semismooth matrix valued functions}, Math. Oper. Res. 27 (2002), 150--169.
 
\bibitem{SSun03}
{\sc D. F. Sun, J. Sun},  {\em Strong Semismoothness of Eigenvalues of Symmetric Matrices and Its Application to Inverse Eigenvalue Problems}. SIAM J. Numer. Anal. 40 (2003), 2352--2367.

\bibitem{SSun08}
{\sc D. F. Sun and J. Sun}, {\em Loewner's operator and spectral functions in Euclidean Jordan algebras}, Math. Oper. Res. 33 (2008), 421--445.

\bibitem{TT}
{\sc T. Tang and K. C. Toh}, {\em A feasible method for general convex low-rank SDP problems},  SIAM J. Optim. 34(2024), 2169--2200.

\bibitem{Torki}
{\sc M. \ Torki}, {\em Second-order directional derivatives of all eigenvalues of a symmetric matrix}. Nonlinear Anal-Theor. 46 (2001), 
1133–1150. 





\bibitem{WDZZhao} {\sc S.W. \ Wang, C. \ Ding, Y.J. \ Zhang, and X.Y. \ Zhao}, {\em  Strong variational sufficiency for nonlinear semidefinite programming and its implications}, SIAM J. Optim. 33 (2023), 2988--3011.

\bibitem{YO} {\sc 
Z. Yang, E. Oja}, {\em Unified development of multiplicative algorithms for linear and quadratic nonnegative matrix factorization}, IEEE Trans. Neural Netw. 22 (2011), 1878--1891.

\bibitem{ZS} {\sc R. Zass, A. Shashua}, {\em Doubly stochastic normalization for spectral clustering}, NeurIPS, 19 (2006).

\bibitem{ZKZhu} {\sc
K. Zhang, A. Koppel, H. Zhu, T. Basar}, {\em Global convergence of policy gradient methods to (almost) locally optimal policies}. SIAM J. Optim. 58 (2020), 3586--3612.


\end{thebibliography}
\end{document}